\documentclass{amsart}
\usepackage[margin=2.8cm]{geometry}
\usepackage{amssymb}
\usepackage{graphicx} 
\usepackage{mathrsfs}
\usepackage{enumerate}
\usepackage{xspace}
\usepackage{color}
\usepackage{enumerate}
\usepackage{enumitem}
\usepackage{amsthm}
\usepackage{dsfont}
\usepackage{bbm}
\usepackage[colorlinks=true, linkcolor = blue, citecolor = blue]{hyperref}
\usepackage[numbers]{natbib}
\usepackage[backgroundcolor=white,bordercolor=red]{todonotes}
\DeclareMathAlphabet{\mathpzc}{OT1}{pzc}{m}{it}
\usepackage[nameinlink]{cleveref}
\numberwithin{equation}{section}

\begin{document}

\theoremstyle{plain}

\newtheorem{theorem}{Theorem}[section]
\newtheorem{lemma}[theorem]{Lemma}
\newtheorem{example}[theorem]{Example}
\newtheorem{proposition}[theorem]{Proposition}
\newtheorem{corollary}[theorem]{Corollary}
\newtheorem{definition}[theorem]{Definition}
\newtheorem{Ass}[theorem]{Assumption}
\newtheorem{condition}[theorem]{Condition}
\theoremstyle{definition}
\newtheorem{remark}[theorem]{Remark}
\newtheorem{SA}[theorem]{Standing Assumption}

\newcommand{\of}{[\hspace{-0.06cm}[}
\newcommand{\gs}{]\hspace{-0.06cm}]}

\newcommand\llambda{{\mathchoice
		{\lambda\mkern-4.5mu{\raisebox{.4ex}{\scriptsize$\backslash$}}}
		{\lambda\mkern-4.83mu{\raisebox{.4ex}{\scriptsize$\backslash$}}}
		{\lambda\mkern-4.5mu{\raisebox{.2ex}{\footnotesize$\scriptscriptstyle\backslash$}}}
		{\lambda\mkern-5.0mu{\raisebox{.2ex}{\tiny$\scriptscriptstyle\backslash$}}}}}

\newcommand{\1}{\mathds{1}}

\newcommand{\F}{\mathbf{F}}
\newcommand{\G}{\mathbf{G}}

\newcommand{\B}{\mathbf{B}}

\newcommand{\M}{\mathcal{M}}

\newcommand{\la}{\langle}
\newcommand{\ra}{\rangle}

\newcommand{\lle}{\langle\hspace{-0.085cm}\langle}
\newcommand{\rre}{\rangle\hspace{-0.085cm}\rangle}
\newcommand{\blle}{\Big\langle\hspace{-0.155cm}\Big\langle}
\newcommand{\brre}{\Big\rangle\hspace{-0.155cm}\Big\rangle}

\newcommand{\X}{\mathsf{X}}

\newcommand{\tr}{\operatorname{tr}}
\newcommand{\N}{{\mathbb{N}}}
\newcommand{\cadlag}{c\`adl\`ag }
\newcommand{\on}{\operatorname}
\newcommand{\oP}{\overline{P}}
\newcommand{\oO}{\mathcal{O}}
\newcommand{\D}{D(\mathbb{R}_+; \mathbb{R})}

\renewcommand{\epsilon}{\varepsilon}

\newcommand{\fPs}{\mathfrak{P}_{\textup{sem}}}
\newcommand{\fPas}{\mathfrak{P}^{\textup{ac}}_{\textup{sem}}}
\newcommand{\rrarrow}{\twoheadrightarrow}
\newcommand{\cA}{\mathcal{C}}
\newcommand{\cR}{\mathcal{R}}
\newcommand{\cK}{\mathcal{K}}
\newcommand{\cQ}{\mathcal{Q}}
\newcommand{\cF}{\mathcal{F}}
\newcommand{\cC}{\mathcal{C}}
\newcommand{\cD}{\mathcal{D}}
\newcommand{\bC}{\mathbb{C}}
\newcommand{\bth}{\overset{\leftarrow}\theta}
\renewcommand{\th}{\theta}

\newcommand{\bR}{\mathbb{R}}
\newcommand{\nnabla}{\nabla}
\newcommand{\f}{\mathfrak{f}}
\newcommand{\g}{\mathfrak{g}}
\newcommand{\oconv}{\overline{\on{co}}\hspace{0.075cm}}
\renewcommand{\a}{\mathfrak{a}}
\renewcommand{\b}{\mathfrak{b}}
\renewcommand{\d}{d}
\newcommand{\bS}{\mathbb{S}^\d_+}
\newcommand{\p}{\dot{\partial}}
\newcommand{\dr}{r} 
\newcommand{\m}{\mathbb{M}}
\newcommand{\Q}{Q}
\newcommand{\dd}{\mathsf{d}}

\renewcommand{\emptyset}{\varnothing}

\allowdisplaybreaks

\makeatletter
\@namedef{subjclassname@2020}{%
	\textup{2020} Mathematics Subject Classification}
\makeatother

 \title[Nonlinear Continuous Semimartingales]{Nonlinear Continuous Semimartingales} 
\author[D. Criens]{David Criens}
\author[L. Niemann]{Lars Niemann}
\address{Albert-Ludwigs University of Freiburg, Ernst-Zermelo-Str. 1, 79104 Freiburg, Germany}
\email{david.criens@stochastik.uni-freiburg.de}
\email{lars.niemann@stochastik.uni-freiburg.de}

\keywords{
nonlinear semimartingales; nonlinear martingale problem; sublinear expectation; nonlinear expectation; path-dependent partial differential equation; viscosity solution; semimartingale characteristics; Knightian uncertainty}

\subjclass[2020]{60G65, 60G44, 60G07, 93E20, 35D40}

\thanks{
We thank the anonymous referee for many helpful comments and suggestions. Moreover, we are grateful to Andrea Cosso for many helpful comments related to the preprint \cite{cosso2}.
DC acknowledges financial support from the DFG project SCHM 2160/15-1 and LN acknowledges financial support from the DFG project SCHM 2160/13-1.}
\date{\today}

\maketitle

\begin{abstract} In this paper we study a family of nonlinear (conditional) expectations that can be understood as a continuous semimartingale with uncertain local characteristics. Here, the differential characteristics are prescribed by a set-valued function that depends on time and path in a non-Markovian way. 
We provide a dynamic programming principle for the nonlinear expectation and we link the corresponding value function to a variational form of a nonlinear path-dependent partial differential equation. In particular, we establish conditions that allow us to identify the value function as the unique viscosity solution.
Furthermore, we prove that the nonlinear expectation solves a nonlinear martingale problem, which confirms our interpretation as a nonlinear semimartingale. 
\end{abstract}

\section{Introduction}
In this paper we study a family of nonlinear (conditional) expectations which we call \emph{nonlinear continuous semimartingale} and which we consider as a continuous semimartingale with uncertain local characteristics.
This line of research started with the seminal work of Peng \cite{peng2007g, peng2008multi} on the \(G\)-Brownian motion. In recent years, there have been several extensions to construct larger classes of nonlinear (Markov) processes, see \cite{fadina2019affine, hu2021g, neufeld2017nonlinear, nutz}. At this point we highlight the articles \cite{ElKa15, neufeld2014measurability, NVH} which establish general abstract measure theoretic concepts to construct nonlinear expectations. Next to this approach, nonlinear Markov processes have also been constructed via nonlinear semigroups, see \cite{denk2020semigroup,hol16,K19,K21} and, in particular, \cite[Chapter 4]{hol16} for a comparison of the methods. 

This paper investigates nonlinear processes with non-Markovian dynamics. We define a nonlinear expectation via 
\[
\mathcal{E}_t (\psi) (\omega) := \sup_{P \in \cA (t, \omega)} E^P \big[ \psi \big], \quad (t, \omega) \in \mathbb{R}_+ \times C(\mathbb{R}_+; \mathbb{R}^\d), 
\]
where \(\psi \colon C(\mathbb{R}_+; \mathbb{R}^\d) \to \mathbb{R}\) is an upper semianalytic function and \(\cC (t, \omega)\) is a set of probability measures on the Wiener space which give point mass to the path \(\omega\) till time \(t\) and afterwards coincide with the law of a semimartingale with absolutely continuous characteristics. 
In this paper, we parameterize drift and volatility by a compact parameter space \(F\) and two functions \(b \colon F \times \bR_+ \times C(\bR_+;\bR^\d) \to \bR^\d\) and \(a \colon F \times \bR_+ \times C(\bR_+;\bR^\d) \to \bS\) such that
\begin{align*}
\cC(t,\omega) := \Big\{ P \in \mathfrak{P}_{\text{sem}}^{\text{ac}}(t)\colon P&(X^t = \omega^t) = 1, 
\\&(\llambda \otimes P)\text{-a.e. } (dB^P_{\cdot + t} /d\llambda, dC^P_{\cdot + t}/d\llambda) \in \Theta (\cdot + t, X) \Big\},
\end{align*}
where
\[
\Theta (t, \omega) := \big\{(b (f, t, \omega), a (f, t, \omega)) \colon f \in F \big\},
\]
and
\(\fPas(t)\) denotes the set of semimartingale laws after \(t\) with absolutely continuous characteristics. This framework includes nonlinear L\'evy processes as introduced (with jumps) in \cite{neufeld2017nonlinear} and the class of nonlinear affine processes as studied in \cite{fadina2019affine}. Furthermore, our setting can also be used to model path-dependent dynamics such as stochastic delay equations under parameter uncertainty.

For this nonlinear expectation we prove the dynamic programming principle (DPP), i.e., we prove the tower property
\[
\mathcal{E}_\sigma (\psi) = \mathcal{E}_\sigma (\mathcal{E}_\tau (\psi))
\]
for all finite stopping times \(\tau \geq \sigma\). To prove the DPP we use an abstract theorem from \cite{ElKa15}. The work lies in the verification of its prerequisites. To check them we extend certain results from \cite{neufeld2014measurability} on the measurability of the semimartingale property and the behavior of the characteristics to a dynamic framework. 

In a second step we identify two properties of \(\mathcal{E}\) which confirm our interpretation as a nonlinear continuous semimartingale. First, we relate the value function 
\begin{align} \label{eq: value function intr}
v (t, \omega) := \mathcal{E}_t (\psi) (\omega) = \sup_{P \in \cA (t, \omega)} E^P \big[ \psi  \big]
\end{align}
to a path-dependent Kolmogorov type partial differential equation and, in the second part, we show that \(\mathcal{E}\) solves a type of nonlinear martingale problem. Let us discuss our contributions in more detail.

Under mere continuity and linear growth conditions on the drift \(b\) and the volatility \(a\), and under the hypothesis that the set \(\{ (b (f, t, \omega), a (f, t, \omega)) \colon f \in F\}\) is convex for every \((t,\omega) \in \bR_+ \times C(\bR_+, \bR^\d) \), we show that the value function \(v\)
is a weak sense viscosity solution, i.e., a path-dependent Crandall--Lions type viscosity solution without regularity properties, to the following path-dependent partial differential equation (PPDE):
\begin{equation} \label{eq: PIDE intro}
\begin{cases}   
\p v (t, \omega) + G (t, \omega, v) = 0, & \text{for } (t, \omega) \in [0, T) \times C(\mathbb{R}_+; \mathbb{R}^\d), \\
v (T, \omega) = \psi (\omega), & \text{for } \omega \in C(\mathbb{R}_+; \mathbb{R}^\d),
\end{cases}
\end{equation}
where
\begin{align} \label{eq: def G intro}
G(t, \omega, \phi) := \sup \Big\{ \langle \nabla \phi (t, \omega), b(f, t, \omega) \rangle
+ \tfrac{1}{2} \on{tr} \big[  \nabla^2\phi (t, \omega) a (f, t, \omega) \big]\colon f \in F \Big\}.
\end{align} 
The proof for the viscosity property is split into two parts, i.e., we prove the sub- and the supersolution property. The ideas of proof are based on applications of Berge's maximum theorem, Skorokhod's existence theorem for stochastic differential equations and Lebesgue's differentiation theorem. In contrast to the proofs from \cite{fadina2019affine, neufeld2017nonlinear} for the viscosity subsolution property in L\'evy and continuous affine frameworks respectively, we do not work with explicit moment estimates. This allows us to extend the class of test functions in our framework to \(C^{1, 2}\) in comparison to the class \(C^{2, 3}\) as used in \cite{fadina2019affine, neufeld2017nonlinear}.

Further, we investigate when the value function is not only a {\em weak sense} viscosity solution but has in addition some regularity properties.
By virtue of the last term in \eqref{eq: value function intr} and Berge's maximum theorem, it is a natural idea to deduce regularity properties of \(v\) from corresponding properties of the set-valued map \( (t,\omega) \mapsto \cC(t,\omega)\). To the best of our knowledge, the idea to deduce regularity properties of value functions from related properties of set-valued maps traces back to the seminal paper \cite{nicole1987compactification} that investigates a controlled diffusion framework.
Due to the appearance of \(\fPas (t)\) and \((dB^P_{\cdot + t}/d \llambda, dC^P_{\cdot + t} /d \llambda)\) in the definition of the set \( \cC(t,\omega) \), regularity properties of \((t, \omega) \mapsto \cC(t, \omega)\) seem at first glance to be difficult to verify. To get a more convenient condition, we show that 
\[
v(t,\omega) = \sup_{P \in \cA(t, \omega)} E^{P}\big[\psi\big]  = \sup_{P \in \mathcal{R}(t, \omega)} E^{P}\big[\psi(\omega \ \widetilde{\otimes}_t \ X) \big],
\]
where
\begin{align*}
\mathcal{R}(t,\omega) := \Big\{ P \in \fPas \colon P &\circ X_0^{-1} = \delta_{\omega (t)}, \\
&(\llambda \otimes P)\text{-a.e. } (dB^{P} /d\llambda, dC^{P}/d\llambda) \in \Theta (\cdot + t, \omega \ \widetilde{\otimes}_t\ X)  \Big\},
\end{align*}
and
\[
\omega\ \widetilde{\otimes}_t\ \omega' :=  \omega \1_{[ 0, t)} + (\omega (t) + \omega' (\cdot - t) - \omega' (0)) \1_{[t, \infty)}.
\]
This reformulation of the value function \(v\) explains that it suffices to investigate the correspondence \((t,\omega) \mapsto \mathcal{R}(t,\omega)\) and it connects the two (seemingly closely related) approaches from \cite{ElKa15} and \cite{NVH} for the construction of nonlinear expectations. 
We show that \((t,\omega) \mapsto \cR(t,\omega)\) is upper hemicontinuous and compact-valued, which establishes upper semicontinuity of \(v\). This requires a profound analysis of the limiting behaviour of semimartingale characteristics and hinges on the continuity of the correspondence \( (t, \omega) \mapsto \Theta(t,\omega)\).
We also present a counterexample which explains that mere continuity and linear growth of \(b\) and \(a\) are insufficient for lower hemicontinuity of \((t, \omega) \mapsto \cR(t, \omega)\).

Provided \(b\) and \(a\) are additionally locally Lipschitz continuous, we prove lower hemicontinuity of 
\( (t, \omega) \mapsto \cR(t,\omega)\). To this end, we combine arguments based on an implicit function theorem, a strong existence property for stochastic differential equations with random locally Lipschitz coefficients and Gronwall's lemma.
To the best of our knowledge, this is the first result regarding lower hemicontinuity in a path-dependent setting related to nonlinear stochastic processes. 

Under a uniform Lipschitz continuity condition on the coefficients \(b\) and \(a\) and the terminal function \(\psi\), we also show that the value function \(v\) has a certain Lipschitz property w.r.t. the function 
\begin{align*}
		\dd ( (t,\omega), (s, \alpha)) :=  \Big( 1 + \sup_{r \in [0, t]} \|\omega (r)\| +  \sup_{r \in [0, s]} &  \|\alpha (r)\|\Big) | t - s |^{1/2} \\&+ \sup_{r \in [0, T]} \|\omega(r \wedge t) - \alpha( r \wedge s) \|,
\end{align*}
for \((t, \omega), (s, \alpha) \in [0, T] \times C([0, T]; \bR^\d).\)
This observation allows us to invoke a novel uniqueness result from \cite{zhou} that identifies \(v\) as the unique viscosity solution to the PPDE \eqref{eq: PIDE intro} that is bounded and Lipschitz continuous w.r.t. \(\dd\).

Next, we discuss the martingale problem related to \(\mathcal{E}\). In case the coefficients \(b\) and \(a\) are compactly parameterized independently of each other, and under linear growth and mere continuity assumptions, for suitable test functions \(\phi\), we show that the process
\[
\phi (X_t) - \int_0^t G (s, X, \phi) ds, \quad t \in \mathbb{R}_+, 
\]
is a local \(\mathcal{E}\)-martingale. 
This seems to be a novel connection between nonlinear martingale problems and nonlinear processes in a path-dependent setting. Our proof is based on applications of a measurable maximum theorem and Skorokhod's existence theorem.

This paper is structured as follows: in Section \ref{sec: setting} we introduce our setting. In Section~\ref{sec: DPP} we state the DPP and in Section \ref{sec: vis} we discuss the relation of the correspondence \(\cR\) to the PPDE \eqref{eq: PIDE intro}. In Section \ref{sec: MP} the nonlinear martingale problem is stated. The proofs for our main results are given in the remaining sections. More precisely, the DPP is proved in Section~\ref{sec: pf DPP}, the viscosity property is proved in Section~\ref{sec: pf vis}, regularity of \(\cR\) and \(v\) are established in Section \ref{sec: pf regularity r} and the martingale problem is proved in Section~\ref{sec: pf MP}.

\section{The Setting}\label{sec: setting}
Let \(\d \in \mathbb{N}\) be a fixed dimension and define $\Omega$ to be the space of continuous functions \(\mathbb{R}_+ \to \mathbb{R}^\d\) endowed with the local uniform topology. The Euclidean scalar product and the corresponding Euclidean norm are denoted by \(\langle \cdot, \cdot\rangle\) and \(\|\cdot\|\).
We write \(X\) for the canonical process on $\Omega$, i.e., \(X_t (\omega) = \omega (t)\) for \(\omega \in \Omega\) and \(t \in \mathbb{R}_+\). 
It is well-known that \(\mathcal{F} := \mathcal{B}(\Omega) = \sigma (X_t, t \geq 0)\).
We define $\F := (\mathcal{F}_t)_{t \geq 0}$ as the canonical filtration generated by $X$, i.e., \(\mathcal{F}_t := \sigma (X_s, s \leq t)\) for \(t \in \mathbb{R}_+\). Notice that we do not make the filtration \(\F\) right-continuous. The set of probability measures on \((\Omega, \mathcal{F})\) is denoted by \(\mathfrak{P}(\Omega)\) and endowed with the usual topology of convergence in distribution.
Let \(F\) be a metrizable space and 
let \(b \colon F \times \mathbb{R}_+ \times \Omega \to \mathbb{R}^\d\) and \(a \colon F \times \mathbb{R}_+ \times \Omega \to \bS\) be Borel functions such that \((t, \omega) \mapsto b(f, t, \omega)\) and \((t, \omega) \mapsto a (f, t, \omega)\) are predictable for every \(f \in F\). Here, \(\bS\) denotes the space of all real-valued symmetric positive semidefinite \(\d \times \d\) matrices. We define the correspondence, i.e., the set-valued mapping, \(\Theta \colon \mathbb{R}_+ \times \Omega \twoheadrightarrow \mathbb{R}^\d \times \bS\) by
\[
\Theta (t, \omega) := \big\{(b (f, t, \omega), a (f, t, \omega)) \colon f \in F \big\}.
\]
\begin{SA} \label{SA: meas gr}
	\(\Theta\) has a measurable graph, i.e., the graph
	\[
	\on{gr} \Theta = \big\{ (t, \omega, b, a) \in \mathbb{R}_+ \times \Omega \times \mathbb{R}^\d \times \bS \colon (b, a) \in \Theta (t, \omega) \big\}
	\]
	is Borel. 
\end{SA}
In Lemma \ref{lem: mbl graph} below we will see that this standing assumption holds once \(F\) is compact and \(b\) and \(a\) are continuous in the \(F\) variable.

We call an \(\bR^\d\)-valued continuous process \(Y = (Y_t)_{t \geq 0}\) a (continuous) \emph{semimartingale after a time \(t^* \in \mathbb{R}_+\)} if the process \(Y_{\cdot + t^*} = (Y_{t + t^*})_{t \geq 0}\) is a \(\d\)-dimensional semimartingale for its natural right-continuous filtration.
Notice that it comes without loss of generality that we consider the right-continuous version of the filtration (see \cite[Proposition~2.2]{neufeld2014measurability}). 
The law of a semimartingale after \(t^*\) is said to be a \emph{semimartingale law after \(t^*\)} and the set of them is denoted by \(\fPs (t^*)\).
Notice also that \(P \in \fPs(t^*)\) if and only if the coordinate process is a semimartingale after \(t^*\), see Lemma \ref{lem: jacod restatements} below.
For \(P \in \fPs (t^*)\) we denote the semimartingale characteristics of the shifted coordinate process \(X_{\cdot + t^*}\) by \((B^P_{\cdot + t^*}, C^P_{\cdot + t^*})\). 
Moreover, we set 
\[
\fPas (t^*) := \big\{ P \in \fPs (t^*) \colon P\text{-a.s. } (B^P_{\cdot + t^*}, C^P_{\cdot + t^*}) \ll \llambda \big\}, \quad \fPas := \fPas (0),
\]
where \(\llambda\) denotes the Lebesgue measure. 
For \(\omega, \omega' \in \Omega\) and \(t \in \mathbb{R}_+\), we define the concatenation
\[
\omega \otimes_t \omega' :=  \omega \1_{[ 0, t)} + (\omega (t) + \omega' - \omega' (t))\1_{[t, \infty)}.
\]
Finally, for $(t,\omega) \in \bR_+ \times \Omega$, we define $\cA(t,\omega) \subset \mathfrak{P}(\Omega)$ by
\begin{align*}
\cC(t,\omega) := \Big\{ P \in \mathfrak{P}_{\text{sem}}^{\text{ac}}(t)\colon &P(X^t = \omega^t) = 1, \\&\qquad
(\llambda \otimes P)\text{-a.e. } (dB^P_{\cdot + t} /d\llambda, dC^P_{\cdot + t}/d\llambda) \in \Theta (\cdot + t, \omega \otimes_t X) \Big\},
\end{align*}
where we use the standard notation \(X^t := X_{\cdot \wedge t}\).

To lighten our notation, let us further define, for two stopping times \(S\) and \(T\), the stochastic interval
	\[
	\of S, T \of \hspace{0.1cm} := \{ (t,\omega) \in \bR_+ \times \Omega \colon S(\omega) \leq t < T(\omega) \}.
	\]
	The stochastic intervals \( \gs S, T \of, \of S, T \gs, \gs S, T \gs  \) are defined accordingly.
	In particular, the equality \(\of 0, \infty\of \hspace{0.1cm} = \bR_+ \times \Omega\) holds.

\begin{SA} \label{SA: non empty}
	\(\cA (t, \omega) \not = \emptyset\) for all \((t, \omega) \in \of 0, \infty\of\).
\end{SA}
In Lemma \ref{lem: SA2 holds} below we will see that this standing assumption holds under continuity and linear growth conditions on \(b\) and \(a\).

\begin{remark}
	\quad
	\begin{enumerate}
		\item[\textup{(i)}] 		The sets $(\cA(t,\omega))_{(t, \omega) \in \of 0, \infty\of}$ are adapted in the following sense:
		for each $t \in \mathbb{R}_+$, the set $\cA(t, \omega)$ depends only on the path of $\omega$ up to time $t$.
		\item[\textup{(ii)}]  Notice that
			\begin{align*}
			\cA(t,\omega) =\Big\{ P \in \mathfrak{P}_{\text{sem}}^{\text{ac}}(t)\colon &P(X^t = \omega^t) = 1, \\&\quad
			(\llambda \otimes P)\text{-a.e. } (dB^P_{\cdot + t} /d\llambda, dC^P_{\cdot + t}/d\llambda) \in \Theta (\cdot + t, X) \Big\}.
			\end{align*}
			We defined \(\cA(t, \omega)\) with the seemingly more complicated ingredient \(\Theta (\cdot + t, \omega \otimes_t X)\), instead of \(\Theta (\cdot + t, X)\), to prevent confusions about measurability. Namely, both \((dB^P_{\cdot + t} /d\llambda, dC^P_{\cdot + t}/d\llambda)\) and \(\Theta (\cdot + t, \omega \otimes_t X)\) are measurable with respect to \(\sigma (X_s, s \geq t)\), while \(\Theta (\cdot + t, X)\) might also depend on \(X_s, s < t\). 
	\end{enumerate}
\end{remark}

In the following we state and discuss some conditions needed to formulate our main results.
\begin{condition}[Linear Growth] \label{cond: LG}
	For every \(T > 0\), there exists a constant \(C = C_T > 0\) such that 
	\[
	\|b (f, t, \omega)\|^2 + \on{tr} \big[ a (f, t, \omega)\big] \leq C \Big( 1 + \sup_{s \in [0, t]} \| \omega (s) \|^2 \Big)
	\]
	for all \(f \in F\) and \((t, \omega) \in \of 0, T\gs\).
\end{condition}

\begin{condition}[Continuity in Control] \label{cond: continuity control}
	For each \((t, \omega) \in \of 0, \infty\of\), the maps \(f \mapsto b(f, t, \omega)\) and \(f \mapsto a (f, t, \omega)\) are continuous.
\end{condition}

\begin{condition}[Joint Continuity] \label{cond: joint continuity in all}
	The functions \((f, t, \omega) \mapsto b (f, t, \omega)\) and \((f, t, \omega) \mapsto a (f, t, \omega)\) are continuous. 
\end{condition}

\begin{condition}[Convexity] \label{cond: convexity}
	For every \((t, \omega) \in \of 0, \infty\of\), the set \(\{ (b (f, t, \omega), a (f, t, \omega)) \colon f \in F\} \subset \bR^\d \times \bS\) is convex.
\end{condition}

Before we present our main results, let us shortly show that our standing assumptions hold under some of the above conditions. We start with Standing Assumption \ref{SA: meas gr}.

\begin{lemma} \label{lem: mbl graph}
	If \(F\) is a compact metrizable space and Condition \ref{cond: continuity control} holds, then the correspondence \(\Theta\)
	has measurable graph, i.e., Standing Assumption \ref{SA: meas gr} holds.
\end{lemma}
The previous lemma is a direct consequence of the following general observation.
\begin{lemma} \label{lem: correspondence meas graph general result}
	Let $(\Sigma, \mathcal{G})$ be a measurable space, let $F$ be a compact metrizable space, let $E$ be a separable metrizable space, and finally let $g \colon F \times \Sigma \to E$ be a Carath\'eodory function, i.e., \(g\) is continuous in the first and measurable in the second variable.
	Then, the correspondence $\varphi$ defined by
	$\varphi(\sigma) := \{ g(f, \sigma) \colon f \in F \}$ 
	has a measurable graph.
\end{lemma}
\begin{proof}
	First, we show that $\varphi$ is weakly measurable, that is, for every open set $U \subset E$ the lower inverse $\varphi^l(U) := \{ \sigma \in \Sigma \colon \varphi(\sigma) \cap U \neq \emptyset \}$
	is measurable. As $F$ is compact, there exists a countable dense subset $\{f_n \colon n \in \mathbb{N} \} \subset F$.
	Since $g$ is a Carath\'eodory function, the functions $\{g_n \colon n \in \mathbb{N} \}$, where $g_n(\sigma) := g(f_n, \sigma)$, form a Castaing representation of $g(F, \cdot) \subset E$, i.e., a countable family of measurable functions such that
	$$ \on{cl} ( \{ g_n( \sigma)\colon n \in \mathbb{N} \} ) = g(F, \sigma), $$
	for every $\sigma \in \Sigma$.
	Thus, for each open subset $U \subset E$,
	\begin{align*}
	\{ \sigma \in \Sigma \colon \varphi(\sigma) \cap U \neq \emptyset \} & = \{ \sigma \in \Sigma \colon g(F, \sigma) \cap U \neq \emptyset \} 
	= \bigcup_{n \in \mathbb{N}} \{ g_n \in U \} \in \mathcal{G}. 
	\end{align*}
	Moreover, as $\varphi$ is compact-valued, and in particular closed-valued, it follows from \cite[Theorem~18.6]{charalambos2013infinite} that $\varphi$ has measurable graph.
\end{proof}
Moreover, also our second standing assumption follows from some conditions above.
\begin{lemma} \label{lem: SA2 holds}
	If the Conditions \ref{cond: LG} and \ref{cond: joint continuity in all} hold, then \(\cA(t, \omega) \not = \emptyset\) for all \((t, \omega) \in \of 0, \infty\of\), i.e., Standing Assumption \ref{SA: non empty} holds.
\end{lemma}
\begin{proof}
	This follows from Lemma \ref{lem: worst case} below.
\end{proof}

\section{The Dynamic Programming Principle} \label{sec: DPP}
We now define a nonlinear expectation and prove the dynamic programming principle (DPP), i.e., the tower property.
Suppose that \(\psi \colon \Omega \to [- \infty, \infty]\) is an upper semianalytic function, i.e., \(\{\omega \in \Omega \colon \psi (\omega) > c\}\) is analytic for every \(c \in \mathbb{R}\),
and define the so-called \emph{value function} by
\[
v (t, \omega) := \sup_{P \in \cA (t, \omega)} E^P \big[ \psi \big], \quad (t, \omega) \in \of 0, \infty\of.
\]
For every finite stopping time \( \tau \), we set \( v( \tau, X ) := v(\tau(X), X) \).
\begin{theorem}[Dynamic Programming Principle] \label{theo: DPP}
	The value function \(v\) is upper semianalytic. Moreover, for every pair \((t, \omega) \in \of 0, \infty\of \) and every stopping time \(\tau\) with \(t \leq \tau < \infty\), we have 
	\begin{align} \label{eq: DPP}
	v(t, \omega) = \sup_{P \in \cA (t, \omega)} E^P \big[ v (\tau, X) \big]. 
	\end{align}
\end{theorem}
The value function can be interpreted as a nonlinear expectation \(\mathcal{E}\) given by 
\[
\mathcal{E}_t (\psi) (\omega) := v (t, \omega), \qquad (t, \omega) \in \of 0, \infty\of.
\]
The DPP in Theorem \ref{theo: DPP} provides the tower rule for \(\mathcal{E}\). Namely, \eqref{eq: DPP} means that 
\begin{align} \label{eq: tower property}
\mathcal{E}_t (\psi) (\omega) = \mathcal{E}_t (\mathcal{E}_\tau (\psi)) (\omega), \quad t \leq \tau \leq \infty.
\end{align}
By its pathwise structure, the equality  \eqref{eq: tower property} also implies that
	\[
	\mathcal{E}_{\sigma} (\psi) = \mathcal{E}_{\sigma} (\mathcal{E}_\tau (\psi)),
	\]
	for all finite stopping times \(\tau \geq \sigma\).
To prove Theorem \ref{theo: DPP} we use a general theorem from \cite{ElKa15}. The work lies in the verification of the prerequisites, which are (i) a measurable graph property, (ii) a stability property for conditioning, and (iii) a stability property for pasting. The proof is given in Section \ref{sec: pf DPP} below.

\begin{example} \label{ex: levy, mark, delay}
		In the following we mention some examples of stochastic models that are covered by our framework. We stress that it includes many previously studied frameworks but also some new ones which are of interest for future investigations.
		\begin{enumerate}
			\item[\textup{(i)}]
			The case where \(\Theta (t, \omega) \equiv \Theta\) is independent of time \(t\) and path \(\omega\) corresponds to the generalized \(G\)-Brownian motion as introduced in \cite{peng2010}, cf. also \cite{neufeld2017nonlinear} for a nonlinear L\'evy setting {\em with jumps}.
			\item[\textup{(ii)}]
			The situation where \(\Theta (t, \omega) \equiv \Theta (\omega (t))\) depends on \((t, \omega)\) only through the value \(\omega (t)\) corresponds to a Markovian setting that has, for instance, been studied in \cite{CN22b, hol16}.
			\item[\textup{(iii)}]
            The example where, for \((f, t, \omega) \in F \times \of 0, \infty\of\) and \( r > 0\), 
            \begin{align*}
			b (f, t, \omega) &:= b_0 (f) \omega (t) + \int_{ (t - r) \vee 0}^t b_1 (f, s) \omega (s)  ds, \qquad
			a (f, t, \omega) := a_0 (f),
			\end{align*}
   	       with Borel functions 
            \[ b_0 \colon F \to \bR, \quad b_1 \colon F \times \bR_+ \to \bR, \quad a_0 \colon F \to \bR_+.
            \]
            This corresponds to a class of linear stochastic delay differential equations with parameter uncertainty (that is captured by the \(F\)-dependence of \(b\) and \(a\)).
			Control problems for such a path-dependent setting were studied in \cite{gozzi}.
            From a modeling perspective, natural assumptions such as sign-constraints and ellipticity can be incorporated through the functions \(b_0, b_1\) and \(a_0\), respectively.
            More references involving control problems for delay equations can be found in \cite{cosso}. 
			Our framework seems to be the first that captures stochastic delay equations from the perspective of nonlinear stochastic processes.
		\end{enumerate}
	\end{example}

Given the DPP, we proceed studying more properties of \(v\) and \(\mathcal{E}\). In the following section we identify \(v\) as a viscosity solution to a certain nonlinear path-dependent partial differential equation (PPDE) and, in Section \ref{sec: MP}, we identify \(\mathcal{E}\) as a solution to a nonlinear martingale problem.


\section{The nonlinear Kolmogorov Equation} \label{sec: vis}
In the following we discuss the relation of the value function to a path-dependent Kolmogorov type partial differential equation. This section is structured as follows. In Section~\ref{sec: test functions} we define the class of test functions for the concept of Crandall--Lions type viscosity solutions in our path-dependent setting. The Kolmogorov type equation is introduced in Section \ref{sec: nonlinear KE}. Finally, we present our main results in Section \ref{sec: main results viscosity}.
\subsection{The class of test functions} \label{sec: test functions}
In the first part of this section we introduce the set of test functions for the concept of Crandall--Lions type viscosity solutions in our path-dependent setting.  
The following definitions are adapted from \cite{cosso, cosso2}.
Let \(T >0\) and \(t_0 \in [0,T)\).
Further, let \(D(\bR_+; \bR^\d)\) be the space of \cadlag functions from \(\bR_+\) into \(\mathbb{R}^\d\).
We define \(\Lambda(t_0) := [t_0, T] \times D(\bR_+; \bR^\d)\) and, on \([0, T] \times D(\bR_+; \bR^\d)\), we further define the pseudometric \(d\) as 
\begin{align} \label{eq: d metric}
d ( (t,\omega), (s, \omega')) := | t - s | + \sup_{r \in [0, T]} \|\omega(r \wedge t) - \omega'( r \wedge s) \|.
\end{align}
We denote the restriction of \(d\) to \(\Lambda(t_0)\) again by \(d\).
For a map \(F \colon \Lambda(t_0) \to \bR\) we say that \(F\) admits a \emph{horizontal derivative} at \((t, \omega) \in \Lambda(t_0)\) with \(t < T\) if
\[
\p F (t, \omega) := \lim_{h \searrow 0} \frac{ F(t + h, \omega(\cdot \wedge t)) - F(t, \omega(\cdot \wedge t))}{h}
\]
exists.
At \(t = T\), the horizontal derivative is defined as 
\[
\p F(T, \omega) := \lim_{h \nearrow T} \p F(h,\omega).
\]
Further, we say that 
\(F\) admits a \emph{vertical derivative} at \((t, \omega) \in \Lambda(t_0)\) if 
\[
\partial_i F (t, \omega) := \lim_{h \to 0} \frac{F (t, \omega + h e_i\1_{[t,T]}) - F (t, \omega)}{h}, \quad i = 1, 2, \dots, \d,
\]
exist, where \(e_1, \dots, e_\d\) are the standard unit vectors in \(\bR^\d\).
Accordingly, the second vertical derivatives \(\partial^2_{ij} F(t,\omega), i, j = 1, \dots, \d,\) at \((t, \omega) \in \Lambda(t_0)\) are defined as
\[
\partial^2_{ij} F(t,\omega) := \partial_i (\partial_j F)(t, \omega).
\]
We write \(\nabla F := (\partial_1 F, \dots, \partial_\d F)\) for the {\em vertical gradient} and \(\nabla^2 F := (\partial_{ij}^2 F)_{i, j = 1, \dots, \d}\) for the {\em vertical Hessian matrix}.

Next, we denote by \(C^{1,2}(\Lambda(t_0); \bR)\) the set of functions \(F \colon \Lambda(t_0) \to \bR\), continuous with respect to \(d\), 
such that
\[
\p F, \nabla F, \nabla^2 F
\]
exist everywhere on \(\Lambda(t_0)\) and are continuous with respect to \(d\).

The set 
\(C^{1,2}( \of t_0, T \gs ; \bR)\) consists of functions 
\(F \colon \of t_0, T \gs  \to \bR \)  
such that there exists \(\hat{F} \in C^{1,2}(\Lambda(t_0); \bR)\) with
\[
F(t, \omega) = \hat{F}(t,\omega), \quad (t, \omega) \in \of t_0, T \gs .
\]
In this case, we define, for \((t, \omega) \in \of t_0, T \gs \),
\[
\p F(t, \omega) := \p \hat{F}(t, \omega), \quad
\nabla F(t,\omega) := \nabla \hat{F}(t,\omega), \quad
\nabla^2 F(t,\omega) := \nabla^2 \hat{F}(t,\omega).
\]
By \cite[Lemma 2.1]{cosso}, the derivatives \(\p F, \nabla F, \nabla^2 F \) are well-defined for \(F \in C^{1,2}(\of t_0, T \gs ; \bR)\).

Finally, the set \(C_{pol}^{1,2}( \of t_0, T \gs ; \bR) \) consists of all \(F \in C^{1,2}( \of t_0, T \gs ; \bR)\) such that there exist constants \(C, q \geq 0\) with
\[
| \p F(t, \omega) | + \| \nabla F(t,\omega) \| + \on{tr} \big[ \nabla^2 F(t,\omega) \big] \leq C \Big( 1 + \sup_{r \in [t_0,T]} \| \omega( r \wedge t ) \|^q\Big)
\]
for all \((t, \omega) \in \of t_0, T \gs \).

\begin{remark}
	Every map \(F \colon [0,T] \times D(\bR_+; \bR^\d) \) that is continuous with respect to \(d\) is \emph{non-anticipative}, i.e., 
	\(F (t, \omega) = F(t, \omega (\cdot \wedge t))\) for all \((t, \omega) \in [0,T] \times D(\bR_+;\bR^\d)\).
\end{remark}

\subsection{The nonlinear Kolmogorov equation} \label{sec: nonlinear KE}
For \((t, \omega,\phi) \in \of 0, T \gs \hspace{0.05cm} \times\hspace{0.05cm} C^{1,2}(\of 0, T \gs ; \bR ) \), we set 
\begin{align*}
G(t, \omega, \phi) := \sup \Big\{ \langle \nabla \phi (t, \omega), b (f, t, \omega) \rangle
+ \tfrac{1}{2} \on{tr} \big[ \nabla^2\phi (t, \omega) a (f, t, \omega) \big]\colon f \in F \Big\}.
\end{align*}
One of our goals is the identification of the value function \(v\) as a so-called \emph{viscosity solution} to the nonlinear~PPDE 
\begin{equation} \label{eq: PIDE}
\begin{cases}   
\p v (t, \omega) + G (t, \omega, v) = 0, & \text{for } (t, \omega) \in \of 0, T\of, \\
v (T, \omega) = \psi (\omega), & \text{for } \omega \in \Omega,
\end{cases}
\end{equation}
where \(\psi \colon \Omega \to \mathbb{R}\) is a bounded continuous function such that \(\psi (\omega) = \psi (\omega (\cdot \wedge T))\).
For notational convenience, we fix the terminal function \(\psi\) from now on.

In contrast to the classical case, where the solution runs over time and space, we consider viscosity solutions which run over time and path. Let us provide a precise definition of a viscosity solution in our setting.

\begin{definition}[Viscosity Solution over Time and Path]
	A function \(u \colon \of 0, T\gs \to \mathbb{R}\) is said to be a \emph{weak sense viscosity subsolution} to \eqref{eq: PIDE} if the following two properties hold:
	\begin{enumerate}
		\item[\textup{(a)}] \(u(T, \cdot) \leq \psi\);
		\item[\textup{(b)}]
		for any \((t,\omega) \in \of 0, T \of \) and \( \phi \in C^{1,2}_{pol}(\of t, T \gs ; \bR )\) satisfying
		\[
		0 = (u-\phi)(t,\omega) = \sup \{ (u-\phi)(s,\omega') \colon (s,\omega') \in \of t, T \gs  \},
		\]
		we have
		\(
		\p \phi (t, \omega) + G (t, \omega, \phi) \geq 0.
		\)
	\end{enumerate}
	Moreover, a function \(u \colon \of 0, T\gs \to \mathbb{R}\) is said to be a \emph{weak sense viscosity supersolution} to~\eqref{eq: PIDE} if the following two properties hold:
	\begin{enumerate}
		\item[\textup{(a)}] \(u(T, \cdot) \geq \psi\);
		\item[\textup{(b)}]
		for any \((t,\omega) \in \of 0, T \of\) and \( \phi \in C^{1,2}_{pol}(\of t, T \gs ; \bR)\) satisfying
		\[
		0 = (u-\phi)(t,\omega) = \inf \{ (u-\phi)(s,\omega') \colon (s,\omega') \in \of t, T \gs \},
		\]
		we have
		\(
		\p \phi (t, \omega) + G (t, \omega, \phi) \leq 0.
		\)
	\end{enumerate}
	Further, \(u\) is called \emph{weak sense viscosity solution} if it is a weak sense viscosity sub- and supersolution. Finally, a continuous weak sense viscosity solution is called \emph{viscosity solution}.
\end{definition}

\subsection{Main Results} \label{sec: main results viscosity}
Before we present our main results, we need a last bit of notation. 
For \(\omega, \omega' \in \Omega\) and \(t \in \mathbb{R}_+\), we define the concatenation
\[
\omega\ \widetilde{\otimes}_t\ \omega' :=  \omega \1_{[ 0, t)} + (\omega (t) + \omega' (\cdot - t) - \omega' (0)) \1_{[t, \infty)},
\]
and the set
\begin{equation}\label{eq: def R}\begin{split}
\mathcal{R}(t,\omega) := \Big\{ P \in \fPas \colon P &\circ X_0^{-1} = \delta_{\omega (t)}, \\
&(\llambda \otimes P)\text{-a.e. } (dB^{P} /d\llambda, dC^{P}/d\llambda) \in \Theta (\cdot + t, \omega \ \widetilde{\otimes}_t\ X)  \Big\}.
\end{split}
\end{equation}
We are in the position to present the main results of this section. Let us start with the viscosity solution part. 

\begin{theorem} \label{theo: viscosity}
	Assume that \(F\) is a compact metrizable space and that
	the Conditions \ref{cond: LG}, \ref{cond: joint continuity in all} and \ref{cond: convexity} hold. Then, the value function \(v\) is a weak sense viscosity solution to \eqref{eq: PIDE}. If the correspondence \((t, \omega) \mapsto \cR (t, \omega)\) is continuous (lower, upper hemicontinuous),\footnote{see \textup{\cite[Definition~17.2]{charalambos2013infinite}}.} then \(v\) is continuous (lower, upper semicontinuous).
\end{theorem}

For a first result regarding upper semicontinuity of \( v\) in a Markovian framework beyond the L{\'e}vy case (albeit under uniform boundedness and global Lipschitz assumptions) we refer to \cite[Lemma 4.42]{hol16}. The thesis \cite{hol16} contains no explicit conditions for lower semicontinuity that appears to be related to martingale problems with possibly non-regular coefficients, which are difficult to study, see also  \cite[Remark~4.43]{hol16}, \cite[Remark 3.4]{K19} and  \cite[Remark~5.4]{K21} for comments in this direction. 
	
	As shown in Theorem \ref{theo: viscosity}, (lower, upper) hemicontinuity of \( (t, \omega) \mapsto \cR(t, \omega) \) provides (lower, upper) semicontinuity of the value function \( (t, \omega) \mapsto v(t,\omega) \). 
	In the framework of nonlinear L{\'e}vy processes from \cite{neufeld2017nonlinear}, reduced to our path-continuous setting, continuity of \(\cR\) is rather straightforward to verify. That is, in case \( \Theta \subset \mathbb{R}^\d \times \bS \) is a convex and compact set, the correspondence
	\[
	\cR(t, \omega) = \big\{ P \in \fPas \colon  P \circ X_0^{-1} = \delta_{\omega (t)}, \ (\llambda \otimes P)\text{-a.e. } (dB^{P} /d\llambda, dC^{P}/d\llambda) \in \Theta \big\},
	\]
	is continuous. We present the details in Appendix \ref{subsec: levy}.

	In the following, we prove more general conditions for upper and lower hemicontinuity of the correspondence \((t, \omega) \mapsto \cR (t, \omega)\) which lead to explicit conditions for the continuity of the value function and thereby identify it as a viscosity solution to the nonlinear PPDE~\eqref{eq: PIDE}.

	\begin{theorem}[Upper hemicontinuity of \(\cR\)] \label{theo: upper hemi}
		Assume that \(F\) is a compact metrizable space and that the Conditions \ref{cond: LG}, \ref{cond: joint continuity in all} and \ref{cond: convexity} hold. Then, the correspondence \((t, \omega) \mapsto \cR (t, \omega)\) is upper hemicontinuous with compact values.
	\end{theorem}
	
	The following example shows that the conditions from Theorem \ref{theo: upper hemi} are not sufficient for lower hemicontinuity of \((t, \omega) \mapsto \cR(t, \omega)\).
	
	\begin{example}
		Suppose that \(\d = 1, a \equiv 0\) and that \(b (f, t, \omega) \equiv b^\circ (\omega (t))\) for all \((f, t, \omega) \in F \times \of 0, \infty \of\), where \(b^\circ \colon \bR \to \bR\) is any bounded continuous function such that \(b^\circ (x) = \on{sgn} (x) \sqrt{|x|}\) for \(|x| \leq 1\) and which is continuously differentiable off \((-1, 1)\). According to \cite[Exercise~12.4.2]{SV}, whenever \(\omega (t) \not = 0\), the set \(\cR (t, \omega)\) is a singleton \(\{P_{\omega (t)}\}\) and \(\bR \backslash \{0\} \ni x \mapsto P_x\) has a weak limit through (strictly) positive or (strictly) negative values (that converge to zero) but these limits are different. To see that \(\cR\) is not lower hemicontinuous, take \((t^n, \omega^n) := (1/n, \on{id}) \in \of 0, \infty\of\), with \(n \in \mathbb{N}\), and \(P := \lim_{x \nearrow 0} P_x \in \cR (0, \on{id})\). Then, as \(\lim_{n \to \infty} P_{\omega^n (1/n)} = \lim_{n \to \infty} P_{1/n} \not = P\), we get from \cite[Theorem 17.21]{charalambos2013infinite} that \((t, \omega) \mapsto \cR(t, \omega)\) is not lower hemicontinuous. 
	\end{example}
	Next, we show that the correspondence \((t, \omega) \mapsto \cR (t, \omega)\) is lower hemicontinuous in case we impose an additional local Lipschitz condition. For a matrix \(A\), we denote its transposed by \(A^*\) and we denote by \(\|A\|_o\) the operator norm of \(A\). Let \(\mathscr{P}\) be the predictable \(\sigma\)-field on \(\of 0, \infty\of\).
	
	\begin{condition}[Local Lipschitz continuity in the last variable uniformly in the first two] \label{cond: lipschitz}
		There exists a dimension \(\dr \in \mathbb{N}\) and a \(\mathcal{B}(F) \otimes \mathscr{P}\)-measurable function \(\sigma \colon F \times \of 0, \infty\of \hspace{0.05cm} \to \bR^{\d \times \dr}\) such that \(a = \sigma \sigma^*\).
		Furthermore, for every \(T, M > 0\), there exists a constant \(C = C_{T, M} > 0\) such that 
		\[
		\|b (f, t, \omega) - b (f, t, \alpha)\| + \|\sigma (f, t, \omega) - \sigma (f, t, \alpha)\|_o \leq C \sup_{s \in [0, t]} \| \omega (s) - \alpha (s) \|
		\]
		for all \(\omega, \alpha \in \Omega \colon \sup_{s \in [0, t]} \| \omega (s) \| \vee \| \alpha (s) \| \leq M\) and \((f, t) \in F \times [0, T]\).
	\end{condition}
	The following theorem is seemingly the first result regarding lower hemicontinuity and, together with Theorem \ref{theo: viscosity}, the first result on lower semicontinuity of the value function in a path-dependent framework related to nonlinear stochastic processes.
	\begin{theorem}[Lower hemicontinuity of \(\cR\)] \label{theo: lower hemi}
		Assume that \(F\) is a compact metrizable space and that the Conditions \ref{cond: LG}, \ref{cond: joint continuity in all}, \ref{cond: convexity} and \ref{cond: lipschitz} hold. 
		Then, the correspondence \((t, \omega) \mapsto \cR (t, \omega)\) is lower hemicontinuous.
	\end{theorem}
	
	\begin{example} \label{ex: interval}
		An interesting situation is the case where \(d = 1\) and \[\Theta (t, \omega) = [\underline{b}_t (\omega), \overline{b}_t (\omega)] \times [\underline{a}_t (\omega), \overline{a}_t (\omega)]\]
		with predictable functions
		\[
		\underline{b}, \overline{b} \colon \of 0, \infty\of \, \to \bR, \quad 
		\underline{a}, \overline{a} \colon \of 0, \infty\of \, \to \bR_+
		\]
		such that
		\[
		\underline{b}_t(\omega) \leq \overline{b}_t(\omega), \quad 
		\underline{a}_t(\omega) \leq \overline{a}_t(\omega), \quad 
		(t, \omega) \in \of 0, \infty \of.
		\]
		This situation is included in our setting. For instance, take \(F := [0, 1] \times [0, 1]\) and, for \(((f_1, f_2),t,\omega) \in F \times \of 0, \infty \of \),
		\begin{align*}
		b ((f_1, f_2), t, \omega) &:= \underline{b}_t (\omega) + f_1 \cdot (\overline{b}_t (\omega) - \underline{b}_t (\omega)), \\
		a ((f_1, f_2), t, \omega) &:= \underline{a}_t (\omega) + f_2 \cdot (\overline{a}_t (\omega) - \underline{a}_t (\omega)).
		\end{align*}
		In the following, we consider these choices of \(b\) and \(a\).
		Evidently, the convexity assumption given by Condition \ref{cond: convexity} is satisfied. Further, the linear growth Condition \ref{cond: LG} holds once the functions \(\underline{b}, \overline{b}, \underline{a}\) and \(\overline{a}\) satisfy itself linear growth conditions, i.e., in case for every \(T > 0\) there exists a constant \(C = C_T > 0\) such that 
		\begin{align*}
		|\underline{b}_t (\omega)|^2 + |\overline{b}_t (\omega)|^2 + | \underline{a}_t (\omega)| + | \overline{a}_t (\omega)| \leq C \Big( 1 + \sup_{s \in [0, t]} | \omega (s) |^2 \Big)
		\end{align*}
		for all \((t, \omega) \in \of 0, T\gs\). Similarly, the continuity Condition \ref{cond: joint continuity in all} is implied by (joint) continuity of \(\underline{b}, \overline{b}, \underline{a}\) and~\(\overline{a}\).
		Under these conditions, i.e., continuity and linear growth, Theorem \ref{theo: upper hemi} implies that the correspondence \( (t, \omega) \mapsto \cR(t,\omega)\) is upper hemicontinuous with compact values, while Theorem \ref{theo: viscosity} implies that the value function \(v\) is upper semicontinuous.
		Moreover, if \(b\) and \(a\) satisfy the local Lipschitz condition given by Condition \ref{cond: lipschitz}, 
		then Theorem \ref{theo: lower hemi} shows that the correspondence \( (t, \omega) \mapsto \cR(t,\omega)\) is lower hemicontinuous. In particular, the value function \(v\) is then continuous. For instance, the local Lipschitz conditions hold in case for every \(T, M > 0\) there exists a constant \(C = C_{T, M} > 0\) such that \(\underline{a}_t (\omega) \geq 1/C\) and 
		\begin{align*}
		|\underline{b}_t (\omega) - \underline{b}_t (\alpha)| + |\overline{b}_t (\omega) - \overline{b}_t (\alpha)| &\leq C  \sup_{s \in [0, t]} | \omega (s) - \alpha (s) |, \\
  | \underline{a}_t (\omega) - \underline{a}_t (\alpha)| + | \overline{a}_t (\omega) - \overline{a}_t (\alpha)| &\leq C  \sup_{s \in [0, t]} | \omega (s) - \alpha (s) |,
		\end{align*}
		for all \(\omega, \alpha \in \Omega \colon \sup_{s \in [0, t]} |\omega (s)| \vee |\alpha (s)| \leq M\) and \(t \in [0, T]\).
	\end{example}
	
	\begin{remark} 
Let us also comment on a result from the paper \cite{nutz} that deals with a regularity property of a value function \(V \colon \of 0, T \gs \hspace{0.05cm}\to \mathbb{R}\) that is closely related to ours. Let \(\mathbf{D} \colon \of 0, T \gs \hspace{0.05cm} \twoheadrightarrow \mathbb{S}^d_+\) be a progressively measurable closed-valued correspondence and define, for \(\delta > 0\) and \(D \subseteq \mathbb{S}^d_+\), \(\on{Int}^\delta D := \{ x \in D \colon B_\delta (x) \subseteq D\}\), where \(B_\delta (x)\) denotes the ball around \(x\) with radius \(\delta\). Further, for \((t, \omega) \in \of 0, T \gs\), let \(\mathcal{M} (t, \omega)\) be the set of all \(P \in \fPas (t)\) such that \(P (X^t = \omega^t) = 1\), \((\llambda \otimes P)\)-a.e. \(b^P_{\cdot + t} = 0\), i.e., that are local martingale measures after time \(t\), and for that there exists a \(\delta = \delta (t, \omega, P)> 0\) such that \((\llambda \otimes P)\)-a.e. \(a^P_{\cdot + t} \in \mathbb{S}^d_{++}\) and 
\(
a^P_{\cdot + t} \in \on{Int}^\delta \mathbf{D} (\cdot + t, X).
\)
The value function \(V\) from \cite{nutz} is given by
\[
V (t, \omega) = \sup_{P \in \mathcal{M} (t, \omega)} E^P \big[ \xi \big],
\]
where \(\xi \colon \of 0, T \gs \hspace{0.05cm} \to \mathbb{R}\) is a bounded uniformly continuous input function.
Our setting covers the (in a certain sense) limiting case where \(\delta (t, \omega, P) = 0\), cf. \cite[Remark 3.7]{nutz}. 
Corollary 3.6 from \cite{nutz} shows that, for fixed \(t \in [0, T]\), the map \(\omega \mapsto V (t, \omega)\) is lower semicontinuous in case \(\mathbf{D}\) satisfies the uniform continuity property as defined in \cite[Definition 3.2]{nutz}. To compare this regularity result to our main theorems, we notice that the special case \(\mathbf{D} (t, \omega) = [\underline{a}_t (\omega), \overline{a}_t (\omega)]\) satisfies the hypothesis from \cite[Corollary 3.6]{nutz} when the functions \(\underline{a}\) and \(\overline{a}\) are uniformly continuous in the sense that for all \(\varepsilon > 0\) there exists a \(\delta > 0\) such that
\[
\omega, \alpha \in \Omega, \ \sup_{s \in [0, T]} | \omega (s) - \alpha (s) | < \delta \ \ \Longrightarrow \ \ \sup_{s \in [0, T]} | a_s (\omega) - a_s (\alpha) | < \varepsilon, \quad a = \underline{a}, \overline{a}. 
\]
In our setting, mere continuity of \(\underline{a}\) and \(\overline{a}\) imply upper semicontinuity of our value function, see Example~\ref{ex: interval}, while we require some local Lipschitz continuity for lower semicontinuity. 
	\end{remark}

	In the following corollary we summarize our main observations concerning the correspondence \(\cR\) and the value function \(v\).
	\begin{corollary} [Continuity of \(\cR\) and the value function] \label{coro: main value fuction summary}
		Assume that \(F\) is a compact metrizable space and that the Conditions \ref{cond: LG}, \ref{cond: joint continuity in all}, \ref{cond: convexity} and \ref{cond: lipschitz} hold. Then, the correspondence \((t, \omega) \mapsto \cR(t, \omega)\) is continuous and the value function \(\of 0, T\gs \ni (t, \omega) \mapsto v (t, \omega)\) is a viscosity solution to \eqref{eq: PIDE}.
	\end{corollary}

	It is a natural question whether the value function is the unique viscosity solution to the nonlinear PPDE \eqref{eq: PIDE}. Uniqueness results for certain Hamilton--Jacobi--Bellman PPDEs were recently proved in \cite{cosso2,zhou}. In the following we apply a theorem from \cite{zhou} and show that, under global Lipschitz conditions, our value function is the unique viscosity solution among all solutions that satisfy a certain Lipschitz property that we explain now. 
	We define \(\dd = \dd_T \colon \of 0, T\gs \times \of 0, T\gs \to \bR_+\) by
	\begin{equation} \label{eq: def d}
		\begin{split}
			\dd ( (t,\omega), (s, \alpha)) :=  \Big( 1 + \sup_{r \in [0, t]} \|\omega (r)\| + \ & \sup_{r \in [0, s]}  \|\alpha (r)\|\Big) | t - s |^{1/2} \\&+ \sup_{r \in [0, T]} \|\omega(r \wedge t) - \alpha( r \wedge s) \|.
		\end{split}
	\end{equation}
	A function \(f \colon \of 0, T\gs \to \mathbb{R}\) is said to be {\em \(\dd\)-Lipschitz continuous}, if there exists a constant \(L = L_T > 0\) such that 
	\[
	| f (t, \omega) - f (s, \alpha) | \leq L \dd ((t, \omega), (s, \alpha))
	\] 
	for all \((t, \omega), (s, \alpha) \in \of 0, T\gs\). We remark that \(\dd\)-Lipschitz continuity entails continuous on~\(\of 0, T\gs\).
	
	\begin{condition}[Global Lipschitz continuity] \label{cond: global lip and bdd}
		There exists a dimension \(\dr \in \mathbb{N}\) and a \(\mathcal{B}(F) \otimes \mathscr{P}\)-measurable function \(\sigma \colon F \times \of 0, \infty\of \hspace{0.05cm} \to \bR^{\d \times \dr}\) such that \(a = \sigma \sigma^*\).
		Moreover, there exists a constant \(C = C_{T} > 0\) such that 
		\[
		\|b (f, t, \omega) - b (f, t, \alpha)\| + \|\sigma (f, t, \omega) - \sigma (f, t, \alpha) \|_o \leq C \sup_{s \in [0, t]} \| \omega (s) - \alpha (s) \|
		\]
		for all \(\omega, \alpha \in \Omega\) and \((f, t) \in F \times [0, T]\).
	\end{condition}
	
	\begin{theorem} \label{theo: holder estimate value function}
		Assume that \(F\) is a compact metrizable space and that the Conditions~\ref{cond: LG}, \ref{cond: continuity control}, \ref{cond: convexity} and \ref{cond: global lip and bdd} hold. Furthermore, suppose that the terminal function \(\psi\) is bounded and Lipschitz continuous, i.e., there exists a constant \(L > 0\) such that
		\[
		|\psi (\omega) - \psi (\alpha) | \leq L \sup_{s \in [0, T]} \|\omega (s) - \alpha (s) \|
		\]
		for all \(\omega, \alpha \in \Omega\). Then, the value function is \(\dd\)-Lipschitz continuous.
	\end{theorem}
	Estimates similar to Theorem \ref{theo: holder estimate value function} are standard in stochastic optimal control, see, e.g., \cite[Theorem 2.11]{zhou} where a version of Theorem \ref{theo: holder estimate value function} for a path-dependent control framework is given.
	
	When it comes to mere continuity, Theorem \ref{theo: holder estimate value function} requires stronger regularity assumption in the path variable than Corollary \ref{coro: main value fuction summary}. In particular, also the input function~\(\psi\) has to be Lipschitz continuous in Theorem \ref{theo: holder estimate value function}.
	Contrary to Corollary \ref{coro: main value fuction summary}, Theorem~\ref{theo: holder estimate value function} shows that the value function is Lipschitz continuous in space for fixed, but arbitrary, times with a uniform (in time) Lipschitz constant. 
	We think this observation is of independent interest.

	We are in the position to present a uniqueness result for the value function.
	\begin{theorem} \label{thm: uniqueness}
		Assume that \(F\) is a compact metrizable space and that the Conditions~\ref{cond: LG}, \ref{cond: joint continuity in all}, \ref{cond: convexity} and \ref{cond: global lip and bdd} hold. Furthermore, suppose that the terminal function \(\psi\) is bounded and Lipschitz continuous, i.e., there exists a constant \(L > 0\) such that
		\[
		|\psi (\omega) - \psi (\alpha) | \leq L \sup_{s \in [0, T]} \|\omega (s) - \alpha (s) \|
		\]
		for all \(\omega, \alpha \in \Omega\). Then, the value function \(v\) is the unique viscosity solution that is bounded and \(\dd\)-Lipschitz continuous.
	\end{theorem}
	\begin{proof}
		Thanks to the Theorems \ref{theo: viscosity} and \ref{theo: holder estimate value function}, the value function \(v\) is a viscosity solution to \eqref{eq: PIDE} that is \(\dd\)-Lipschitz continuous. Hence, the uniqueness statement follows from \cite[Theorem~6.2]{zhou}.
	\end{proof}

 	\begin{remark}
 		Under more assumptions on the coefficients, a uniqueness result in a larger class of continuous function has been proved in the recent paper \cite{cosso2}.
 	\end{remark}
 \begin{example}
\quad 
     \begin{enumerate}
         \item[\textup{(i)}]
         In the Markovian case where \(b(f, t, \omega)\) and \(a (f, t, \omega)\) depend on \((t, \omega)\) only through the value \(\omega (t)\), assumptions in the spirit of those from Theorem \ref{thm: uniqueness} imply that the (point-dependent) value function is unique among all bounded viscosity solutions, see \cite{CN22b} for a precise result in this direction. 
     \item[(ii)] Let us again comment on the situation from Example \ref{ex: interval}, i.e., the case where \(d = 1\) and \[\Theta (t, \omega) = [\underline{b}_t (\omega), \overline{b}_t (\omega)] \times [\underline{a}_t (\omega), \overline{a}_t (\omega)]\]
		with predictable functions
		\[
		\underline{b}, \overline{b} \colon \of 0, \infty\of \, \to \bR, \quad 
		\underline{a}, \overline{a} \colon \of 0, \infty\of \, \to \bR_+
		\]
		such that
		\[
		\underline{b}_t(\omega) \leq \overline{b}_t(\omega), \quad 
		\underline{a}_t(\omega) \leq \overline{a}_t(\omega), \quad 
		(t, \omega) \in \of 0, \infty \of.
		\]
  As explained in Example \ref{ex: interval}, this situation can be included in our framework with \(F := [0, 1] \times [0, 1]\) and 
  \begin{align*}
  b ( (f_1, f_2), t, \omega) &:= \underline{b}_t (\omega) + f_1 \cdot (\overline{b}_t (\omega) - \underline{b}_t (\omega)), \\
  a ((f_1, f_2), t, \omega) &:= \underline{a}_t (\omega) + f_2 \cdot (\overline{a}_t (\omega) - \underline{a}_t (\omega)),
  \end{align*}
  for \( (f_1, f_2) \in F\) and \((t, \omega) \in \of 0, \infty\of\). In this setting, Condition \ref{cond: lipschitz} holds in case there exists a constant \(C > 0\) such that 
  \[
  | \underline{b} (t, \omega) |, | \overline{b} (t, \omega) |  \leq C, \qquad \frac{1}{C} \leq \underline{a} (t, \omega), \overline{a} (t, \omega) \leq C,
  \]
and
  \begin{align*}
  	| \underline{b} (t, \omega) - \underline{b} (t, \alpha) | + |\overline{b} (t, \omega) - \overline{b} (t, \alpha) | &\leq C \sup_{s \in [0, t]} | \omega (s) - \alpha (s) |, \\
  	| \underline{a} (t, \omega) - \underline{a} (t, \alpha) | + |\overline{a} (t, \omega) - \overline{a} (t, \alpha) | &\leq C \sup_{s \in [0, t]} | \omega (s) - \alpha (s) |,
  \end{align*}
  for all \(t \in [0, T]\) and \(\omega, \alpha \in \Omega\).
%
     \end{enumerate}
 \end{example}


\section{The martingale problem} \label{sec: MP}

As a final main result, we show that \(\mathcal{E}\) solves a type of nonlinear martingale problem. This result supports our interpretation of \(\mathcal{E}\) as a nonlinear continuous semimartingale. We restrict our attention to the one-dimensional situation, i.e., we presume that \(\d = 1\).
Let \(\mathbb{M}_{icx}\) be the set of all \(\phi \in C^2(\mathbb{R}; \mathbb{R})\) such that \(\phi', \phi'' \geq 0\). Furthermore, for \(n > 0\), we set 
\[
\rho_n := \inf \{t \geq 0 \colon |X_t| \geq n\}.
\]
\begin{definition}
	We say that \(\mathcal{E}\) solves the \emph{martingale problem associated to \(G\)} if for all \(n \in \mathbb{N}\) and \(\phi \in \mathbb{M}_{icx}\) the process
	\[
	\phi (X_{t \wedge \rho_n}) - \int_0^{t \wedge \rho_n} G (r, X, \phi) dr, \quad t \in \mathbb{R}_+,
	\]
	is an \(\mathcal{E}\)-martingale, i.e., 
	\[
	\mathcal{E}_s \Big( \phi (X_{t \wedge \rho_n}) - \phi (X_{s \wedge \rho_n}) - \int_{s \wedge \rho_n}^{t \wedge \rho_n} G (r, X, \phi) dr \Big) (\omega) = 0, \quad s \leq t, \ \omega \in \Omega.
	\]
\end{definition}
Of course, the acronym \(icx\) stands for increasing and convex. Although this set of test functions looks non-standard at first sight, it is perfectly fine in the linear setting, because it is well-known to be measure-determining for the set of Borel probability measures on \(\mathbb{R}\) with finite first moment (cf. \cite[Theorems~5.2, 5.3]{muller97} and \cite[Theorem~2.1]{DenMul02}).

\begin{theorem} \label{theo: MP}
	Suppose that \(F = F_0 \times F_1\) for two compact metrizable spaces  \(F_0\) and \(F_1\), that \(b (f, t, \omega) = b(f_0, t, \omega)\) and \(a (f, t, \omega) = a(f_1, t, \omega)\) for all \((f, t, \omega) = (f_0, f_1, t, \omega) \in F \times \of 0, \infty\of\), and that the Conditions~\ref{cond: LG} and \ref{cond: joint continuity in all} hold. Then, \(\mathcal{E}\) solves the martingale problem associated to \(G\).
\end{theorem}
In \cite{guo2018martingale} a related result for generalized \(G\)-Brownian motion was proved. 
Let us discuss the relation of Theorem \ref{theo: MP} and the approach from \cite{guo2018martingale} in more detail. In \cite{guo2018martingale} it is shown that a generalized \(G\)-Brownian motion (see \cite{peng2010}) solves a nonlinear martingale problem that is defined not only via test functions but also via a class of nonlinear test generators. This approach uses the power of \(G\)-It\^o calculus in a crucial manner. In our framework we cannot rely on such a stochastic calculus. Instead, we only identify a suitable class of test functions, namely \(\mathbb{M}_{icx}\), for that we can prove a nonlinear martingale property. In that sense our treatment of the martingale problem is certainly less complete than those for the generalized \(G\)-Brownian motion from \cite{guo2018martingale}. We think that Theorem~\ref{theo: MP} is a good indicator for our interpretation of \(\mathcal{E}\) as a nonlinear continuous semimartingale.


\section{Proof of the Dynamic Programming Principle: Theorem \ref{theo: DPP}} \label{sec: pf DPP}
The proof is based on an application of \cite[Theorem 2.1]{ElKa15}, which provides three abstract conditions on the correspondence \(\cA\) which imply the DPP. The work lies in the verification of these conditions. For reader's convenience, let us restate them. 
For a probability measure \(P\) on \((\Omega, \mathcal{F})\), a kernel \(\Omega \ni \omega \mapsto Q_\omega \in \mathfrak{P}(\Omega)\), and a finite stopping time \(\tau\), we define the pasting measure
\[
(P \otimes_\tau Q) (A) := \iint \1_A (\omega \otimes_{\tau(\omega)} \omega') Q_\omega (d \omega') P(d \omega), \quad A \in \mathcal{F}.
\]
We are in the position to formulate the assumptions from \cite[Theorem 2.1]{ElKa15}.
\begin{enumerate}
	\item[\textup{(i)}]
	\emph{Measurable graph condition:} The set 
	\[
	\big\{ (t, \omega, P) \in \of 0, \infty\of \hspace{0.05cm} \times\hspace{0.05cm} \mathfrak{P}(\Omega) \colon P \in \cA(t, \omega) \big\}
	\]
	is analytic.
	\item[\textup{(ii)}]
	\emph{Stability under conditioning:} For any \(t \in \mathbb{R}_+\), any stopping time \(\tau\) with \(t \leq \tau < \infty\), and any \(P \in \cA(t, \alpha)\) there exists a family \(\{P (\cdot | \mathcal{F}_\tau) (\omega) \colon\) \(\omega \in \Omega \}\) of regular \(P\)-conditional probabilities given \(\mathcal{F}_\tau\) such that \(P\)-a.s. \(P (\cdot | \mathcal{F}_\tau) \in \cA(\tau, X)\).
	\item[\textup{(iii)}]
	\emph{Stability under pasting:} For any \(\alpha \in \Omega\), \(t \in \mathbb{R}_+\), any stopping time \(\tau\) with \(t \leq \tau < \infty\), any \(P \in \cA(t, \alpha)\) and any \(\mathcal{F}_\tau\)-measurable map \(\Omega \ni \omega \mapsto Q_\omega \in \mathfrak{P}(\Omega)\) the following implication holds:
	\[
	P\text{-a.s. } Q \in \cA (\tau, X)\quad \Longrightarrow \quad P \otimes_\tau Q \in \cA(t, \alpha).
	\]
\end{enumerate}

\begin{remark}
	To verify (ii) above, it is necessary to introduce the sets $\fPas(t)$ for $t > 0$.
	Notice that for a finite stopping time \(\tau\) it is not necessarily true that \(P \in \fPas \) implies  \( \{P(\cdot | \mathcal{F}_\tau) (\omega) \colon \omega \in \Omega\} \subset \fPas.\)
	Indeed, we only have the semimartingale property of $X$ after $\tau(\omega)$ under $P(\cdot | \mathcal{F}_\tau)(\omega)$.
	This forces us to study the correspondence $t \mapsto \fPas(t)$. To see the issue, let \(P\) be the Wiener measure, i.e., the law of a one-dimensional Brownian motion. Then, for any \(t > 0\) and \(P\)-a.a. \(\omega \in \Omega\), \(P(\cdot | \mathcal{F}_t) (\omega) \not \in \fPas\). This follows simply from the fact that \(P\)-a.a. \(\omega \in \Omega\) are locally of infinite variation, which implies that the stopped process \(X_{\cdot \wedge t}\), which coincides \(P (\cdot| \cF_t)(\omega)\)-a.s. with \(\omega (\cdot \wedge t)\), is no \(P(\cdot | \mathcal{F}_t) (\omega)\)-semimartingale for \(P\)-a.a. \(\omega \in \Omega\) (see \cite[Proposition~I.4.28]{JS}). 
\end{remark}
In the following three sections we check these properties. In the fourth (and last) section, we finalize the proof of Theorem \ref{theo: DPP}.

	Before we start our program let us shortly comment on our strategy of proof and relate it to existing literature.
	Dynamic programming principles for nonlinear expectations related to nonlinear L\'evy processes, and more general Markovian semimartingales, have been proved in \cite{hol16, neufeld2017nonlinear}. In contrast to our approach, which has also been used in \cite{ElKa15} for controlled martingale problems, the proofs in \cite{hol16, neufeld2017nonlinear} are based on an abstract result from \cite{NVH}. In general, the methodologies are different in the sense that the uncertainty sets of measures in \cite{hol16, neufeld2017nonlinear} consist of laws of semimartingales that start at zero, while time and path dependence persists in both the set \(\Theta\) and the test function~\(\psi\). In our setting, however, the dependence on \((t, \omega)\) only enters via the set \(\mathcal{C}(t, \omega)\). This construction seems to us closer to those of a classical linear conditional expectation. 
	
	Compared to \cite{hol16, neufeld2017nonlinear}, the main difference in our setting is the time-dependence \(t \mapsto \fPas (t)\) for which we have to prove new measurability and stability properties. While this differs from \cite{hol16, neufeld2017nonlinear} on a technical level, we closely follow, however, their general strategy of proof.

\subsection{Measurable graph condition}
The proof of the measurable graph condition is split into several parts. 
For \(t \in \mathbb{R}_+\), we define the usual shift \(\th_t \colon \Omega \to \Omega\) by \(\th_t (\omega) = \omega(\cdot + t)\) for all \(\omega \in \Omega\).
The next two lemmata give some rather elementary observations.
\begin{lemma} \label{lem: map prod meas}
	The map \((t, P) \mapsto P \circ \th_t^{-1} =: P_t\) is continuous. 
\end{lemma}
\begin{proof}
	This follows from the continuity of \((t, \omega) \mapsto \theta_t (\omega)\) and \cite[Theorem 8.10.61]{bogachev}. 
\end{proof}

\begin{lemma} \label{lem: filtraation shift}
	Let \(Y = (Y_t)_{t \geq 0}\) be an \(\bR^\d\)-valued continuous process and set \(\mathcal{F}^Y_t := \sigma (Y_s, s\leq t)\) for \(t \in \mathbb{R}_+\). Then, \(
	Y^{-1} (\mathcal{F}_{s+}) = \mathcal{F}^Y_{s+}\) for all \(s \in \mathbb{R}_+.\)
\end{lemma}
\begin{proof}
	First of all, it is clear that \(Y^{-1} (\mathcal{F}_s) = \mathcal{F}^Y_s\) for all \(s \in \mathbb{R}_+\).
	Since \(\mathcal{F}_{s+} \subset \mathcal{F}_{s + \varepsilon}\) for all \(\varepsilon > 0\), we obtain
	\(
	Y^{-1} (\mathcal{F}_{s +}) \subset \bigcap_{\varepsilon > 0} Y^{-1} (\mathcal{F}_{s + \varepsilon}) = \mathcal{F}^Y_{s +}.
	\)
	Conversely, take \(A \in \mathcal{F}^Y_{s+}\). By definition, \(A \in Y^{-1} (\mathcal{F}_{s + 1/n})\) for every \(n \in \mathbb{N}\). Hence, for each \(n \in \mathbb{N}\) there exists a set \(B_{n} \in \mathcal{F}_{s + 1/n}\) such that \(A = Y^{-1}(B_n)\). Define 
	\(
	G := \bigcap_{n \in \mathbb{N}} \bigcup_{m \geq n} B_{m}\) and notice that \(G \in \mathcal{F}_{s +}.
	\)
	Finally, as
	\(
	Y^{-1} (G) = \bigcap_{n \in \mathbb{N}} \bigcup_{m \geq n} Y^{-1} (B_m) = A,
	\)
	we conclude that \(A \in Y^{-1}(\mathcal{F}_{s+})\). The proof is complete.
\end{proof}

The following lemma is a restatement of \cite[Lemma 2.9 a)]{jacod80} for a path-continuous setting. 

\begin{lemma} \label{lem: jacod restatements}
	Let \(\B^* = (\Omega^*, \mathcal{F}^*, (\mathcal{F}^*_t)_{t \geq 0}, P^*)\) and  \(\B' = (\Omega', \mathcal{F}', (\mathcal{F}'_t)_{t \geq 0}, P')\) be two filtered probability spaces with right-continuous filtrations and the property that there is a map \(\phi \colon \Omega' \to \Omega^*\) such that
	\(
	\phi^{-1} (\mathcal{F}^*) \subset \mathcal{F}',P^* = P' \circ \phi^{-1}\) and \(\phi^{-1} (\mathcal{F}^*_t) = \mathcal{F}'_t\) for all \(t \in \mathbb{R}_+\).
	Then,
	\(X^*\) is a \(\d\)-dimensional continuous semimartingale (local martingale) on \(\B^*\) if and only if \(X' = X^* \circ \phi\) is a \(\d\)-dimensional continuous semimartingale (local martingale) on~\(\B'\). Moreover, \((B^*, C^*)\) are the characteristics of \(X^*\) if and only if \((B^* \circ \phi, C^* \circ \phi)\) are the characteristics of \(X' = X^* \circ \phi\).
\end{lemma}

\begin{lemma} \label{lem: set identity}
	We have
	\begin{align*}\big\{(t&, \omega, P) \colon P \in \fPas (t), 
	(\llambda \otimes P)\text{-a.e. } (dB^P_{\cdot + t} /d\llambda, dC^P_{\cdot + t}/d\llambda) \in \Theta (\cdot + t, \omega \otimes_t X)  \big\} \\& = \big\{ (t, \omega, P) \colon P_t \in \fPas,
	(\llambda \otimes P_t)\text{-a.e. } (dB^{P_t} /d\llambda, dC^{P_t}/d\llambda) \in \Theta (\cdot + t, \omega \ \widetilde{\otimes}_t\ X)  \big\}.\end{align*}
\end{lemma}

\begin{proof}
	Since \((\omega\ \widetilde{\otimes}_t\ X) \circ \theta_t = \omega \otimes_t X\), the claim follows from Lemmata \ref{lem: filtraation shift} and \ref{lem: jacod restatements}.
\end{proof}

Finally, we are in the position to prove the measurable graph condition.

\begin{lemma} \label{lem: correspondence 2}
	\label{lem_atw_measurable}
	The set $\{(t,\omega, P) \in \of 0, \infty\of \hspace{0.05cm} \times\hspace{0.05cm} \mathfrak{P}(\Omega) \colon P \in \cA(t,\omega) \}$ is Borel. 
\end{lemma}
\begin{proof} 
	Notice that
	$\{ (t,\omega, P) \in \of 0, \infty\of \hspace{0.05cm} \times\hspace{0.05cm} \mathfrak{P}(\Omega) \colon P(X^t = \omega^t) = 1 \} $
	is Borel by \cite[Theorem~8.10.61]{bogachev}.
	Further, notice that 
	\begin{align*}
	\big\{ (t&, \omega, P)  \colon P \in \fPas, \  ( \llambda \otimes P)\text{-a.e. } (dB^P /d\llambda, dC^P/d\llambda) \in \Theta (\cdot + t, \omega \ \widetilde{\otimes}_t\ X)  \big\}
	\\&= \big\{ (t, \omega, P) \colon P \in \fPas, \  (\llambda \otimes P)( (\cdot + t, \omega \ \widetilde{\otimes}_t\ X, dB^P /d\llambda, dC^P/d\llambda) \not \in \on{gr} \Theta ) = 0 \big\}.
	\end{align*}
	By virtue of \cite[Theorem 2.6]{neufeld2014measurability} and \cite[Theorem~8.10.61]{bogachev}, this set is Borel.
	Consequently, the lemma follows from Lemmata \ref{lem: map prod meas} and \ref{lem: set identity}.
\end{proof}

\subsection{Stability under Conditioning}
Next, we check stability under conditioning. Throughout this section, we fix \((t^*, \omega^*) \in \of 0, \infty\of\), a stopping time \(\tau\) with \(t^* \leq \tau <\infty\), and a probability measure \(P\) on \((\Omega, \mathcal{F})\) such that \(P (X = \omega^* \text{ on } [0, t^*]) = 1\). We denote by \(P(\cdot | \mathcal{F}_\tau)\) a version of the regular conditional \(P\)-probability given \(\mathcal{F}_\tau\).
The following three observations are simple but useful. 
\begin{lemma} \label{lem: countably generated easy observation}
	There exists a \(P\)-null set \(N \in \mathcal{F}_\tau\) such that \(P (A | \mathcal{F}_\tau) (\omega) = \1_A (\omega)\) for all \(A \in \mathcal{F}_\tau\) and~\(\omega \not \in N\).
\end{lemma}
\begin{proof}
	Notice that \(\mathcal{F}_\tau = \sigma (X_{t \wedge \tau}, t \in \mathbb{Q}_+)\) is countably generated (see \cite[Lemma 1.3.3]{SV}). Now, the claim is classical (\cite[Theorem 9.2.1]{stroockanalytic}).
\end{proof}

\begin{lemma} \label{lem: eq concatenation}
	If \(\omega \mapsto Q_\omega\) is a kernel from \(\mathcal{F}\) into \(\mathcal{F}\) such that 
	\(Q_\omega ( \omega = X \text{ on } [0, \tau(\omega)]) = 1\) for \(P\)-a.a. \(\omega \in \Omega\), then 
	\(
	Q_\omega (\omega \otimes_{\tau(\omega)} X = X) = 1
	\)
	for \(P\)-a.a. \(\omega \in \Omega\).
	In particular, there exists a \(P\)-null set \(N \in \mathcal{F}_\tau\) such that \(P (\omega \otimes_{\tau(\omega)} X = X | \mathcal{F}_\tau) (\omega) = 1\) for all \(\omega \not \in N\).
\end{lemma}
\begin{proof}
	The first claim is obvious and the second follows from the first and Lemma \ref{lem: countably generated easy observation}.
\end{proof}

\begin{lemma} \label{lem: predi pasting}
	Take a predictable process \(H = (H_t)_{t \geq 0}, t \in \mathbb{R}_+\) and \(\omega \in \Omega\). Then, \(H_{\cdot + t} (\omega \otimes_{t} X)\) is predictable for the (right-continuous) filtration generated by \(X_{\cdot + t}\).
\end{lemma}
\begin{proof}
	This follows from \cite[Proposition~10.35 (d)]{jacod79}.
\end{proof}

The next lemma is a partial restatement of \cite[Theorem~1.2.10]{SV}.
\begin{lemma} \label{lem: stroock varadhan cond mg}
	If \(M - M^{t^*}\) is a \(P\)-martingale, then there exists a \(P\)-null set \(N\) such that, for all \(\omega \not \in N\), \(M - M^{\tau (\omega)}\) is a \(P(\cdot | \mathcal{F}_\tau) (\omega)\)-martingale.
\end{lemma}

The following lemma should be compared to \cite[Theorem 3.1]{neufeld2017nonlinear}.
\begin{lemma} \label{lem: characteristics under conditioning}
	Suppose that \(P \in \fPs(t^*)\) and denote the \(P\)-semimartingale characteristics of \(X_{\cdot + t^*}\) by \((B_{\cdot + t^*}, C_{\cdot + t^*})\). Then, there exists a \(P\)-null set \(N\) such that, for all \(\omega \not \in N\), \(X_{\cdot + \tau(\omega)}\) is a \(P (\cdot | \mathcal{F}_\tau) (\omega)\)-semimartingale (for its right-continuous natural filtration) and the corresponding \(P(\cdot | \mathcal{F}_\tau)(\omega)\)-characteristics are given by 
	\begin{align*}
	&(B_{\cdot + \tau(\omega)} - B_{\tau (\omega)}) (\omega \otimes_{\tau (\omega)} X), \quad (C_{\cdot + \tau(\omega)} - C_{\tau (\omega)}) (\omega \otimes_{\tau (\omega)} X). 
	\end{align*}
\end{lemma}
\begin{proof}
	For notational convenience,  we prove the claim only for \(t^* = 0\).
	Define \(M := X - B - X_0\) and 
	\[\tau_n := \inf \{t \geq 0 \colon \|M_t\| \geq n\}, \quad n > 0.\] The process \(M^{\tau_n}\) is clearly bounded. 
	Thus, by definition of the first characteristic, the process \(M^{\tau_n}\) is a \(P\)-martingale. By Lemma \ref{lem: stroock varadhan cond mg}, there exists a \(P\)-null set \(N\), such that, for all \(\omega \not \in N\), \(M^{\tau_n} - M^{\tau_n \wedge \tau(\omega)}\) is a \(P(\cdot | \mathcal{F}_\tau) (\omega)\)-martingale. As 
	\[
	E^P \Big[ P \Big( \liminf_{n \to \infty} \tau_n < \infty | \mathcal{F}_\tau \Big)\Big] = 0, 
	\]
	possibly making \(N\) a bit larger, we can w.l.o.g. assume that \(P(\tau_n \to \infty | \mathcal{F}_\tau)(\omega) = 1\) for all \(\omega \not \in N\). 
	Now, using Lemmata \ref{lem: countably generated easy observation} and \ref{lem: eq concatenation}, possibly making \(N\) again a bit larger, we get \(P(\cdot | \mathcal{F}_\tau)(\omega)\)-a.s.
	\begin{align*}
	X^{\tau_n} - X^{\tau_n \wedge \tau (\omega)} = (M^{\tau_n}  &- M^{\tau_n \wedge \tau (\omega)})(\omega \otimes_{\tau(\omega)} X) + (B^{\tau_n} - B^{\tau_n \wedge \tau (\omega)} (\omega)) (\omega \otimes_{\tau(\omega)} X).
	\end{align*}
	By the tower rule, \((M^{\tau_n}_{\cdot + \tau(\omega)} - M_{\tau_n \wedge \tau(\omega)})(\omega \otimes_{\tau(\omega)} X)\) is a \(P(\cdot |\mathcal{F}_\tau)(\omega)\)-martingale for the filtration generated by \(X_{\cdot + \tau(\omega)}\).
	Hence, for all \(\omega \not \in N\), by virtue of Lemma \ref{lem: predi pasting}, we deduce that \(X_{\cdot + \tau(\omega)}\) is a \(P (\cdot | \mathcal{F}_\tau) (\omega)\)-semimartingale (for its right-continuous canonical filtration), and further we get the formula for the first characteristic. 
	For the second characteristic, notice that, for all \(\omega \not \in N\), \(P(\cdot | \mathcal{F}_\tau)(\omega)\)-a.s.
	\begin{align*}
	[X - X^{\tau (\omega)}, X - X^{\tau (\omega)}] &= 
	[X, X] - 2 [X, X^{\tau (\omega)}] + [X^{\tau(\omega)}, X^{\tau(\omega)}] 
	\\&= 
	C - C^{\tau(\omega)} 
 \\&= (C - C^{\tau(\omega)}(\omega))(\omega \otimes_{\tau(\omega)} X).
	\end{align*}
	This observation yields the formula for the second characteristic.
\end{proof}
We are in the position to deduce stability under conditioning. The following corollary should be compared to \cite[Corollary 3.2]{neufeld2017nonlinear}.
\begin{corollary} \label{coro: stab cond}
	If \(P \in \cA (t^*, \omega^*)\), then \(P\)-a.s. \(P (\cdot | \mathcal{F}_\tau) \in \cA (\tau, X)\). 
\end{corollary}
\begin{proof}
	Let \(N\) be as in Lemma \ref{lem: countably generated easy observation} and take \(\omega \not \in N\). Then, \(P (\tau = \tau (\omega)| \mathcal{F}_\tau)(\omega) = 1\) and
	\[
	P (X^{\tau (\omega)} = \omega^{\tau (\omega)} | \mathcal{F}_\tau) (\omega) = P (X^{\tau} = \omega^{\tau (\omega)} | \mathcal{F}_\tau) (\omega) = \1 \{X^{\tau} = \omega^{\tau (\omega)}\} (\omega) = 1.
	\]
	Thanks to Lemma \ref{lem: characteristics under conditioning}, for \(P\)-a.a. \(\omega \in \Omega\), \(P(\cdot | \mathcal{F}_\tau)  (\omega) \in \fPas (\tau (\omega))\) and the Lebesgue densities of the characteristics of \(X_{\cdot + \tau(\omega)}\) are given by \((b_{\cdot + \tau(\omega)}, a_{\cdot + \tau(\omega)})(\omega \otimes_{\tau(\omega)} X)\), where \((b_{\cdot + t^*}, a_{\cdot + t^*})\) are the Lebesgue densities of the characteristics of \(X_{\cdot + t^*}\) under \(P\). Thus, we get from Lemma \ref{lem: eq concatenation}, Fubini's theorem, Lemma~\ref{lem: countably generated easy observation}, the tower rule, and \(P \in \cA(t^*, \omega^*)\) that 
	\begin{align*}
	\iint_{\tau (\omega)}^\infty E^P \big[ \1 \{ (b_t, a_t) & (\omega \otimes_{\tau (\omega)} X) \not \in \Theta (t, \omega \otimes_{\tau (\omega)} X) \} | \mathcal{F}_\tau \big] (\omega) dt P(d \omega) 
	\\&= \iint_{\tau (\omega)}^\infty E^P \big[ \1 \{ (b_t, a_t) \not \in \Theta (t, X) \} | \mathcal{F}_\tau \big] (\omega) dt P(d \omega) 
	\\&= \int  E^P \Big[ \int_{\tau (\omega)}^\infty \1 \{ (b_t, a_t) \not \in \Theta (t, X) \} dt \big| \mathcal{F}_\tau\Big] (\omega) P(d\omega)
	\\&= \int  E^P \Big[ \int_{\tau}^\infty \1 \{ (b_t, a_t) \not \in \Theta (t, X) \} dt \big| \mathcal{F}_\tau\Big] (\omega) P(d\omega)
	\\&= E^P \Big[ \int_{\tau}^\infty \1 \{ (b_t, a_t) \not \in \Theta (t, X) \} dt \Big] = 0.
	\end{align*}
	This observation completes the proof of \(P\)-a.s. \(P(\cdot | \mathcal{F}_\tau) \in \cA(\tau, X)\). 
\end{proof}


\subsection{Stability under Pasting}
In this section we check stability under pasting. The proof is split into several steps. Throughout this section, we fix \((t^*, \omega^*) \in \of 0, \infty\of\), a probability measure \(P\) on \((\Omega, \mathcal{F})\) such that \(P (X = \omega^* \text{ on } [0, t^*]) = 1\), a stopping time \(\tau\) with \(t^* \leq \tau < \infty\), and an \(\mathcal{F}_\tau\)-measurable map \(\Omega \ni \omega \mapsto Q_\omega \in \mathfrak{P}(\Omega)\) such that 
\(Q_\omega ( \omega = X \text{ on } [0, \tau(\omega)]) = 1\) for \(P\)-a.a. \(\omega \in \Omega\). For simplicity, we further set \(\overline{P} := P \otimes_\tau Q\). 
The following corollary is an immediate consequence of Lemma \ref{lem: eq concatenation}.
\begin{corollary} \label{coro: meas identity}
	\(\overline{P} = E^P \big[ Q (\cdot ) \big]\).
\end{corollary}
\begin{lemma} \label{lem: identiy up to tau}
	\(P = \overline{P}\) on \(\mathcal{F}_\tau\) and  \(\overline{P}\)-a.s. \(\overline{P} (\cdot | \mathcal{F}_\tau) = Q\).
\end{lemma}
\begin{proof}
	Take \(A \in \mathcal{F}_\tau\). By hypothesis, 
	\(
	Q_\omega(\omega \otimes_{\tau(\omega)} X = \omega \textup{ on } [0, \tau(\omega)]) = 1\) for \(P\)-a.a. \(\omega \in \Omega\).
	Hence, Galmarino's test yields that 
	\[
	\overline{P} (A) = \int Q_\omega ( \omega \in A ) P(d \omega) = P(A).
	\]
	This is the first claim. For the second claim, take \(G \in \mathcal{F}\). Then, again by Galmarino's test, Lemma~\ref{lem: eq concatenation} and the first part, 
	\[
	\overline{P} (A \cap G) = \int Q_\omega ( \omega \in A, \omega \otimes_{\tau(\omega)} X \in G) P(d \omega) = \int_A Q_\omega(G) P(d \omega) = E^{\overline{P}} \big[ \1_A Q (G) \big].
	\]
	This yields \(\overline{P}\)-a.s. \(\overline{P}(G | \mathcal{F}_\tau) = Q(G)\). Since \(\mathcal{F}_\tau\) is countably generated, a monotone class argument completes the proof.
\end{proof}

The following two lemmata should be compared to \cite[Proposition 4.2]{neufeld2017nonlinear}.
\begin{lemma} 
	\label{lem_pasting_s}
	Let $P \in \fPs(t^*)$. 
	If \(P\)-a.s. $Q \in \fPs(\tau)$, then $\overline{P} \in \fPs(t^*)$.
\end{lemma} 

\begin{proof}
	Let \(T > t^*\), and take a sequence \((H^n)_{n \in \mathbb{Z}_+}\) of simple predictable processes on \( \of t^*, T \gs \) such that \(H^n \to H^0\) uniformly in time and \(\omega\). Then, by Lemma \ref{lem: identiy up to tau}, we get, for every \(i = 1, 2, \dots, \d\),
	\begin{align*}
	E^{\overline{P}} \Big[ \Big| \int_{t^*}^T(H^n_s &- H_s) d X^{(i)}_s\Big| \wedge 1 \Big] 
	\\&\leq E^{P} \Big[ \Big| \int_{t^*}^{\tau \wedge T} (H^n_s - H_s) d X^{(i)}_s \Big| \wedge 1 \Big] \\&\hspace{1.5cm} + \int E^{Q_\omega} \Big[ \Big|\int_{\tau (\omega) \wedge T}^T (H^n_s - H_s) (\omega \otimes_{\tau (\omega)} X) d X^{(i)}_s\Big| \wedge 1 \Big] P(d \omega).
	\end{align*}
	The first term converges to zero since \(P \in \fPs (t^*)\) and thanks to the Bichteler--Dellacherie (BD) Theorem (\cite[Theorem III.43]{protter}). The second term converges to zero by dominated convergence, the assumption that \(P\)-a.s. \(Q \in \fPs (\tau)\) and, by virtue of Lemma \ref{lem: predi pasting}, again the BD Theorem. Consequently, invoking the BD Theorem a third time yields that \(\overline{P} \in \fPs (t^*)\).
\end{proof}

\begin{lemma}
	\label{lem_pasting_as}
	Let $P \in \fPas(t^*)$.
	If \(P\)-a.s. $Q \in \fPas(\tau)$, then $\overline{P} \in \fPas(t^*)$.
\end{lemma}

\begin{proof}
	By Lemma \ref{lem_pasting_s}, we already know that $\overline{P} \in \fPs(t^*)$.
	Denote by $B_{\cdot + t^*}$ the first characteristic of $X_{\cdot + t^*}$ under $\overline{P}$, and let
	$$ B_{\cdot + t^*} = \int_{t^*}^\cdot \phi_s ds + \psi $$
	be the Lebesgue decomposition of (the paths of) $B$.
	Since \(P \in \fPas (t^*)\) and $\overline{P} = P$ on $\cF_\tau$ by Lemma~\ref{lem: identiy up to tau}, 
	simply by the definition of the first characteristic, we get that \(B \ll \llambda\) on \(\of 0, \tau\gs\).
	Hence, by virtue of Corollary~\ref{coro: meas identity}, it suffices to show that
	$$ D := \Big\{B_{\cdot + \tau} - B_{\tau} \neq 
	\int_{\tau}^{\cdot + \tau} \phi_s ds\Big
	\}
	$$ is a \(Q_\omega\)-null set for \(\overline{P}\)-a.a. \(\omega \in \Omega\).
	Due to Lemmata \ref{lem: characteristics under conditioning} and \ref{lem: identiy up to tau}, for \(\overline{P}\)-a.a. \(\omega \in \Omega\), we have $Q_\omega \in \fPs (\tau (\omega))$ and the \(Q_\omega\)-characteristics of the shifted process \(X_{\cdot + \tau(\omega)}\) are given by \((B_{\cdot + \tau(\omega)} - B_{\tau(\omega)}, C_{\cdot + \tau(\omega)} - C_{\tau (\omega)}) (\omega \otimes_{\tau (\omega)} X)\).
	As \(P\)-a.s. \(Q \in \fPas (\tau)\), Lemma~\ref{lem: identiy up to tau} and the
	uniqueness of the Lebesgue decomposition yield that for \(\overline{P}\)-a.a. \(\omega \in \Omega\)
	$$ Q_\omega \Big( B_{\cdot + \tau (\omega)} - B_{\tau (\omega)} \neq 
	\int_{\tau (\omega)}^{\cdot + \tau (\omega)} \phi_s ds
	\Big) = 0.
	$$
	As 
	\(Q_\omega (\tau = \tau(\omega)) = 1\) for \(\overline{P}\)-a.a. \(\omega \in \Omega\),
	\(\overline{P}\)-a.s. \(Q(D) = 0\). 
	One may proceed similarly for the other characteristic.	
\end{proof}

\begin{lemma} \label{lem: stab pasting}
	Let $P \in \cA(t^*,\omega^*)$.
	If $Q_\omega \in \cA(\tau(\omega), \omega)$ for $P$-a.e. $\omega \in \Omega$, then
	$\overline{P} \in \cA(t^*,\omega^*)$.
\end{lemma}

\begin{proof}
	Lemma \ref{lem_pasting_as} implies $\overline{P} \in \fPas(t^*)$. Clearly, \(\{X = \omega^* \text{ on } [0, t^*]\} \in \mathcal{F}_\tau\) as \(\tau \geq t^*\). Thus, Lemma \ref{lem: identiy up to tau} yields \(\overline{P} ( X = \omega^* \text{ on } [0, t^*] ) = P ( X = \omega^* \text{ on } [0, t^*] ) = 1.\)
	Denote by $(b_{\cdot + t^*},a_{\cdot + t^*})$ the Lebesgue densities of the characteristics
	of $X_{\cdot + t^*}$ under $\overline{P}$.
	Since $\overline{P} = P$ on $\cF_\tau$, it suffices to show that
	$$ R := \big\{ (t,\omega) \in \of\tau, \infty \of \hspace{0.05cm} \colon (b_t(\omega), a_t(\omega)) \notin \Theta(t,\omega)   \big\} $$
	is a $(\llambda \otimes \overline{P})$-null set.
	By virtue of Lemmata \ref{lem: characteristics under conditioning} and \ref{lem: identiy up to tau}, the assumption that \(Q_\omega \in \cA (\tau(\omega), \omega)\) for \(P\)-a.a. \(\omega \in \Omega\) yields that \(P\)-a.s.
	$$ (\llambda \otimes Q)\big((t,\omega') \in \of\tau,  \infty \of \hspace{0.05cm} \colon (b_t(\omega'),a_t(\omega')) \notin \Theta(t,\omega')\big) = 0.
	$$
	Finally, as \(Q_\omega (\tau = \tau (\omega)) = 1\) for \(P\)-a.a. \(\omega \in \Omega\), Corollary \ref{coro: meas identity} and Fubini's theorem yield that
	\begin{align*}
	(\llambda \otimes \overline{P}) (R) &= E^{\overline{P}} \Big[ \int_{\tau}^{\infty} \1 \{ (b_t, a_t) \not \in \Theta (t, X) \} d t \Big] 
	\\&= E^P \Big[ E^Q \Big[ \int_{\tau}^{\infty} \1 \{ (b_t, a_t) \not \in \Theta (t, X) \} d t \Big] \Big]
	\\&= E^P \Big[ \int_{\tau}^{\infty} Q\big( (b_t, a_t) \not \in \Theta (t, X) \big) dt \Big] = 0.
	\end{align*}
	This completes the proof.
\end{proof}
\subsection{Proof of Theorem \ref{theo: DPP}} \label{sec: pf DPD}
Lemma \ref{lem: correspondence 2}, Corollary \ref{coro: stab cond} and Lemma \ref{lem: stab pasting} show that the prerequisites of \cite[Theorem 2.1]{ElKa15} are fulfilled and this theorem then directly implies the DPP. \qed


\section{A nonlinear  Kolmogorov equation: Proof of Theorem \ref{theo: viscosity}} \label{sec: pf vis}

In this section we prove that the value function is a weak sense viscosity solution to the PPDE \eqref{eq: PIDE} and we discuss some of its regularity properties.
By its very definition, regularity of the value function is closely linked to the regularity of the correspondence \( (t,\omega) \mapsto \cA(t,\omega) \).
Due to the appearance of \(\fPas (t)\) and \((dB^P_{\cdot + t}/d \llambda, dC^P_{\cdot + t} /d \llambda)\) in the set \( \cA(t,\omega) \), certain regularity properties of \((t, \omega) \mapsto \cA(t, \omega)\) seem at first glance to be difficult to verify. To get a more convenient condition, in Section \ref{subsec: preparations}, we show that 
	\[
	v(t,\omega) = \sup_{P \in \cA(t, \omega)} E^{P}\big[\psi\big]  = \sup_{P \in \mathcal{R}(t, \omega)} E^{P}\big[\psi(\omega \ \widetilde{\otimes}_t \ X) \big].
	\]
	This reformulation of the value function \(v\) explains that it suffices to investigate the correspondence \((t,\omega) \mapsto \mathcal{R}(t,\omega)\) from \eqref{eq: def R}. Thereby, we shift the main \((t, \omega)\) dependence to the explicitly given correspondence~\(\Theta\), as the remaining parts from \(\cR (t, \omega)\) only depend on \(\fPas = \fPas (0)\) and \((d B^P /d \llambda, d C^P / d \llambda)\).
	In the Sections \ref{sec: pf subsolution} and \ref{sec: pf supersolution}, we show that \(v\) is a viscosity sub- and supersolution to \eqref{eq: PIDE}. The ideas of proof are based on applications of Berge's maximum theorem, Skorokhod's existence theorem for stochastic differential equations and Lebesgue's differentiation theorem. In contrast to the proofs from \cite{fadina2019affine, neufeld2017nonlinear} for the viscosity subsolution property in L\'evy and continuous affine frameworks respectively, we do not work with explicit moment estimates. This allows us to extend the class of test functions to \(C^{1, 2}\) in comparison to the class \(C^{2, 3}\) as used in \cite{fadina2019affine, neufeld2017nonlinear}. Further, this extension also enables us to apply the uniqueness result from \cite{zhou} to deduce Theorem \ref{thm: uniqueness}.
	Given the above mentioned results, the proof of Theorem \ref{theo: viscosity} is finalized in Section~\ref{sec: pf viscosity}.

\subsection{Some preparations} \label{subsec: preparations}
To execute the program outlined above, we require more notation.
First, for \(t \in \mathbb{R}_+\), we define another shift operator \(\gamma_t \colon \Omega \to \Omega\) by \(\gamma_t (\omega) := \omega((\cdot - t)^+)\) for all \(\omega \in \Omega\). For \( P \in \mathfrak{P}(\Omega) \), we denote
\(P^t := P  \circ \gamma_t^{-1} \).
Moreover, for \( (t, \omega) \in \of 0, \infty \of \), we write \( \xi_{t, \omega}(\omega') := \omega \otimes_t \omega' \)
and define
\begin{align*}
\mathcal{Q}(t,\omega) := \Big\{ P \in \fPas(t) \colon P &\circ X_t^{-1} = \delta_{\omega (t)},  \\&(\llambda \otimes P)\text{-a.e. } (dB^P_{\cdot + t} /d\llambda, dC^P_{\cdot + t}/d\llambda) \in \Theta (\cdot + t, \omega \otimes_t X)  \Big\}.
\end{align*}
Further, recall the definition of \(\mathcal{R}\) as given in \eqref{eq: def R}.

\begin{lemma} \label{lem: connection Q and R}
	The equality
	\( \mathcal{R}(t, \omega) = \{ P_t \colon P \in \mathcal{Q}(t,\omega) \} \)
	holds for every \( (t,\omega) \in \of 0, \infty \of \).
\end{lemma}
\begin{proof}
	Lemma \ref{lem: set identity} shows the inclusion
	\( \{ P_t \colon P \in \mathcal{Q}(t,\omega) \} \subset \mathcal{R}(t, \omega) \), as \( P_t \circ X_0^{-1} = P \circ X_t^{-1} = \delta_{\omega(t)} \) for every \(P \in \cQ(t, \omega)\).
	Conversely, given \( P \in \mathcal{R}(t, \omega) \), the measure
	\(P^t \) is contained in \( \mathcal{Q}(t,\omega) \).
	To see this, note first that \(P^t \circ X_t^{-1} = P \circ X_0^{-1} = \delta_{\omega (t)}\). 
	Second, \( P^t \in \fPas(t) \) since  \( (P^t)_t = P \in \fPas \) as \(\th_t \circ \gamma_t = \on{id} \).
	Finally, Lemma \ref{lem: jacod restatements}
	together with \( (\omega \otimes_t X ) \circ \gamma_t = \omega \ \widetilde{\otimes}_t \ X \) implies that \( (\llambda \otimes P^t)\text{-a.e.} \)
	\(
	(dB^{P^t}_{\cdot + t} /d\llambda, dC^{P^t}_{\cdot + t}/d\llambda) \in \Theta (\cdot + t, \omega \otimes_t X).
	\) 
\end{proof}

\begin{lemma} \label{lem: connection Q and C}
	The equality
	\( \mathcal{C}(t, \omega) = \{  P \circ 
	\xi_{t,\omega}^{-1} \colon P \in \mathcal{Q}(t,\omega) \} \)
	holds for every \( (t,\omega) \in \of 0, \infty \of \).    
\end{lemma}
\begin{proof}
	For any measure \( P \in \mathfrak{P}(\Omega) \), and every \( (t, \omega) \in \of 0, \infty \of \), we have
	\( P( X = \omega \text{ on } [0,t]) = 1 \) if and only if 
	\( P = P \circ \xi_{t,\omega}^{-1} \).  
	Note that the canonical process \(X \) is a semimartingale after time \(t\) if 
	and only if the process \( \omega \otimes_t X \) is a semimartingale after time \(t \). Thus, Lemma~\ref{lem: jacod restatements} implies, together with the identity \( \xi_{t,\omega} \circ \xi_{t,\omega} = \xi_{t,\omega} \),
	that \( P \in \mathcal{Q}(t, \omega) \) if and only if \( P \circ \xi_{t,\omega}^{-1} \in \mathcal{Q}(t, \omega) \).
	This completes the proof.
\end{proof}

Summarizing the above, we may conclude the following corollary, which provides, for our framework, the connection between the approaches in \cite{ElKa15} and \cite{NVH} for the construction of nonlinear expectations on path spaces.

\begin{corollary} \label{coro: connection ElK NVH}
	For every upper semianalytic function \( \psi \colon \Omega \to \mathbb{R} \), the equality
	\[ \sup_{P \in \cA(t, \omega)} E^{P}\big[\psi\big]  = \sup_{P \in \mathcal{R}(t, \omega)} E^{P}\big[\psi(\omega \ \widetilde{\otimes}_t \ X) \big] \]
	holds for every \( (t,\omega) \in \of 0, \infty \of \).        
\end{corollary}
\begin{proof}
	Using Lemma \ref{lem: connection Q and C} for the first, the identity \( \xi_{t, \omega} = (\omega \ \widetilde{\otimes}_t \ X) \circ \th_t \) for the second, and Lemma~\ref{lem: connection Q and R} for the final equality, we obtain
	\begin{align*} \sup_{P \in \cA(t, \omega)} E^{P}\big[\psi\big] &= \sup_{P \in \mathcal{Q}(t, \omega)} E^{P}\big[\psi \circ \xi_{t,\omega}\big]\\
	&=\sup_{P \in \mathcal{Q}(t, \omega)} E^{P_t}\big[\psi(\omega \ \widetilde{\otimes}_t \ X) \big], \\&= \sup_{P \in \mathcal{R}(t, \omega)} E^{P}\big[\psi(\omega \ \widetilde{\otimes}_t \ X) \big]. \end{align*}
	The proof is complete.
\end{proof}

For \(N > 0\), we call a set \(G \subset \Omega\) {\em \(N\)-bounded} if
\[
\sup \big\{ \| \omega (t) \| \colon \omega \in G, t \in [0, N] \big\} < \infty.
\]
By the Arzel\`a--Ascoli theorem, any relatively compact set \(G \subset \Omega\) is \(N\)-bounded for every~\(N > 0\). 

\begin{lemma} \label{lem: rel comp}
	Suppose that Condition \ref{cond: LG} holds.
	For any \(N > 0\), any \(N\)-bounded set \(G \subset \Omega\) and any bounded set \(K \subset \bR^\d\), the set 
	\begin{align*}
	\cR^* := \bigcup_{\substack{t \in [0, N]\\ \omega \in G}}\Big\{ P \in \fPas \colon P &\circ X_0^{-1} \in \{\delta_x \colon x \in K\}, \\
	&(\llambda \otimes P)\text{-a.e. } (dB^{P} /d\llambda, dC^{P}/d\llambda) \in \Theta (\cdot + t, \omega \ \widetilde{\otimes}_t\ X)  \Big\}
	\end{align*}
	is relatively compact in \(\mathfrak{P}(\Omega)\). Moreover, for every \(p \geq 1\),
	\begin{align} \label{eq: moment estimate with general p}
	\sup_{P \in \cR^*} E^P \Big[ \sup_{s \in [0, T]} \| X_s \|^{2p} \Big] < \infty.
	\end{align}
\end{lemma}

\begin{proof}
	Thanks to Prohorov's theorem, we need to show that \(\cR^*\) is tight, which we do by an application of Kolmogorov's criterion (\cite[Theorem~21.42]{klenke}), 
	i.e., we show the following two conditions:
	\begin{enumerate}
		\item[(a)]
		the family \(\{P \circ X_0^{-1} \colon P \in \cR^*\}\) is tight;
		\item[(b)] 
		for each \(T > 0\) there are numbers \(C, \alpha, \beta > 0\) such that, for all \(s, t \in [0, T]\), we have
		\[
		\sup_{P \in \cR ^*} E^P \big[ \|X_s - X_t\|^\alpha \big] \leq C|s - t|^{\beta + 1}.
		\]
	\end{enumerate}
	Part (a) follows easily from the fact that \(\sup_{P \in \mathcal{R}^*} E^P[ \|X_0\|] < \infty\), which uses the boundedness of~\(K\).
	We now show (b). Fix \(T > 0, p \geq 1\) and set 
	\[
	T_n := \inf \{t \geq 0 \colon \|X_t\| \geq n\}, \quad n \in \mathbb{N}.
	\]
	For a moment, take \(P \in \cR^*\) and let \((b^P, a^P)\) be the Lebesgue densities of the semimartingale characteristics of \(X\) under \(P\). Using the Burkholder--Davis--Gundy inequality, H\"older's inequality and the linear growth assumption, i.e., Condition \ref{cond: LG}, for all \(t \in [0, T]\), we obtain that 
	\begin{align*}
	E^P \Big[\sup_{s \in [0, t \wedge T_n]} &\|X_s\|^{2p} \Big] 
 \\&\leq C \Big( 1 + T^{2p - 1} \int_0^t E^P \big[ \|b^P_{s \wedge T_n}\|^{2p} \big] ds + T^{p-1} \int_0^t E^P \big[ \on{tr} \big[a^P_{s \wedge T_n}\big]^{p} \big] ds \Big)
	\\&\leq C \Big( 1 + \int_0^t E^P \Big[ \sup_{r \in [0, s \wedge T_n]} \|X_r\|^{2p} \Big] ds \Big).
	\end{align*}
	Here, the constant might depend on \(\d, p, T, N, G\) and \(K\) but it is independent of \(n, t\) and \(P\).
	Finally, using Gronwall's lemma and Fatou's lemma, we conclude that \eqref{eq: moment estimate with general p} holds. It remains to finish the proof for (b).
	Take \(0 \leq s \leq t \leq T\).
	Using again the Burkholder--Davis--Gundy inequality, the linear growth assumption and \eqref{eq: moment estimate with general p}, we get, for every \(P \in \cR^*\), that
	\begin{equation*} \begin{split}
	E^P \big[ \| X_t - X_s \|^4 \big] &\leq C \Big( (t - s)^3 \int_{s}^t E^P \big[\|b^P_r\|^4 \big] dr + (t - s) \int_{s}^{t} E^P \big[ \on{tr} \big[a^P_r\big]^2 \big] dr \Big)
	\\&\leq C \big((t - s)^4 + (t - s)^2\big)   \Big( 1 + E^P \Big[ \sup_{r \in [0, T]} \|X_r\|^4 \Big] \Big)
	\\&\leq C \big((t - s)^4 + (t - s)^2\big),
	\end{split} \end{equation*}
	where the constant \(C > 0\) is independent of \(s,t\) and \(P\).
	Consequently, we conclude that (b) holds with \(\alpha = 4\) and \(\beta = 1\). The proof is complete.
\end{proof}

We collect more technical observations.
Recall that \(D(\bR_+; \bR^\d)\) denotes the space of \cadlag functions from \(\bR_+\) into \(\bR^\d\). In the following we endow \(D (\bR_+; \bR^\d)\) with the Skorokhod \(J_1\) topology, see \cite[Section~VI.1]{JS} or \cite[Section~XV.1]{HWY} for details on this topology. For \(\omega, \alpha \in \Omega\) and \(t \in \bR_+\), we define the concatenation
\[
\omega \hspace{0.05cm}\widehat{\otimes}_t \hspace{0.05cm} \alpha := \omega \1_{[0, t)} + \alpha \1_{[t, \infty)}.
\]
Notice that \(\omega \hspace{0.05cm}\widehat{\otimes}_t \hspace{0.05cm} \alpha \in D(\bR_+; \bR^\d)\). Furthermore, recall that we endow \(\Omega\) with the local uniform topology.
\begin{lemma} \label{lem: concatenation continuous}
	The map 
	\(
	\Omega \times \Omega \times (0, \infty) \ni (\omega, \alpha, t) \mapsto \omega \hspace{0.05cm} \widehat{\otimes}_t \hspace{0.05cm} \alpha \in D(\bR_+; \bR^\d)
	\)
	is continuous.
\end{lemma}
\begin{proof}
	Take two sequences \((\omega^n)_{n \in \mathbb{N}}, (\alpha^n)_{n \in \mathbb{N}} \subset \Omega\) and two functions \(\omega, \alpha \in \Omega\) such that \(\omega^n \to \omega\) and \(\alpha^n \to \alpha\). Furthermore, take a sequence \((t^n)_{n \in \mathbb{N}} \subset (0, \infty)\) and a time \(t > 0\) such that \(t^n \to t\). We have to prove that 
	\(
	\omega^n \hspace{0.05cm}\widehat{\otimes}_{t^n} \hspace{0.05cm} \alpha^n \to \omega \hspace{0.05cm}\widehat{\otimes}_t \hspace{0.05cm} \alpha
	\)
	in the Skorokhod \(J_1\) topology. For \(n \in \mathbb{N}\), define 
	\[
	\lambda^n (s) := \tfrac{t^n s}{t} \1_{\{s < t\}} + (s + t^n - t) \1_{\{s \geq t\}}, \quad s \in \bR_+.
	\]
	Evidently, \(\lambda^n\) is a strictly increasing continuous function such that \(\lambda (0) = 0\) and \(\lambda^n (s) \to \infty\) as \(s \to \infty\). Furthermore, for every \(N > 0\), we have 
	\[
	\sup_{s \in [0, N]}| \lambda^n (s) - s | \leq N | \tfrac{t^n}{t} - 1 | + | t^n - t | \to 0
	\]
	as \(n \to \infty\). Notice that 
	\[
	(\omega^n \hspace{0.05cm}\widehat{\otimes}_{t^n} \hspace{0.05cm} \alpha^n) (\lambda_n (s)) = \omega^n ( \tfrac{t^ns}{t}) \1_{\{s < t\}} + \alpha^n (s + t^n - t) \1_{\{s \geq t\}}, \quad s \in \bR_+.
	\]
	Thus, for every \(N > 0\) and some \(T = T_N > 0\) large enough, we have 
	\begin{align*}
	\sup_{s \in [0, N]} \|  (\omega^n  \hspace{0.05cm}&\widehat{\otimes}_{t^n} \hspace{0.05cm} \alpha^n)(\lambda^n (s)) - (\omega \hspace{0.05cm}\widehat{\otimes}_t \hspace{0.05cm} \alpha) (s) \| 
	\\&\leq \sup \{ \| \omega^n (s) - \omega (s) \| \colon s \in [0, N] \} + \sup \{ \| \alpha^n (s) - \alpha (s) \| \colon s \in [0, N]\}
	\\&\hspace{2.9cm} + \sup \{ \| \omega^n (z) - \omega^n (v) \| \colon z, v \in [0, T], |z - v| \leq N | \tfrac{t^n}{t} - 1|\}
	\\&\hspace{2.9cm} + \sup \{ \| \alpha^n (z) - \alpha^n (v) \| \colon z, v \in [0, T], |z - v| \leq  | t^n - t |\}.
	\end{align*}
	Thanks to the Arzel\`a--Ascoli theorem, we conclude that all terms on the r.h.s. converge to zero as \(n \to \infty\), which implies that 
	\[
	\sup_{s \in [0, N]} \| (\omega^n \hspace{0.05cm}\widehat{\otimes}_{t^n} \hspace{0.05cm} \alpha^n)(\lambda^n (s)) - (\omega \hspace{0.05cm}\widehat{\otimes}_t \hspace{0.05cm} \alpha) (s) \| \to 0 
	\]
	as \(n \to \infty\). Consequently, by virtue of \cite[Theorem 15.10]{HWY},  \(
	\omega^n \hspace{0.05cm}\widehat{\otimes}_{t^n} \hspace{0.05cm} \alpha^n \to \omega \hspace{0.05cm}\widehat{\otimes}_t \hspace{0.05cm} \alpha
	\)
	in the Skorokhod \(J_1\) topology. The proof is complete.
\end{proof}

In the following lemma we take care of the case \(t = 0\).
\begin{lemma} \label{lem: concatenation cont in zero}
	Let \((\omega^n)_{n \in\mathbb{Z}_+}, (\alpha^n)_{n \in \mathbb{Z}_+} \subset \Omega\) and \((t^n)_{n \in \mathbb{N}} \subset \bR_+\) such that \(\omega^n \to \omega^0, \alpha^n \to \alpha^0\) and \(t^n \to 0\). Furthermore, suppose that \(\omega^0 (0) = \alpha^0 (0)\). Then, 
	\[
	\omega^n \ \widehat{\otimes}_{t^n}\ \alpha^n \to \alpha^0 
	\]
	locally uniformly. 
\end{lemma}
\begin{proof}
	For every \(T > 0\), we estimate 
	\begin{align*}
	\sup_{s \in [0, T]} \| (\omega^n \ &\widehat{\otimes}_{t^n}\ \alpha^n) (s) - \alpha^0 (s) \| 
	\\&\leq \sup_{s \in [0, T]} \| \omega^n (s) - \alpha^0 (s)\| \1_{\{s < t^n\}} + \sup_{s \in [0, T]} \| \alpha^n (s) - \alpha^0 (s) \| \1_{\{s \geq t^n\}}
	\\&\leq \sup_{s \in [0, T]} \| \omega^n (s) - \omega^n (0) + \omega^n (0) - \omega^0(0) + \alpha^0 (0) - \alpha^0 (s) \| \1_{\{s < t^n\}} 
	\\&\hspace{2cm}+ \sup_{s \in [0, T]} \| \alpha^n (s) - \alpha^0 (s) \|
	\\&\leq \sup \{ \| \omega^n (s) - \omega^n (r)\| \colon s, r \in [0, T], |s - r| \leq t^n \} + \|\omega^n (0) - \omega^0(0)\|
	\\&\hspace{2cm}+ \sup \{ \| \alpha^0 (s) - \alpha^0 (r) \| \colon s, r \in [0, T], |s - r| \leq t^n \} 
	\\&\hspace{2cm}+ \sup_{s \in [0, T]} \| \alpha^n (s) - \alpha (s) \|.
	\end{align*}
	By the Arzel\`a--Ascoli theorem, the r.h.s. tends to zero as \(n \to \infty\). This completes the proof. 
\end{proof}

\begin{corollary} \label{coro: real concatenation continuous}
	The maps 
	\[
	\Omega \times \Omega \times \bR_+ \ni (\omega, \alpha, t) \mapsto \omega \otimes_t \alpha \in \Omega,\qquad  \Omega \times \Omega \times \bR_+ \ni (\omega, \alpha, t) \mapsto  \omega \hspace{0.05cm} \widetilde{\otimes}_t \hspace{0.05cm} \alpha \in \Omega
	\] are continuous.
\end{corollary}
\begin{proof}
	We only prove that \((\omega, \alpha, t) \mapsto \omega \hspace{0.05cm} \widetilde{\otimes}_t \hspace{0.05cm} \alpha\) is continuous. The continuity of \((\omega, \alpha, t) \mapsto \omega \otimes_t \alpha\) follows the same way.
	Let \((\omega^n)_{n \in \mathbb{Z}_+}, (\alpha^n)_{n \in \mathbb{Z}_+} \subset \Omega\) and \((t^n)_{n \in \mathbb{Z}_+}\subset \bR_+\) be sequences such that \(\omega^n \to \omega^0, \alpha^n \to \alpha^0\) and \(t^n \to t^0\). Furthermore, define 
	\[
	\gamma^n := \omega^n (t^n) + \alpha^n ( (\cdot - t^n)^+ ) - \alpha^n (0), \qquad n \in \mathbb{Z}_+.
	\]
	For every \(T > \sup_{n \in \mathbb{Z}_+} t^n\), we estimate
	\begin{align*}
	\sup_{s \in [0, T]} \| &\gamma^n (s) - \gamma^0 (s) \| 
	\\&\leq \|\omega^n (t^n) - \omega^0 (t^0)\| + \sup_{s \in [0, T]} \| \alpha^n ( (s - t^n)^+ ) - \alpha^0 ( (s - t^0)^+ ) \| + \|\alpha^n (0) - \alpha^0 (0) \|
	\\&\leq \sup \{ \|\omega^n (s) - \omega^n (r)\| \colon s, r \in [0, T], |s - r| \leq |t^n - t^0|\} 
	+ \|\omega^n (t^0) - \omega^0 (t^0)\|
	\\&\hspace{3.325cm}+ \sup \{\|\alpha^n (s) - \alpha^n (r)\| \colon s, r \in [0, T], |s - r| \leq |t^n - t^0|\} 
	\\&\hspace{3.325cm}+ \sup_{s \in [0, T]} \|\alpha^n (s) - \alpha^0 (s)\|
	+ \|\alpha^n (0) - \alpha^0 (0)\|.
	\end{align*}
	By the Arzel\`a--Ascoli theorem, the r.h.s. converges to zero as \(n \to \infty\). Hence, we conclude that \(\gamma^n \to \gamma^0\). If \(t^0 > 0\), then we can assume that \(t^n > 0\) for all \(n \in \mathbb{N}\) and Lemma \ref{lem: concatenation continuous} yields that 
	\begin{align} \label{eq: contin convergence coro main step}
	\omega^n \ \widetilde{\otimes}_{t^n} \ \alpha^n = \omega^n \ \widehat{\otimes}_{t^n} \ \gamma^n \to \omega^0\ \widehat{\otimes}_{t^0} \ \gamma^0 = \omega^0 \ \widetilde{\otimes}_{t^0} \ \alpha^0
	\end{align}
	in the Skorokhod \(J_1\) topology. As \(\omega^0 \ \widetilde{\otimes}_{t^0} \ \alpha^0 \in \Omega\), we get from \cite[Proposition~VI.1.17]{JS} that the convergence is even locally uniformly. Finally, in case \(t^0 = 0\), Lemma \ref{lem: concatenation cont in zero} yields that \eqref{eq: contin convergence coro main step} holds locally uniformly, because \(\gamma^0 (0) = \omega^0 (0)\).
	We conclude the continuity of \((\omega, \alpha, t) \mapsto \omega \ \widetilde{\otimes}_t \ \alpha\).
\end{proof}

For \(\alpha, \omega \in D(\bR_+; \bR^\d)\) and \(t \in \bR_+\), we define 
\[
\omega \ \overline{\otimes}_t \ \alpha := \omega \1_{[0, t)} + \alpha (\cdot - t) \1_{[t, \infty)} \in D(\bR_+; \bR^\d).
\]

\begin{lemma} \label{lem: derivatives}
	Let \((t,\omega) \in [0, T ) \times D(\bR_+; \bR^\d)\) and \( \phi \in C^{1,2}([t, T] \times D(\bR_+; \bR^\d) ; \bR)\). Then, 
	\[
	((s,\alpha) \mapsto F(s,\alpha) := \phi(s+t, \omega \ \overline{\otimes}_t \ \alpha)) \in C^{1,2}([0, T - t] \times D(\bR_+; \bR^\d) ; \bR),
	\]
	with
	\begin{equation}\label{eq: formuals derivatives}
	\begin{split}
	\p F(s,\alpha) & = \p\phi(s+t, \omega \ \overline{\otimes}_t \ \alpha), \\
	\nabla F(s,\alpha) & = \nabla\phi(s+t, \omega \ \overline{\otimes}_t \ \alpha),\\
	\nabla^2 F(s,\alpha) & = \nabla^2\phi(s+t, \omega \ \overline{\otimes}_t \ \alpha).
	\end{split}
	\end{equation}
	for all \( (s,\alpha) \in [0, T - t] \times D(\bR_+; \bR^\d) \).
\end{lemma}
\begin{proof}
	For \(\alpha \in D(\bR_+; \bR^\d)\), \(0 \leq s < T-t\) and \(h > 0\) small enough, a short computation shows that
	\begin{align*}
	F(h + s, \alpha( \cdot \wedge s)) &- F(s, \alpha(\cdot \wedge s)) 
	\\&= \phi(h + s + t , \omega \ \overline{\otimes}_t \ \alpha(\cdot \wedge s)) - \phi(s + t, \omega \ \overline{\otimes}_t \ \alpha(\cdot \wedge s)) 
	\\&= \phi(h + s + t , (\omega \ \overline{\otimes}_t \ \alpha)(\cdot \wedge (s + t))) - \phi(s + t, (\omega \ \overline{\otimes}_t \ \alpha)(\cdot \wedge (s + t))),
	\end{align*}
	and we conclude that \(\p F(s,\alpha) = \p\phi(s+t, \omega \ \overline{\otimes}_t \ \alpha)\).
	Furthermore, we get, for every \(i = 1, \dots, \d\),
	\begin{align*}
	F(s , \alpha + h e_i\1_{[s, T-t]}) - F(s, \alpha) 
	&= \phi(s + t, \omega \ \overline{\otimes}_t \ (\alpha +h e_i\1_{[s, T-t]})) - \phi(s+t, \omega \ \overline{\otimes}_t \ \alpha)
	\\&= \phi(s + t, (\omega \ \overline{\otimes}_t \ \alpha) +h e_i\1_{[s+t, T]}) - \phi(s+t, \omega \ \overline{\otimes}_t \ \alpha),
	\end{align*}
	which implies \(\nabla F(s,\alpha)  = \nabla \phi(s+t, \omega \ \overline{\otimes}_t \ \alpha)\). 
	In the same way, we derive the formula for \(\nabla^2 F\). 
	Similar to \eqref{eq: d metric}, for every \((u, \zeta), (s, \alpha) \in [0, T] \times D (\bR_+; \bR^\d)\), we define 
	\[
	d_T ( (u, \zeta), (s, \alpha) ) := |u - s| + \sup_{r \in [0, T]} \| \zeta (r \wedge u) - \alpha (r \wedge s) \|.
	\]
	For every \((u, \zeta), (s, \alpha) \in [0, T - t] \times D (\bR_+; \bR^\d)\), we compute that 
	\begin{align*}
	d_T ( (u + t, \omega \ \overline{\otimes}_t \ \zeta)&, (s  + t, \omega \ \overline{\otimes}_t \ \alpha) ) 
 \\&= |u - s| + \sup_{r \in [0, T]} \| \zeta ( (r - t) \wedge u ) - \alpha ( (r - t) \wedge s ) \| \1_{\{r \geq t\}} 
	\\&\leq d_{T - t} ( (u, \zeta), (s, \alpha) ).
	\end{align*}
	Hence, the map
	\begin{equation} \label{eq: pf derivatives}
	(s,\alpha) \mapsto (s + t, \omega \ \overline{\otimes}_t \ \alpha) 
	\end{equation}
	is (Lipschitz) continuous from \(( [0, T - t] \times D (\bR_+; \bR^\d), d_{T - t})\) into \(([0, T] \times D(\bR_+; \bR^\d), d_T)\). We conclude that the claimed continuity properties of \( F, \p F, \nabla F, \nabla^2 F \) 
	follow from continuity of \eqref{eq: pf derivatives} together with the assumed continuity properties of
	\( \phi, \p \phi, \nabla \phi, \nabla^2 \phi\). The proof is complete.
\end{proof}

\subsection{Subsolution Property} \label{sec: pf subsolution}
	In this section we prove that the value function is a weak sense viscosity subsolution to the nonlinear PPDE \eqref{eq: PIDE}. 
	\begin{lemma} \label{lem: subsolution}
		Assume that \(F\) is a compact metrizable space and that the Conditions \ref{cond: LG}, \ref{cond: joint continuity in all} and \ref{cond: convexity} hold. The value function \(v\) is a weak sense viscosity subsolution to \eqref{eq: PIDE}.
	\end{lemma}
	\begin{proof}
		Notice that
		\[
		v (T, \omega) = \sup_{P \in \cA(T, \omega)} E^P \big[ \psi \big] = \psi (\omega), 
		\]
		as \(P( X = \omega \text{ on } [0, T]) = 1\) for all \(P \in \cA(T, \omega)\).
		Thus, we only have to prove the subsolution property. 
		Take \((t,\omega) \in \of 0,T \of \) and \( \phi \in C^{1,2}_{pol}(\of t,T \gs ; \bR )\) satisfying
		\[
		0 = (v - \phi)(t,\omega) = \sup \{ (v - \phi)(s,\omega') \colon (s,\omega') \in \of t,T \gs \}.
		\]
		Fix \(0 < u < T - t\). 
		The DPP (Theorem~\ref{theo: DPP}) and Corollary \ref{coro: connection ElK NVH} show that 
		\begin{equation} \label{eq: visco sub conseq of DPP}
		\begin{split}
		0 = \sup_{P \in \cA (t, \omega)} E^P \big[ v (u + t, X)  - v (t, \omega) \big] &\leq \sup_{P \in \cA(t, \omega)} E^P \big[ \phi (u + t, X) - \phi (t, \omega) \big]
		\\&= \sup_{P \in \cR( t, \omega) } E^P \big[ \phi (u + t, \omega \ \widetilde{\otimes}_t \ X) - \phi (t, \omega) \big].
		\end{split}
		\end{equation}
		We fix \(P \in \cR (t, \omega)\) and denote the Lebesgue densities of the \(P\)-characteristics of \(X\) by \((b^P, a^P)\).
		The pathwise It\^o formula given by \cite[Theorem 2.2]{cosso} yields, together with Lemma~\ref{lem: derivatives}, that 
		\begin{align*}
		\phi (u + t, \omega \ \widetilde{\otimes}_t& \ X) - \phi (t, X) \\&= \int_0^{u} \p \phi (s + t, \omega \ \widetilde{\otimes}_t \ X) d s + \frac{1}{2} \int_0^{u }  \on{tr} \big[  \nabla^2 \phi (s + t, \omega \ \widetilde{\otimes}_t \ X) a^P_s \big] ds 
		\\&\hspace{6cm} + \int_0^{u} \langle \nabla \phi (s + t, \omega \ \widetilde{\otimes}_t \ X) , d X_{s} \rangle.
		\end{align*}
		By virtue of the linear growth condition, the polynomial growth of \(\nabla \phi\) and the moment bound from Lemma \ref{lem: rel comp}, the local martingale part of the stochastic integral above is a true martingale and we get
		\begin{align*}
		E^P \Big[ \int_0^{u} &\langle \nabla \phi (s + t, \omega \ \widetilde{\otimes}_t \ X), d X_s\rangle \Big] = E^P \Big[ \int_0^{u} \langle \nabla \phi (s + t, \omega \ \widetilde{\otimes}_t \ X), b^P_{s} \rangle ds \Big],
		\end{align*}
		which implies that
		\begin{equation} \label{eq: after martingale part vanish}
		\begin{split}
		E^P \big[ \phi (u &+ t, \omega \ \widetilde{\otimes}_t \ X) - \phi (t, \omega) \big] 
		\\&= E^P \Big[ \int_0^{u} \p \phi (s + t, \omega \ \widetilde{\otimes}_t \ X ) d s 
		\\&\hspace{1.25cm} + \int_0^{u} \big(  \langle \nabla \phi (s + t, \omega \ \widetilde{\otimes}_t \ X), b^P_{s} \rangle + \tfrac{1}{2} \on{tr} \big[  \nabla^2 \phi (s + t, \omega \ \widetilde{\otimes}_t \ X) a^P_s \big] \big)  ds \Big].
		\end{split}
		\end{equation}
		For \( (u, \alpha) \in \of 0, T-t \gs\), we set 
		\begin{align*}
		\mathfrak{G} (u, \alpha) := \p \phi (u + t, \omega \ \widetilde{\otimes}_t\ \alpha) \ + &\ \sup \big\{  \langle \nabla \phi (u + t, \omega \ \widetilde{\otimes}_t\ \alpha) , b(f, u + t, \omega \ \widetilde{\otimes}_t\ \alpha) \rangle \\&\ \ + \tfrac{1}{2} \on{tr} \big[ \nabla^2 \phi (u + t, \omega \ \widetilde{\otimes}_t\ \alpha) a (f, u + t, \omega \ \widetilde{\otimes}_t\ \alpha) \big] \colon f \in F \big\}.
		\end{align*}
		From \eqref{eq: visco sub conseq of DPP} and \eqref{eq: after martingale part vanish}, we get that 
		\begin{align}\label{eq: subsolution main ineq}
		0 \leq \frac{1}{u} \int_0^u \sup_{P \in \cR(t, \omega)} E^P \big[ \mathfrak{G} (s, X) \big] ds, \quad u > 0.
		\end{align}
		We investigate the right hand side when \(u \searrow 0\). 
		\begin{lemma} \label{lem: Lebesgue theorem for G}
			\begin{align*}
			\frac{1}{u} \int_0^u \sup_{P \in \cR (t, \omega)} E^P \big[\mathfrak{G} (s, X) \big] ds \to \sup_{P \in \cR (t, \omega)} E^P \big[ \mathfrak{G}(0, X) \big], \quad u \searrow 0.
			\end{align*}
		\end{lemma}
		\begin{proof}
			By Lebesgue's differentiation theorem, it suffices to prove that the function \(s \mapsto \sup_{P \in \cR (t, \omega)} E^P [ \mathfrak{G} (s, X) ]\) is continuous. By the compactness of \(\cR (t, \omega)\), which follows from Theorem \ref{theo: upper hemi}, and Berge's maximum theorem (\cite[Theorem 17.31]{charalambos2013infinite}), continuity of \(s \mapsto \sup_{P \in \cR (t, \omega)} E^P [ \mathfrak{G} (s, X) ]\) is implied by the continuity of
			\[
			\cR (t, \omega) \times [0, T] \ni (P, s) \mapsto E^P\big[ \mathfrak{G} (s, X) \big].
			\]
			Take a sequence \((P^n, s^n)_{n \in \mathbb{Z}_+} \subset \cR (t, \omega) \times [0, T]\) such that \((P^n, s^n) \to (P^0, s^0)\). By Skorokhod's coupling theorem, on some probability space, there are random variables \((X^n)_{n \in \mathbb{Z}_+}\) with laws \((P^n)_{n \in \mathbb{Z}_+}\) such \(X^n \to X^0\) almost surely. Due to the compactness of~\(F\), the continuity assumptions on \(b\) and \(a\) and Corollary \ref{coro: real concatenation continuous}, we deduce from Berge's maximum theorem that \(\mathfrak{G}\) is continuous. Thus, we get a.s. 
			\(
			\mathfrak{G} (s^n, X^n) \to \mathfrak{G} (s^0, X^0).
			\)
			Thanks to the linear growth conditions on \(b\) and \(a\), the polynomial growth assumptions on the derivatives of \(\phi\), and Lemma \ref{lem: rel comp}, we notice that 
			\begin{align*}
			\sup_{n \in \mathbb{N}} E \big[ | \mathfrak{G} (s^n, X^n) |^2 \big] \leq C \Big(1 + \sup_{P \in \cR (t, \omega)} E^P \Big[ \sup_{r \in [0, T]} \|X_r\|^p \Big] \Big) < \infty,
			\end{align*}
			where \(p \geq 4\) is a suitable power (which depends on the polynomial bounds of the derivatives of \(\phi\)).
			This estimate shows that the sequence \((\mathfrak{G} (s^n, X^n))_{n \in \mathbb{N}}\) is uniformly integrable and we conclude that 
			\[
			E \big[ \mathfrak{G} (s^n, X^n) \big] \to E \big[ \mathfrak{G} (s^0, X^0) \big].
			\]
			This is the claimed continuity and therefore, the proof is complete.
		\end{proof}
		Finally, it follows from \eqref{eq: subsolution main ineq} and Lemma \ref{lem: Lebesgue theorem for G} that 
		\begin{align*}
		0 &\leq \sup_{P \in \cR (t, \omega)} E^P \big[ \mathfrak{G}(0, X) \big] 
		\\&= \p \phi (t, \omega) + \sup \big\{ \langle \nabla \phi (t, \omega), b(f, t, \omega) \rangle + \tfrac{1}{2} \on{tr} \big[ \nabla^2 \phi (t, \omega ) a (f, t, \omega) \big] \colon f \in F \big\}.
		\end{align*}
		This completes the proof of the subsolution property.
	\end{proof}

\subsection{Supersolution Property} \label{sec: pf supersolution}
	
	The next lemma is the central tool for the proof of the supersolution property.
	\begin{lemma} \label{lem: worst case}
		Suppose that the Conditions \ref{cond: LG} and \ref{cond: joint continuity in all} hold. 
		For every \((f, t, \omega) \in F \times \of 0, \infty\of\), there exists a probability measure \(\oP\in \cR(t, \omega)\) such that the \(\oP\)-characteristics of \(X\) (for its right-continuous natural filtration) have Lebesgue densities \((b (f, \cdot + t, \omega \ \widetilde{\otimes}_t \ X),\) \(a (f, \cdot + t, \omega \ \widetilde{\otimes}_t \ X))\).
	\end{lemma}
	\begin{proof}
		Thanks to the Conditions \ref{cond: LG} and \ref{cond: joint continuity in all}, and Corollary \ref{coro: real concatenation continuous}, the functions \((s, \alpha) \mapsto b(f, s + t, \omega \ \widetilde{\otimes}_t \ \alpha)\) and \((s, \alpha) \mapsto a (f, s + t, \omega \ \widetilde{\otimes}_t \ \alpha)\) are continuous and of linear growth (in the sense of (4) on p. 258 in \cite{skorokhod}). Hence, the claim follows from \cite[Theorem 4, p. 265]{skorokhod} and Lemma \ref{lem: jacod restatements}.
	\end{proof}

	\begin{lemma} \label{lem: supersolution}
		Suppose that the Conditions \ref{cond: LG}, \ref{cond: joint continuity in all} and \ref{cond: convexity} hold.     The value function \(v\) is a weak sense viscosity supersolution to \eqref{eq: PIDE}.
	\end{lemma}
	\begin{proof}
		Similar to the proof of Lemma \ref{lem: subsolution}, it suffices to prove the supersolution property.
		Take \((t,\omega) \in \of 0, T \of \) and \( \phi \in C^{1,2}_{pol}(\of t,T \gs ; \bR)\) such that
		\[
		0 = (u-\phi)(t, \omega) = \inf \{ (u-\phi)(s,\omega') \colon (s,\omega') \in \of t,T \gs \}.
		\]
		Theorem~\ref{theo: DPP} and Corollary \ref{coro: connection ElK NVH} show that 
		\begin{equation} \label{eq: visco sup conseq of DPP}
		\begin{split}
		0 = \sup_{P \in \cA (t, \omega)} E^P \big[ v (u + t, X)  - v (t, \omega) \big] &\geq \sup_{P \in \cR( t, \omega) } E^P \big[ \phi (u + t, \omega \ \widetilde{\otimes}_t \ X) - \phi (t, \omega) \big].
		\end{split}
		\end{equation}
		Let us fix an \(f \in F\). By Lemma \ref{lem: worst case}, there exists a probability measure \(\oP\in \cR(t, \omega)\) such that the \(\oP\)-characteristics of \(X\) have Lebesgue densities \((b (f, \cdot + t, \omega \ \widetilde{\otimes}_t \ X), a (f, \cdot + t, \omega \ \widetilde{\otimes}_t \ X))\). As in the proof of Lemma \ref{lem: subsolution}, we get that 
		\begin{equation} \label{eq: after martingale part vanish super}
		\begin{split}
		E^{\oP} & \big[ \phi (u + t, \omega \ \widetilde{\otimes}_t \ X) - \phi (t, \omega) \big] 
		\\& \ = E^{\oP} \Big[ \int_0^{u} \big( \p \phi (s + t, \omega \ \widetilde{\otimes}_t \ X ) + \langle \nabla \phi (s + t, \omega \ \widetilde{\otimes}_t \ X), b (f, s + t, \omega \ \widetilde{\otimes}_t \ X) \rangle \big)d s 
		\\&\hspace{4cm} + \int_0^{u}  \tfrac{1}{2} \on{tr} \big[\nabla^2 \phi (s + t, \omega \ \widetilde{\otimes}_t \ X ) a (f, s + t, \omega \ \widetilde{\otimes}_t \ X) \big] ds \Big].
		\end{split}
		\end{equation}
		For \( (u, \alpha) \in \of 0, T-t \gs \), we set 
		\begin{align*}
		\mathfrak{K} (u, \alpha) := \p \phi (u + t, \omega \ \widetilde{\otimes}_t\ \alpha) &+ \langle \nabla \phi (u + t, \omega \ \widetilde{\otimes}_t\ \alpha) ,b(f, u + t, \omega \ \widetilde{\otimes}_t\ \alpha) \rangle \\&+ \tfrac{1}{2} \on{tr} \big[ \nabla^2 \phi (u + t, \omega \ \widetilde{\otimes}_t\ \alpha) a (f, u + t, \omega \ \widetilde{\otimes}_t\ \alpha) \big].
		\end{align*}
		With \eqref{eq: visco sup conseq of DPP} and \eqref{eq: after martingale part vanish super}, we obtain that 
		\[
		0 \geq \frac{1}{u} \int_0^u E^{\oP} \big[ \mathfrak{K} (s, X) \big] ds, \quad u > 0.
		\]
		As \(s \mapsto E^{\oP}[ \mathfrak{K} (s, X)]\) is continuous (see the proof of Lemma~\ref{lem: Lebesgue theorem for G}), we conclude, with \(u \searrow 0\), that 
		\begin{align*}
		0 \geq \frac{1}{u} \int_0^u E^{\oP} \big[ \mathfrak{K} (s, X) \big] ds &\to E^{\oP} \big[ \mathfrak{K} (0, X) \big] 
		\\&= \p \phi (t, \omega) + \langle \nabla \phi (t, \omega), b(f, t, \omega) \rangle + \tfrac{1}{2} \on{tr} \big[\nabla^2 \phi (t, \omega ) a (f, t, \omega) \big].
		\end{align*}
		As \(f \in F\) is arbitrary, taking the \(\sup\) over all \(f \in F\) shows that \(v\) is a weak sense viscosity supersolution. The proof is complete.
	\end{proof}

\subsection{Proof of Theorem \ref{theo: viscosity}} \label{sec: pf viscosity}
It follows from Lemmata \ref{lem: subsolution} and \ref{lem: supersolution} that \(v\) is a weak sense viscosity solution to the PPDE \eqref{eq: PIDE}. We now discuss its regularity properties. Recall that \(\psi\) is bounded and continuous.
Thanks to Corollary \ref{coro: real concatenation continuous}, the map 
\[
\Omega \times \Omega \times [0, T]  \ni (\omega, \alpha, t) \mapsto \psi (\omega \ \widetilde{\otimes}_t \ \alpha)
\]
is continuous. Thus, by \cite[Theorem 8.10.61]{bogachev}, the map 
\begin{align} \label{eq: joint continuity value function proof}
\of0, T\gs \times \mathfrak{P}(\Omega) \ni (t, \omega, P) \mapsto E^{P}\big[\psi(\omega \ \widetilde{\otimes}_t \ X) \big]\end{align}
is continuous, too.
Furthermore, by Theorem \ref{theo: upper hemi}, \(\cR\) is compact-valued.
From Corollary~\ref{coro: connection ElK NVH} we get that
\[
v (t, \omega) = \sup_{P \in \cR(t, \omega)} E^{P}\big[\psi(\omega \ \widetilde{\otimes}_t \ X) \big], \quad (t, \omega) \in \of 0, T\gs.
\]
Hence, under the continuity hypothesis on \(\cR\), the continuity (upper, lower semicontinuity) of \(v\) follows from the continuity of \eqref{eq: joint continuity value function proof} and Berge's maximum theorem (\cite[Theorem~17.31, Lemmata~17.29, 17.30]{charalambos2013infinite}). The proof is complete.
\qed

\section{The Regularity of \(\cR\) and \(v\): Proofs of Theorems \ref{theo: upper hemi}, \ref{theo: lower hemi} and \ref{theo: holder estimate value function}} \label{sec: pf regularity r}

In the following section we prove the Theorems \ref{theo: upper hemi}, \ref{theo: lower hemi} and \ref{theo: holder estimate value function}. By a careful refinement of some key arguments, Theorem \ref{theo: upper hemi} extends \cite[Proposition~3.8]{CN22b} beyond the Markovian case.
	Using the implicit function theorem, strong existence properties of stochastic differential equations with random coefficients and Gronwall arguments, we establish Theorem \ref{theo: lower hemi} that is seemingly the first result regarding lower hemicontinuity of the correspondence \((t, \omega) \mapsto \cR (t, \omega)\) and thus the first result on lower semicontinuity of the value function in a fully path-dependent framework related to nonlinear stochastic processes.
	The proof for Theorem \ref{theo: holder estimate value function} on the uniform continuity of \(v\) with respect to \(d\) combines some observations made in the proof for the lower hemicontinuity of \(\cR\) with an application of the DPP and a Gronwall argument, cf. also the proof of \cite[Lemma 3.6]{tang}.

\subsection{Proof of Theorem \ref{theo: upper hemi}}
To prove Theorem \ref{theo: upper hemi}, we use the sequential characterization of upper hemicontinuity as given by \cite[Theorem 17.20]{charalambos2013infinite}.
Namely, by \cite[Theorem 17.20]{charalambos2013infinite}, \(\cR\) is upper hemicontinuous with compact values if and only if for all sequences \((t^n, \omega^n)_{n \in \mathbb{N}} \subset \of 0, \infty\of\) and \((P^n)_{n \in \mathbb{N}} \subset \mathfrak{P}(\Omega)\) such that \((t^n, \omega^n) \to (t, \omega) \in \of 0, \infty\of\) and \(P^n \in \cR (t^n, \omega^n)\), the sequence \((P^n)_{n \in \mathbb{N}}\) has an accumulation point in \(\cR (t, \omega)\). 
 
From now on, take \((t^n, \omega^n)_{n \in \mathbb{N}} \subset \of 0, \infty\of\) and \((P^n)_{n \in \mathbb{N}} \subset \mathfrak{P}(\Omega)\) such that \((t^n, \omega^n) \to (t, \omega) \in \of 0, \infty\of\) and \(P^n \in \cR (t^n, \omega^n)\).
		Thanks to Lemma \ref{lem: rel comp}, the set \(\bigcup_{ n \in \mathbb{N} } \cR (t^n, \omega^n)\) is relatively compact. Hence, up to passing to a subsequence, we can assume that the sequence \((P^n)_{n \in \mathbb{N}}\) converges weakly to a limiting measure \(P\). In the following we prove that \(P \in \cR (t, \omega)\). Clearly, as \(P^n \circ X_0^{-1} \to P \circ X^{-1}_0\) and \(P^n \circ X^{-1}_0 = \delta_{\omega^n (t^n)} \to \delta_{\omega (t)}\), we have \(P \circ X^{-1}_0 = \delta_{\omega (t)}\). Thus, we need to show that \(P\in \fPas\) with differential characteristics \((\llambda \otimes P)\)-a.e. in \(\Theta (\cdot + t, \omega \ \widetilde{\otimes}_t\ X)\).
		The proof of this is split into four steps.

  Before we start our program, we need a last bit of notation.
  For each \(n \in \mathbb{N}\), denote the \(P^n\)-characteristics of \(X\) by \((B^n, C^n)\).		
  Set \(\Omega^* := \Omega \times \Omega \times C(\bR_+; \bR^{\d \times \d})\) and denote the coordinate process on \(\Omega^*\) by \(Y = (Y^{(1)}, Y^{(2)}, Y^{(3)})\). Further, set \(\cF^* := \sigma (Y_s, s \geq 0)\) and let \(\F^* = (\cF^*_s)_{s \geq 0}\) be the right-continuous filtration generated by \(Y\).
		
		\emph{Step 1.} First, we show that the family \(\{P^n \circ (X, B^n, C^n)^{-1} \colon n \in \mathbb{N}\}\) is tight (when seen as a sequence of probability measures on the measurable space \((\Omega^*, \cF^*)\)). 
		Since \(P^n \to P\), it suffices to prove tightness of \(\{P^n \circ (B^n, C^n)^{-1} \colon n \in \mathbb{N}\}\). We use Aldous' tightness criterion (\cite[Theorem~VI.4.5]{JS}), i.e., we show the following two conditions:
		\begin{enumerate}
			\item[(a)]
			for every \(T, \varepsilon > 0\), there exists a \(K \in \mathbb{R}_+\) such that
			\[
			\sup_{n \in \mathbb{N}} P^n \Big( \sup_{s \in [0, T]} \|B^n_s\| + \sup_{s \in [0, T]} \on{tr} \big[ C^n_s \big] \geq K \Big) \leq \varepsilon;
			\]
			\item[(b)] 
			for every \(T, \varepsilon > 0\),
			\[
			\lim_{\theta \searrow 0} \limsup_{n \to \infty} \sup \big\{P^n (\|B^n_L - B^n_S\| + \|C^n_L - C^n_S\|_o \geq \varepsilon) \big\} = 0, 
			\]
			where the \(\sup\) is taken over all stopping times \(S, L \leq T\) such that \(S \leq L \leq S + \theta\).
		\end{enumerate}
		By Lemma \ref{lem: rel comp}, we have
		\begin{align} \label{eq: second moment bound}
		\sup_{n \in \mathbb{N}} E^{P^n} \Big[ \sup_{s \in [0, T]} \|X_s\|^2 \Big] 
		< \infty.
		\end{align}
		Using the linear growth assumption, we also obtain that \(P^n\)-a.s.
		\[
		\sup_{s \in [0, T]} \|B^n_s\| + \sup_{s \in [0, T]} \on{tr} \big[C^n_s \big] \leq \mathsf{C} \Big( 1 + \sup_{s \in [0, T]} \|X_s\|^2 \Big),
		\]
		where the constant \(\mathsf{C} >0\) is independent of \(n\).
		By virtue of \eqref{eq: second moment bound}, this bound yields~(a). For (b), take two stopping times \(S, L \leq T\) such that \(S \leq L \leq S + \theta\) for some \(\theta > 0\). Then, using again the linear growth assumptions, we get \(P^n\)-a.s.
		\[
		\|B^n_L - B^n_S\| + \|C^n_L - C^n_S\|_o \leq \mathsf{C} (L - S) \Big( 1 + \sup_{s \in [0, T]} \|X_s\|^2 \Big) \leq \mathsf{C} \theta \Big( 1 + \sup_{s \in [0, T]} \|X_s\|^2 \Big),
		\]
		which yields (b) by virtue of \eqref{eq: second moment bound}. We conclude that \(\{P^n \circ (X, B^n, C^n)^{-1} \colon n \in \mathbb{N}\}\) is tight. 
		Up to passing to a subsequence, from now on we assume that \(P^n \circ (X, B^n, C^n)^{-1} \to \Q\) weakly.
		
\emph{Step 2.} Next, we show that \(Y^{(2)}\) and \(Y^{(3)}\) are \(Q\)-a.s. locally absolutely continuous.
For \(M > 0\) and \(\omega \in \Omega\), define 
\[
\rho_M (\omega) := \inf \{s \geq 0 \colon \|\omega(s)\| \geq M\} \wedge M.
\]
Furthermore, for \(\omega = (\omega^{(1)}, \omega^{(2)}) \in \Omega \times \Omega\), we set 
\[
\zeta_M (\omega) := \sup \Big\{ \frac{\|\omega^{(2)} (u \wedge \rho_M(\omega)) - \omega^{(2)}(s \wedge \rho_M(\omega))\|}{u - s} \colon 0 \leq s < u\Big\}.
\]
Similar to the proof of \cite[Lemma 3.6]{CN22b}, we obtain the existence of a dense set \(D \subset \bR_+\) such that, for every \(M \in D\), the map \(\zeta_M\) is \(\Q \circ (Y^{(1)}, Y^{(2)})^{-1}\)-a.s. lower semicontinuous. By the linear growth conditions and the definition of \(\rho_M\), for every \(M \in D\), there exists a constant \(\mathsf{C} = \mathsf{C} (M) > 0\) such that \(P^n(\zeta_M (X, B^n) \leq \mathsf{C}) = 1\) for all \(n \in \mathbb{N}\). As \(\zeta_M\) is \(\Q \circ (Y^{(1)}, Y^{(2)})\)-a.s.  lower semicontinuous, \cite[Example~17, p. 73]{pollard} yields that 
\[
0 = \liminf_{n \to \infty} P^n (\zeta_M (X, B^n) > \mathsf{C}) \geq \Q (\zeta_M (Y^{(1)}, Y^{(2)}) > \mathsf{C}). 
\]
Further, as \(D\) is dense in \(\bR_+\), we obtain that \(Q\)-a.s. \(Y^{(2)}\) is locally Lipschitz continuous, i.e., in particular locally absolutely continuous. Similarly, we get that \(Y^{(3)}\) is \(Q\)-a.s. locally Lipschitz and hence, locally absolutely continuous. 
		
		\emph{Step 3.} 
		Define the map \(
		\Phi \colon \Omega^* \to \Omega
		\)
		by \(\Phi (\omega^{(1)}, \omega^{(2)}, \omega^{(3)}) := \omega^{(1)}\). Clearly, \(Q \circ \Phi^{-1} = P\) and \(Y^{(1)} = X \circ \Phi\). 
		In this step, we prove that \((\llambda \otimes Q)\)-a.e. \((dY^{(2)} /d \llambda, dY^{(3)}/ d \llambda) \in \Theta (\cdot + t, \omega \ \widetilde{\otimes}_t\ X) \circ \Phi\). 
		By \cite[Lemma 3.2]{CN22b} and the Conditions \ref{cond: joint continuity in all} and \ref{cond: convexity}, the correspondence \(\Theta\) is continuous with compact and convex values. Furthermore, again by Condition~\ref{cond: joint continuity in all}, the set \(\Theta ([s, s + 1], \omega^o)\) is compact for every \((s, \omega^o) \in \of 0, \infty\of\). Consequently, as in completely metrizable locally convex spaces the closed convex hull of a compact set is itself compact (\cite[Theorem 5.35]{charalambos2013infinite}), we conclude that \(\oconv \Theta ([s, s + 1], \omega^o)\) is compact and, by \cite[Lemma~3.4]{CN22b}, we get that
		\begin{align} \label{eq: conseq upper hemi}
		\bigcap_{m \in \mathbb{N}} \oconv \Theta ([s, s + 1/m], \omega^o) \subset \Theta (s, \omega^o)
		\end{align} 
		for all \((s, \omega^o) \in \of 0, \infty \of\). Here, \(\oconv\hspace{-0.075cm}\) denotes the closure of the convex hull.
		By virtue of \cite[Corollary~8, p.~48]{diestel}, \(P^n\)-a.s. for all \(s \in \bR_+\), we have 
		\begin{equation}\label{eq: P as inclusion theta}
		\begin{split}
		m (B^n_{s + 1/m} - B^n_s, C^n_{s + 1/m} - C^n_s) &\in \oconv ( dB^n / d \llambda, d C^n / d \llambda) ([ s, s + 1/m ]) \\&\subset \oconv \Theta ([t^n + s, t^n + s + 1/m], \omega^n \ \widetilde{\otimes}_{t_n}\ X).
		\end{split}
		\end{equation}
		Thanks to Skorokhod's coupling theorem, with a little abuse of notation, there exist random variables \[(X^0, B^0, C^0), (X^1, B^1, C^1), (X^2, B^2, C^2), \dots,\] defined on some probability space \((\Sigma, \mathcal{G}, R)\), such that \((X^0, B^0, C^0)\) has distribution \(Q\), \((X^n, B^n, C^n)\) has distribution \(P^n\circ (X, B^n, C^n)^{-1}\) and \(R\)-a.s. \((X^n, B^n, C^n) \to (X^0, B^0, C^0)\) in the local uniform topology. 
		Thanks to Condition \ref{cond: joint continuity in all}, \cite[Lemmata 3.2, 3.3]{CN22b} and \cite[Theorem 17.23]{charalambos2013infinite}, the correspondence \((s, \omega^o) \mapsto \Theta ([s, s + 1/m], \omega^o)\) is continuous. Since, for every \((s, \omega^o) \in \bR_+ \times \Omega\), the set \(\oconv \Theta ([s, s + 1/m], \omega^o)\) is compact, we deduce from \cite[Theorem 17.35]{charalambos2013infinite} that the correspondence \((s, \omega^o) \mapsto \oconv \Theta ([s, s + 1/m], \omega^o)\) is upper hemicontinuous with compact values. 
		Thus, \eqref{eq: P as inclusion theta}, \cite[Theorem~17.20]{charalambos2013infinite} and Corollary \ref{coro: real concatenation continuous} yield that, for all \(m \in \mathbb{N}\), \(R\)-a.s. for all \(s \in \bR_+\)
		\[
		m (B^0_{s + 1/m} - B^0_s, C^0_{s + 1/m} - C^0_s) \in \oconv \Theta ([s+ t, s + 1/m + t], \omega \ \widetilde{\otimes}_{t}\ X^0).
		\]
		Notice that \((\llambda \otimes R)\)-a.e.
		\[
		(d B^0 / d \llambda, d C^0 / d \llambda) =\lim_{m \to \infty} m (B^0_{\cdot + 1/m} - B^0_\cdot, C^0_{\cdot + 1/m} - C^0_\cdot).
		\]
		Now, with \eqref{eq: conseq upper hemi}, we get that \(R\)-a.s. for \(\llambda\)-a.a. \(s \in \bR_+\)
		\[
		(d B^0 / d \llambda, d C^0 / d \llambda) (s) \in \bigcap_{m \in \mathbb{N}} \oconv \Theta ([s + t, s + 1/m+ t], \omega \ \widetilde{\otimes}_{t}\ X^0) \subset \Theta (s + t, \omega \ \widetilde{\otimes}_{t}\ X^0).
		\]
		This shows that \( (\llambda \otimes Q)\)-a.e. \((dY^{(2)} /d \llambda, dY^{(3)}/ d \llambda) \in \Theta (\cdot + t, \omega \ \widetilde{\otimes}_t\ X) \circ \Phi\). 
		
		\emph{Step 4.} In the final step of the proof, we show that \(P \in \fPas\) and we relate \((Y^{(2)}, Y^{(3)})\) to the \(P\)-semimartingale characteristics of the coordinate process. 
		Thanks to \cite[Lemma~11.1.2]{SV}, there exists a dense set \(D \subset \bR_+\) such that \(\rho_M \circ \Phi\) is \(Q\)-a.s. continuous for all \(M \in D\). Take some \(M \in D\). Since \(P^n \in \fPas\), it follows from the definition of the first characteristic that the process \(X_{\cdot \wedge \rho_M} - B^n_{\cdot \wedge \rho_M}\) is a local \(P^n\)-\(\F_+\)-martingale. Furthermore, by the definition of the stopping time \(\rho_M\) and the linear growth Condition~\ref{cond: LG}, we see that \(X_{\cdot \wedge \rho_M} - B^n_{\cdot \wedge \rho_M}\) is \(P^n\)-a.s. bounded by a constant independent of~\(n\), which, in particular, implies that it is a true \(P^n\)-\(\F_+\)-martingale. Now, it follows from \cite[Proposition~IX.1.4]{JS} that \(Y^{(1)}_{\cdot \wedge \rho_M \circ \Phi} - Y^{(2)}_{\cdot \wedge \rho_M \circ \Phi}\) is a \(Q\)-\(\F^*\)-martingale. Recalling that \(Y^{(2)}\) is \(Q\)-a.s. locally absolutely continuous by Step 2, this means that \(Y^{(1)}\) is a \(Q\)-\(\F^*\)-semimartingale with first characteristic \(Y^{(2)}\). Similarly, we see that the second characteristic is given by~\(Y^{(3)}\). Finally, we need to relate these observations to the probability measure \(P\) and the filtration \(\F_+\). We denote by \(A^{p, \Phi^{-1}(\F_+)}\) the dual predictable projection of a process~\(A\), defined on \((\Omega^*, \cF^*)\), to the filtration \(\Phi^{-1}(\F_+)\). Recall from \cite[Lemma 10.42]{jacod79} that, for every \(s \in \bR_+\), a random variable \(Z\) on \((\Omega^*, \cF^*)\) is \(\Phi^{-1}(\cF_{s+})\)-measurable if and only if it is \(\cF^*_s\)-measurable and \(Z (\omega^{(1)}, \omega^{(2)}, \omega^{(3)})\) does not depend on \((\omega^{(2)}, \omega^{(3)})\).
		Thanks to Stricker's theorem (see, e.g., \cite[Lemma~2.7]{jacod80}), \(Y^{(1)}\) is a \(Q\)-\(\Phi^{-1} (\F_+)\)-semimartingale. 
		Notice that each \(\rho_M \circ \Phi\) is a \(\Phi^{-1}(\F_+)\)-stopping time and recall from Step~3 that \( (\llambda \otimes Q)\)-a.e. \((d Y^{(2)}/ d \llambda, d Y^{(3)}/ d \llambda) \in \Theta (\cdot + t, \omega \ \widetilde{\otimes}_t\ X) \circ \Phi\). Hence, by definition of \(\rho_M\) and the linear growth assumption, for every \(M \in D\) and \(i, j = 1, \dots, \d\), we have
		\begin{align*}
		E^Q \big[ \on{Var} (Y^{(2, i)})_{\rho_M \circ \Phi} \big] &+ E^Q \big[ \on{Var}(Y^{(3, ij)})_{\rho_M \circ \Phi} \big] \\&= E^Q \Big[ \int_0^{\rho_M \circ \Phi} \Big(\Big| \frac{d Y^{(2,i)}}{d \llambda} \Big| + \Big| \frac{d Y^{(3,ij)}}{d \llambda} \Big| \Big) d \llambda \Big] < \infty,
		\end{align*}
		where \(\on{Var} (\cdot)\) denotes the variation process.
		By virtue of this, we get from \cite[Proposition~9.24]{jacod79} that the \(Q\)-\(\Phi^{-1}(\F_+)\)-characteristics of \(Y^{(1)}\) are given by \(((Y^{(2)})^{p, \Phi^{-1}(\F_+)},\) \((Y^{(3)})^{p, \Phi^{-1}(\F_+)})\). 
		Hence, thanks to Lemma \ref{lem: jacod restatements}, the coordinate process \(X\) is a \(P\)-\(\F_+\)-semimartingale whose characteristics \((B^P, C^P)\) satisfy \(Q\)-a.s.
		\[(B^P, C^P) \circ \Phi = ((Y^{(2)})^{p, \Phi^{-1}(\F_+)}, (Y^{(3)})^{p, \Phi^{-1}(\F_+)}).\] Consequently, we deduce from the Steps~2 and 3, and \cite[Theorem 5.25]{HWY}, that \(P\)-a.s. \((B^P, C^P) \ll \llambda\) and 
		\begin{align*}
		(\llambda \otimes P) \big( (&d B^P / d \llambda, d C^P / d \llambda) \not \in \Theta (\cdot + t, \omega \ \widetilde{\otimes}_t\ X) \big) 
		\\&= (\llambda \otimes Q \circ \Phi^{-1}) \big( (d B^P / d \llambda, d C^P / d \llambda) \not \in \Theta (\cdot + t, \omega \ \widetilde{\otimes}_t\ X) \big)
		\\&= (\llambda \otimes Q) \big( E^Q [(d Y^{(2)} / d \llambda, d Y^{(3)} / d \llambda) | \Phi^{-1} (\F_+)_-] \not \in \Theta (\cdot + t, \omega \ \widetilde{\otimes}_t\ Y^{(1)}) \big) = 0,
		\end{align*}
		where we use \cite[Corollary 8, p. 48]{diestel} for the final equality. 
 \qed
	
	\subsection{Proof of Theorem \ref{theo: lower hemi}}
	To prove Theorem \ref{theo: lower hemi}, we use the sequential characterization of lower hemicontinuity as given by \cite[Theorem 17.21]{charalambos2013infinite}. Namely, we prove that for every sequence \((t^n, \omega^n)_{n \in \mathbb{Z}_+} \subset \of 0, \infty\of\) such that \((t^n, \omega^n) \to (t^0, \omega^0)\) and every \(P \in \cR (t^0, \omega^0)\), there exists 
	a sequence \(P^{n} \in \cR (t^{n}, \omega^{n})\), for \(n \in \mathbb{N}\), such that \(P^{n} \to P\) weakly. 
	From now on, we fix \((t^n, \omega^n)_{n \in \mathbb{Z}_+} \subset \of 0, \infty\of\) such that \((t^n, \omega^n) \to (t^0, \omega^0)\) and \(P \in \cR (t^0, \omega^0)\).
	
	{\em Step 1: On the structure of \(P\).}
	Recall that \(\mathscr{P}\) denotes the predictable \(\sigma\)-field on \(\of 0, \infty \of\) and let \((b^P, a^P)\) be the Lebesgue densities of the \(P\)-\(\F\)-characteristics of the coordinate process \(X\). 
	\begin{lemma} \label{lem: representation via control}
		Assume that \(F\) is a compact metrizable space and that Condition \ref{cond: continuity control} holds. Then, for every \((t^0, \omega^0) \in \of 0, \infty\of\) and every \(P \in \cR (t^0, \omega^0)\), there exists a predictable map \(\f = \f (P) \colon \of 0, \infty\of \hspace{0.05cm} \to F\) such that, for \((\llambda\otimes P)\)-a.e. \((t, \omega) \in \of 0, \infty\of\),
		\[
		(b^P, a^P) (t, \omega)= ( b (\f (t, \omega),  t + t^0,\omega^0 \ \widetilde{\otimes}_{t^0}\ \omega), a (\f (t, \omega), t + t^0, \omega^0 \ \widetilde{\otimes}_{t^0}\ \omega) ),
		\]
		where \((b^P, a^P)\) denote the Lebesgue densities of the \(P\)-\(\mathbf{F}\)-characteristics of the coordinate process \(X\).
	\end{lemma}
	\begin{proof}
		By virtue of \cite[Theorem 97, p. 147]{dellacheriemeyer}, the maps \((t, \omega) \mapsto b (f, t + t^0, \omega^0 \ \widetilde{\otimes}_{t^0}\ \omega)\) and \((t, \omega) \mapsto a (f, t + t^0, \omega^0 \ \widetilde{\otimes}_{t^0}\ \omega)\) are predictable for every \(f \in F\). Hence, 
		by Lemma \ref{lem: correspondence meas graph general result} and Condition~\ref{cond: continuity control}, the correspondence \((t, \omega) \mapsto \Theta^* (t, \omega) := \Theta (t + t^0, \omega^0 \ \widetilde{\otimes}_{t^0}\ \omega)\) has a
		\( \mathscr{P} \otimes \mathcal{B}(\bR) \otimes \mathcal{B}(\mathbb{R}_+)\)-measurable graph and consequently, 
		\begin{align*}
		G :\hspace{-0.1cm}&= \big\{ (t, \omega) \in \of 0, \infty\of \colon (b^P_t (\omega), a^P_t (\omega)) \in \Theta ( t + t^0, \omega^0 \ \widetilde{\otimes}_{t^0}\ \omega )\big\}
		\\&= \big\{ (t, \omega) \in \of 0, \infty\of \colon (t, \omega, b^P_t (\omega), a^P_t (\omega)) \in \on{gr} \Theta^* \big\} 
		\in \mathscr{P}.
		\end{align*}
		For some fixed, but arbitrary, \(f_0 \in F\), define 
		\[
		\pi (t, \omega) := \begin{cases} (b (f_0, t + t^0,\omega^0 \ \widetilde{\otimes}_{t^0}\ \omega), a (f_0, t + t^0,\omega^0 \ \widetilde{\otimes}_{t^0}\ \omega)),& \text{if } (t, \omega) \not \in G,\\
		(b^P_t (\omega), a^P_t (\omega)),& \text{if } (t, \omega) \in G. \end{cases}
		\]
		As the function \((f, t, \omega) \mapsto (b (f, t + t^0, \omega^0 \ \widetilde{\otimes}_{t^0}\ \omega), a (f, t + t^0, \omega^0 \ \widetilde{\otimes}_{t^0}\ \omega))\) is continuous in the \(F\) and \(\mathscr{P}\)-measurable in the \(\of 0, \infty\of\) variable, we deduce from the measurable implicit function theorem \cite[Theorem~18.17]{charalambos2013infinite} that the correspondence \(\gamma \colon \of 0, \infty\of \hspace{0.075cm} \twoheadrightarrow F\) defined by 
		\[
		\gamma (t, \omega) := \big\{ f \in F \colon (b (f, t + t^0, \omega^0 \ \widetilde{\otimes}_{t^0}\ \omega), a (f, t + t^0,\omega^0 \ \widetilde{\otimes}_{t^0}\ \omega)) = \pi (t, \omega) \big\}
		\]
		is \(\mathscr{P}\)-measurable and that it admits a measurable selector, i.e., there exists a \(\mathscr{P}\)-measurable function \(\f \colon \of 0, \infty\of\hspace{0.075cm} \to F\) such that \[\pi (t, \omega) = (b (\f(t, \omega), t + t^0,\omega^0 \ \widetilde{\otimes}_{t^0}\ \omega),a (\f(t, \omega), t + t^0,\omega^0 \ \widetilde{\otimes}_{t^0}\ \omega))\] for all \((t, \omega) \in \of 0, \infty\of\). Since \(P \in \cR(t, \omega)\), we have \((\llambda \otimes P)\)-a.e. \(\pi = (b^P, a^P)\). This completes the proof.
	\end{proof}
	Let \(\f = \f (P)\) be as in Lemma \ref{lem: representation via control}.
	As Condition \ref{cond: lipschitz} is assumed in Theorem \ref{theo: lower hemi}, we have a decomposition \(a = \sigma \sigma^*\) for an \(\bR^{\d \times \dr}\)-valued function \(\sigma\).
	By a standard integral representation result for continuous local martingales (see, e.g., \cite[Theorem II.7.1\('\)]{IW}), possibly on an extension of the filtered probability space \((\Omega, \cF, \F, P)\), there exists an \(\dr\)-dimensional standard Brownian motion \(W\) such that \(P\)-a.s.
	\[
	d X_t = b (\f (t, X), t + t^0, \omega^0 \ \widetilde{\otimes}_{t^0}\ X) d t + \sigma(\f (t, X), t + t^0, \omega^0 \ \widetilde{\otimes}_{t^0}\ X) d W_t, \quad X_0 = \omega^0 (t^0).
	\]
	For simplicity, we ignore the standard extension in our notation. 
	
	{\em Step 2: A candidate for the approximating sequence.} By Condition \ref{cond: lipschitz}, we obtain from a short computation that, for every \(n \in \mathbb{Z}_+\), \(T, M > 0\) and every \(\alpha^0 \in \Omega\), there exists a constant \(C > 0\), that depends on \(T, M, \sup_{n \in \mathbb{Z}_+} t^n\) and \(\sup_{n \in \mathbb{Z}_+} \sup_{s \in [0, t^n]} \| \omega^n (s) \|\), such that 
	\begin{equation} \label{eq: Lipschitz concaten}
	\begin{split}
	\| b (\f (t, \alpha^0), t + t^n, \omega^n \ \widetilde{\otimes}_{t^n}\ \omega) - b (\f (t, \alpha^0), t + t^n, \omega^n \ \widetilde{\otimes}_{t^n}\ \alpha)\| &\leq C \sup_{s \in [0, t]} \| \omega(s) - \alpha(s)\|,
	\\ \| \sigma (\f (t, \alpha^0), t + t^n, \omega^n \ \widetilde{\otimes}_{t^n}\ \omega) - \sigma (\f (t, \alpha^0), t + t^n, \omega^n \ \widetilde{\otimes}_{t^n}\ \alpha)\|_o &\leq C \sup_{s \in [0, t]} \| \omega(s) - \alpha(s)\|,
	\end{split}
	\end{equation}
	for all \(\omega, \alpha \in \Omega \colon \sup_{s \in [0, t]} \| \omega (s) \| \vee \| \alpha (s) \| \leq M\) and \(t \in [0, T]\).
	Furthermore, we deduce from Condition \ref{cond: LG} that, for every \(T > 0\), there exists a constant \(C > 0\), that depends on \(T,\) \(\sup_{n \in \mathbb{Z}_+} t^n\) and \(\sup_{n \in \mathbb{Z}_+} \sup_{s \in [0, t^n]} \| \omega^n (s) \|\), such that
	\begin{align*}
	\|b (\f (t, \alpha^0), t + t^n, \omega^n \ \widetilde{\otimes}_{t^n}\ \omega)\| + \|\sigma (\f (t, \alpha^0), t + t^n, \omega^n \ \widetilde{\otimes}_{t^n}\ \omega)\|_o \leq C \Big(1 + \sup_{s \in [0, t]} \| \omega (s)\| \Big),
	\end{align*}
	for all \(\omega, \alpha \in \Omega\) and \(t \in [0, T]\).
	Hence, by a standard existence result for stochastic differential equations with random locally Lipschitz coefficients of linear growth (see, e.g., \cite[Theorem 14.30]{jacod79} or \cite[Theorem 4.5]{jacod1981weak}), for every \(n \in \mathbb{N}\), there exists a continuous adapted process \(Y^n\) such that \(P\)-a.s.
	\[
	d Y^n_t = b (\f (t, X), t + t^n, \omega^n \ \widetilde{\otimes}_{t^n}\ Y^n) d t + \sigma(\f (t, X), t + t^n, \omega^n \ \widetilde{\otimes}_{t^n}\ Y^n) d W_t, \quad Y^n_0 = \omega^n (t^n).
	\] 
	Next, we show that the laws of \(\{Y^n \colon n \in \mathbb{N}\}\) form a candidate for an approximation sequence of \(P\).
	\begin{lemma}  \label{lem: good candidate}
		Suppose that the Conditions \ref{cond: LG} and \ref{cond: convexity} hold and take \((t^0, \omega^0) \in \of 0, \infty \of\). On some filtered probability space, let \(\zeta\) be an \(F\)-valued measurable process and let \(Y\) be a continuous semimartingale starting at \(\omega^0 (t^0)\) those semimartingale characteristics are absolutely continuous with densities \((b (\zeta, \cdot + t^0, \omega^0 \ \widetilde{\otimes}_{t^0}\ Y), a(\zeta, \cdot + t^0, \omega^0 \ \widetilde{\otimes}_{t^0}\ Y))\). Then, the law of \(Y\) is an element of \(\cR (t^0, \omega^0)\).
	\end{lemma}
	\begin{proof}
 Let \(P\) be the probability measure of the underlying filtered space and let \(Q := P \circ Y^{-1}\) be the law of \(Y\).
		Further, denote the natural right-continuous filtration of~\(Y\) by \((\mathcal{G}_t)_{t \geq 0}\). Thanks to Stricker's theorem (see, e.g., \cite[Lemma~2.7]{jacod80}), \(Y\) is a \(P\)-\((\mathcal{G}_t)_{t \geq 0}\)-semimartingale. Furthermore, by virtue of Condition \ref{cond: LG}, we deduce from \cite[Proposition 9.24]{jacod79} and \cite[Theorem 5.25]{HWY} that the \((\mathcal{G}_t)_{t \geq 0}\)-characteristics of \(Y\) are given by \[
		\Big( \int_0^\cdot E^P \big[ b (\zeta, t + t^0, \omega^0 \ \widetilde{\otimes}_{t^0}\ Y) | \mathcal{G}_{t-} \big] dt, \int_0^\cdot E^P \big[ a (\zeta, t + t^0, \omega^0 \ \widetilde{\otimes}_{t^0}\ Y) | \mathcal{G}_{t-} \big] dt \Big). 
		\]
		By Lemma \ref{lem: jacod restatements}, the coordinate process \(X\) is a \(Q\)-\(\F_+\)-semimartingale and its characteristics \((B, C)\) are \(Q\)-a.s. absolutely continuous and such that \(P\)-a.s. for \(\llambda\)-a.a. \(t \in \bR_+\)
		\[
		\Big(\frac{d B}{d \llambda} , \frac{d C }{d \llambda} \Big)(t, Y) = E^P \big[ \big(b (\zeta, t + t^0, \omega^0\ \widetilde{\otimes}_{t^0}\ Y), a (\zeta, t + t^0, \omega^0 \ \widetilde{\otimes}_{t^0}\ Y) \big) | \mathcal{G}_{t-} \big].
		\]
		Using the convexity assumption given by Condition \ref{cond: convexity} and \cite[Corollary 8, p. 48]{diestel}, we obtain that 
		\begin{align*}
		(\llambda \otimes Q) ( (d B &/ d \llambda, d C /d \llambda) \not \in \Theta (\cdot + t^0, \omega^0 \ \widetilde{\otimes}_{t^0} X) ) 
		\\&= \iint \1 \big \{ E^P \big[ \big(b (\zeta, t + t^0, \omega^0 \ \widetilde{\otimes}_{t^0}\ Y), \\&\hspace{2cm} a (\zeta, t + t^0, \omega^0 \ \widetilde{\otimes}_{t^0}\ Y) \big) | \mathcal{G}_{t-} \big] \not \in \Theta (t + t^0, \omega^0\ \widetilde{\otimes}_{t^0} \ Y) \big\} d (\llambda \otimes P)
		\\&= 0.
		\end{align*}
		This completes the proof.
	\end{proof}
	Thanks to Lemma \ref{lem: good candidate}, for every \(n \in \mathbb{N}\), the law of \(Y^n\) is an element of \(\cR (t^n, \omega^n)\). Consequently, the laws of \(\{Y^n \colon n \in \mathbb{N}\}\) are candidates for an approximation sequence of the measure \(P\).
	
	{\em Step 3: ucp convergence of \(Y^n\) to \(X\):} In this step we prove that \(Y^n \to X\) in the ucp topology, i.e., uniformly on compact time sets in probability. Hereby, we use some ideas we learned from the proof of \cite[Theorem, p.~88]{MP80}.
	Take \(T ,\varepsilon > 0\) and define 
	\[
	\tau^m_n := \inf \{t \geq 0 \colon \|X_t\| \vee \|Y^n_t\| \geq m\}, \qquad n, m > 0.
	\]
	Thanks to the moment bound from Lemma \ref{lem: rel comp}, we obtain that 
	\begin{align*}
	P (T^m_n \leq T) = P \Big( \sup_{s \in [0, T]} \|X_s\| \vee \|Y^n_s\| \geq m \Big) \leq \frac{C}{m}, 
	\end{align*}
	where the constant \(C > 0\) is independent of \(n\) and \(m\). Thus, we can take \(M = M_\varepsilon > 0\) large enough such that 
	\begin{align} \label{eq: M large enough}
	\sup_{n \in \mathbb{N}} P (T^M_n \leq T) \leq \varepsilon.
	\end{align}
	Next, using the Burkholder--Davis--Gundy inequality and H\"older's inequality, and \eqref{eq: Lipschitz concaten}, we obtain, for every \(t \in [0, T]\) and \(n \in \mathbb{N}\), that 
	\begin{align*}
	E^P &\Big[ \sup_{s \in [0, t \wedge T^M_n]} \|X_s - Y^n_s\|^2 \Big] \\&\leq C \Big( \|\omega^0 (t^0) - \omega^n (t^n)\|^2 \\&\qquad \quad+ E^P \Big[ \int_0^{t \wedge T^M_n} \big( \| b (\f (s, X), s + t^0, \omega^0 \ \widetilde{\otimes}_{t^0}\ X) - b (\f (s, X), s + t^n, \omega^n \ \widetilde{\otimes}_{t^n}\ Y^n)\|^2
	\\&\hspace{2.2cm}+ \| \sigma (\f (s, X), s + t^0, \omega^0 \ \widetilde{\otimes}_{t^0}\ X) - \sigma (\f (s, X), s + t^n, \omega^n \ \widetilde{\otimes}_{t^n}\ Y^n)\|_o^2 \big) ds \Big]\Big)
	\\&\leq C \Big( \|\omega^0 (t^0) - \omega^n (t^n)\|^2 \\&\qquad \quad+ E^P \Big[ \int_0^{t \wedge T^M_n} \big( \| b (\f (s, X), s + t^0, \omega^0 \ \widetilde{\otimes}_{t^0}\ X) - b (\f (s, X), s + t^n, \omega^n \ \widetilde{\otimes}_{t^n}\ X)\|^2
	\\&\hspace{2.2cm} +\| b (\f (s, X), s + t^n, \omega^n \ \widetilde{\otimes}_{t^n}\ X) - b (\f (s, X), s + t^n, \omega^n \ \widetilde{\otimes}_{t^n}\ Y^n)\|^2
	\\&\hspace{2.2cm} +\| \sigma (\f (s, X), s + t^0, \omega^0 \ \widetilde{\otimes}_{t^0}\ X) - \sigma (\f (s, X), s + t^n, \omega^n \ \widetilde{\otimes}_{t^n}\ X)\|_o^2
	\\&\hspace{2.2cm} +\| \sigma (\f (s, X), s + t^n, \omega^n \ \widetilde{\otimes}_{t^n}\ X) - \sigma (\f (s, X), s + t^n, \omega^n \ \widetilde{\otimes}_{t^n}\ Y^n)\|_o^2 \big) ds \Big]\Big)
	\\&\leq C \Big( \|\omega^0 (t^0) - \omega^n (t^n)\|^2 \\&\qquad \quad+ E^P \Big[ \int_0^{t \wedge T^M_n} \big( \| b (\f (s, X), s + t^0, \omega^0 \ \widetilde{\otimes}_{t^0}\ X) - b (\f (s, X), s + t^n, \omega^n \ \widetilde{\otimes}_{t^n}\ X)\|^2
	\\&\hspace{2.2cm} +\| \sigma (\f (s, X), s + t^0, \omega^0 \ \widetilde{\otimes}_{t^0}\ X) - \sigma (\f (s, X), s + t^n, \omega^n \ \widetilde{\otimes}_{t^n}\ X)\|_o^2 \big) ds \Big]
	\\&\hspace{2.2cm} + E^P \Big[ \sup_{s \in [0, t \wedge T^M_n]} \|X_s - Y^n_s\|^2 \Big] \Big).
	\end{align*}
	Hence, from Gronwall's lemma, we get, for all \(n \in \mathbb{N}\), that
	\begin{align*}
	E^P \Big[ &\sup_{s \in [0, T \wedge T^M_n]} \|X_s - Y^n_s\|^2 \Big] 
	\\&\leq C \Big( \|\omega^0 (t^0) - \omega^n (t^n)\|^2 \\&\hspace{0.5cm}+ E^P \Big[ \int_0^{T \wedge T^M_n} \big( \| b (\f (s, X), s + t^0, \omega^0 \ \widetilde{\otimes}_{t^0}\ X) - b (\f (s, X), s + t^n, \omega^n \ \widetilde{\otimes}_{t^n}\ X)\|^2
	\\&\hspace{2.2cm} +\| \sigma (\f (s, X), s + t^0, \omega^0 \ \widetilde{\otimes}_{t^0}\ X) - \sigma (\f (s, X), s + t^n, \omega^n \ \widetilde{\otimes}_{t^n}\ X)\|_o^2 \big) ds \Big]\Big).
	\end{align*}
	where the constant \(C > 0\) is independent of \(n\). Clearly, the first term on the r.h.s. converges to zero as \(n \to \infty\). The second term converges to zero by Condition \ref{cond: joint continuity in all}, Corollary \ref{coro: real concatenation continuous} and the dominated convergence theorem, which can be applied thanks to the linear growth Condition \ref{cond: LG} and either the definition of the sequence \((T^M_n)_{n \in \mathbb{N}}\), or the uniform moment bound from Lemma \ref{lem: rel comp}. We conclude that 
	\begin{align} \label{eq: conv difference}
	E^P \Big[ \sup_{s \in [0, T \wedge T^M_n]} \|X_s - Y^n_s\|^2 \Big]  \to 0 \text{ as } n \to \infty.
	\end{align}
	We are in the position to complete the proof for ucp convergence. Namely, using \eqref{eq: M large enough} and \eqref{eq: conv difference}, we get, for every \(\delta > 0\), that 
	\begin{align*}
	P \Big( \sup_{s \in [0, T]} \|X_s - Y^n_s\| \geq \delta \Big) &\leq P \Big( \sup_{s \in [0, T]} \|X_s - Y^n_s\| \geq \delta, T^M_n > T \Big) + \varepsilon
	\\&\leq \frac{1}{\delta^2} \ E^P \Big[ \sup_{s \in [0, T \wedge T^M_n]} \|X_s - Y^n_s\|^2 \Big] + \varepsilon 
	\to \varepsilon
	\end{align*}
	as \(n \to \infty\). This proves that \(Y^n \to X\) in the ucp topology. For continuous processes, ucp convergences implies weak convergences (which follows, for instance, from \cite[Corollary 23.5]{Kallenberg}). Hence, the proof of lower hemicontinuity is complete. \qed
	
	\subsection{Proof of Theorem \ref{theo: holder estimate value function}} \label{sec: pf holder}
	
	Fix \(T > 0\) and let \(\psi\) be a bounded Lipschitz continuous function. 
	We fix \(\omega^0, \alpha^0 \in \Omega\) and \(t^0 \in [0, T]\). 
	Take an arbitrary measure \(P \in \cR (t^0, \omega^0)\).
	By Lemma~\ref{lem: representation via control} and some classical representation result for continuous local martingales, there exists a predictable function \(\f = \f (P) \colon \of 0, \infty\of \hspace{0.1cm} \to F\) such that, possibly on an extension of the filtered probability space \((\Omega, \cF, \F, P)\), there exists an \(\dr\)-dimensional standard Brownian motion \(W\) such that \(P\)-a.s.
	\[
	d X_t = b (\f (t, X), t + t^0, \omega^0 \ \widetilde{\otimes}_{t^0}\ X) d t + \sigma(\f (t, X), t + t^0, \omega^0 \ \widetilde{\otimes}_{t^0}\ X) d W_t, \quad X_0 = \omega^0 (t^0).
	\]
	Thanks to Condition \ref{cond: global lip and bdd}, by a standard existence result for stochastic differential equations with random locally Lipschitz coefficients of linear growth, there exists a continuous adapted process \(Y\) such that \(P\)-a.s.
	\[
	d Y_t = b (\f (t, X), t + t^0, \alpha^0 \ \widetilde{\otimes}_{t^0}\ Y) d t + \sigma(\f (t, X), t + t^0, \alpha^0 \ \widetilde{\otimes}_{t^0}\ Y) d W_t, \quad Y_0 = \alpha^0 (t^0).
	\] 
	Using the Burkholder--Davis--Gundy inequality, H\"older's inequality and the global Lipschitz part from Condition \ref{cond: global lip and bdd}, we obtain that 
	\begin{align*}
	E^P &\Big[ \sup_{s \in [0, T]} \|X_s - Y_s\|^2 \Big] 
	\\ &\leq C \Big( \|\omega (t^0) - \alpha (t^0)\|^2 
	\\&\hspace{1cm}+ \int_0^T E^P \big[ \|b (\f (t, X), t + t^0, \omega^0 \ \widetilde{\otimes}_{t^0}\ X) - b (\f (t, X), t + t^0, \alpha^0 \ \widetilde{\otimes}_{t^0}\ Y)\|^2 \big]ds 
	\\&\hspace{1cm} + \int_0^T E^P \big[ \|\sigma (\f (t, X), t + t^0, \omega^0 \ \widetilde{\otimes}_{t^0}\ X)- \sigma(\f (t, X), t + t^0, \alpha^0 \ \widetilde{\otimes}_{t^0}\ Y)\|^2_o \big] ds \Big) 
	\\&\leq C \Big( \sup_{s \in [0, t^0]} \| \omega^0 (s) - \alpha^0 (s) \|^2 + \int_0^T E^P \Big[ \sup_{s \in[0, t]} \| X_s - Y_s\|^2 \Big] ds\Big).
	\end{align*}
	Now, Gronwall's lemma yields that 
	\[
	E^P \Big[ \sup_{s \in [0, T]} \|X_s - Y_s\|^2 \Big]  \leq C \hspace{0.05cm} \sup_{s \in [0, t^0]} \| \omega^0 (s) - \alpha^0 (s) \|^2,
	\]
	and, by Jensen's inequality,
	\begin{align} \label{eq: Lipschitz inequality in space}
	E^P \Big[ \sup_{s \in [0, T]} \|X_s - Y_s\| \Big]  \leq C \hspace{0.05cm} \sup_{s \in [0, t^0]} \| \omega^0 (s) - \alpha^0 (s) \|.
	\end{align}
	Thanks to Lemma \ref{lem: good candidate}, we have \(P \circ Y^{-1} \in \cR (t^0, \alpha^0)\). Thus, in case  \(v (t^0, \omega^0) \geq v(t^0, \alpha^0)\), we get, from Corollary~\ref{coro: connection ElK NVH}, the Lipschitz continuity of \(\psi\) and \eqref{eq: Lipschitz inequality in space}, that 
	\begin{align*}
	v (&t^0, \omega^0) - v (t^0, \alpha^0) 
	\\&\leq \sup_{P \in \cR (t^0, \omega^0)} \Big(E^P \big[ | \psi (\omega^0 \ \widetilde{\otimes}_{t^0}\ X) - \psi (\alpha^0 \ \widetilde{\otimes}_{t^0}\ Y) | \big] + E^P \big[ \psi (\alpha^0 \ \widetilde{\otimes}_{t^0}\ Y) \big] \Big) - v (t^0, \alpha^0)
	\\&\leq C \Big(\sup_{s \in [0, t^0]} \|\omega^0 (s) - \alpha^0(s)\| + \sup_{P \in \cR (t^0, \omega)} E^P \Big[ \sup_{s \in [0, T]} \|X_s - Y_s\| \Big] \Big) + v (t^0, \alpha^0) - v (t^0, \alpha^0)
	\\&\leq C \hspace{0.05cm} \sup_{s \in [0, t^0]} \| \omega^0 (s) - \alpha^0 (s) \|.
	\end{align*}
	By symmetry, this inequality yields that 
	\begin{align} \label{eq: lipschitz ineq value function}
	|v (t^0, \omega^0) - v (t^0, \alpha^0)| \leq C \hspace{0.05cm} \sup_{s \in [0, t^0]} \| \omega^0 (s) - \alpha^0 (s) \|.
	\end{align}
	Next, let \(0 \leq s_0 \leq t_0\). Using the DPP as given by Theorem \ref{theo: DPP} for the first line and \eqref{eq: lipschitz ineq value function} for the third line, we obtain that 
	\begin{equation} \label{eq: holder inequality in time p1}
	\begin{split}
	| v (s^0, \omega^0) - v (t^0, \omega^0) | &= \Big| \sup_{P \in \cA (s^0, \omega^0)} E^P \big[ v (t^0, X) \big] - v (t^0, \omega^0) \Big| 
	\\&\leq \sup_{P \in \cA (s^0, \omega^0)} E^P \big[ | v (t^0,  X) - v (t^0, \omega^0) | \big]
	\\&\leq \sup_{P \in \cA (s^0, \omega^0)} E^P \Big[ \sup_{s \in [0, t^0]} \| X_s - \omega^0 (s) \| \Big]
	\\&= \sup_{P \in \cA (s^0, \omega^0)} E^P \Big[ \sup_{s \in [s^0, t^0]} \| X_s - \omega^0 (s) \| \Big]
	\\&\leq \sup_{P \in \cA (s^0, \omega^0)} E^P \Big[ \sup_{s \in [s^0, t^0]} \| X_s - \omega^0 (s^0) \| \Big] \\&\hspace{4cm}+ \sup_{s \in [s^0, t^0]} \| \omega^0 (s) - \omega^0 (s^0) \|.
	\end{split}
	\end{equation}
	Using Condition \ref{cond: LG}, it follows as in the solution to \cite[Problem 3.3.15]{KaraShre} that
	\begin{align} \label{eq: holder inequality in time p2}
	\sup_{P \in \cA (s^0, \omega^0)} E^P \Big[ \sup_{s \in [s^0, t^0]} \| X_s - \omega^0 (s^0) \| \Big] \leq C \Big( 1 + \sup_{s \in [0, s^0]} \|\omega^0 (s)\| \Big) |t^0 - s^0|^{1/2}.
	\end{align}
	Finally, using \eqref{eq: lipschitz ineq value function}, \eqref{eq: holder inequality in time p1} and \eqref{eq: holder inequality in time p2}, we obtain that 
	\begin{align*}
	| v (t^0, \omega^0) &- v (s^0, \alpha^0) | 
	\\&\leq | v (t^0, \omega^0) - v (s^0, \omega^0) | + | v (s^0, \omega^0) - v (s^0, \alpha^0) | 
	\\&\leq C \Big[ \Big( 1 + \sup_{s \in [0, s^0]} \|\omega^0 (s)\| \Big) |t^0 - s^0|^{1/2} 
	\\&\hspace{3cm}+ \sup_{s \in [s^0 , t^0]} \| \omega^0 (s) - \omega^0 (s^0) \| + \sup_{s \in [0, s^0]} \| \omega^0 (s) - \alpha^0 (s) \| \Big]
	\\&\leq C \Big[ \Big( 1 + \sup_{s \in [0, s^0]} \|\omega^0 (s)\| \Big) |t^0 - s^0|^{1/2} 	\\&\hspace{3cm}+ \sup_{s \in [s^0 , t^0]} \| \omega^0 (s) - \alpha^0 (s^0) \| + 2\sup_{s \in [0, s^0]} \| \omega^0 (s) - \alpha^0 (s) \| \Big]
	\\&\leq C \Big[ \Big( 1 + \sup_{s \in [0, s^0]} \|\omega^0 (s)\| \Big) |t^0 - s^0|^{1/2} + \sup_{s \in [0, T]} \| \omega^0 (s \wedge t^0) - \alpha^0 (s \wedge s^0) \| \Big]
	\\&\leq C \dd_T ( (t^0, \omega^0), (s^0, \alpha^0)).
	\end{align*}
	The proof is complete.
	\qed


\section{Proof of the nonlinear martingale problem: Theorem \ref{theo: MP}} \label{sec: pf MP}

In this section we work under the assumptions of Theorem \ref{theo: MP}, i.e., we assume that \(d = 1\), that \(F = F_0 \times F_1\) for two compact metrizable spaces  \(F_0\) and \(F_1\), that \(b\) depends on  \(F\) only through the \(F_0\) variable, that \(a\) depends on \(F\) only through the \(F_1\) variable, and that the Conditions \ref{cond: LG} and \ref{cond: joint continuity in all} hold. 

\begin{lemma} \label{lem: real worst case}
	For every \((t, \omega) \in \of 0, \infty\of\), there exists a \(\oP \in \cA (t, \omega)\) such that \((\llambda \otimes \oP)\)-a.e.
	\begin{align*}
	b^{\oP}_{\cdot +t} = \sup \big\{ b (f_0, \cdot + t, \omega \otimes_t X) \colon f_0 \in F_0\big\}, \qquad 
	a^{\oP}_{\cdot + t} = \sup \big\{ a (f_1, \cdot + t, \omega \otimes_t X ) \colon f_1 \in F_1 \big\},
	\end{align*}
	where \((b^{\oP}_{\cdot + t}, a^{\oP}_{\cdot + t})\) denote the Lebesgue densities of the \(\oP\)-characteristics of \(X_{\cdot + t}\) (for its right-continuous natural filtration).
\end{lemma}
\begin{proof}
	Thanks to the Conditions \ref{cond: LG} and \ref{cond: joint continuity in all}, Corollary \ref{coro: real concatenation continuous} and Berge's maximum theorem (\cite[Theorem~17.31]{charalambos2013infinite}), the functions 
	\[(s, \alpha) \mapsto \sup \big\{ b (f_0, s + t, \omega \otimes_t \alpha) \colon f_0 \in F_0\big\},  \quad (s, \alpha) \mapsto \sup \big\{ a (f_1, s + t, \omega \otimes_t \alpha ) \colon f_1 \in F_1 \big\}\] are continuous and of linear growth (in the sense of (4) on p. 258 in \cite{skorokhod}). Hence, the claim follows from \cite[Theorem 4, p. 265]{skorokhod} and Lemma \ref{lem: jacod restatements}.
\end{proof}

\begin{proof}[Proof of Theorem \ref{theo: MP}]
	Fix \(\omega \in \Omega, n \in \mathbb{N}, \phi \in \mathbb{M}_{icx}\) and \(s, h \geq 0\). 
	Take \(P \in \cA(s, \omega)\), denote the Lebesgue densities of the \(P\)-characteristics of \(X_{\cdot + s}\) by \((b^{P}_{\cdot + s}, a^{P}_{\cdot + s})\) and denote the local \(P\)-martingale part of \(X_{\cdot + s}\) by \(M^P_{\cdot + s}\).
	It\^o's formula yields that 
	\[
	\phi (X_{s + h}) - \phi (X_{s}) = \int_0^{h} \phi' (X_{r + s}) d M^P_{r + s} +  \int_s^{s + h} \Big( \phi' (X_r) b^P_r + \frac{1}{2} \phi'' (X_{r}) a^{P}_{r} \Big) dr.
	\]
	Hence, using the definition of \(\rho_n = \inf \{r \geq 0 \colon |X_r| \geq n\}\) and Condition \ref{cond: LG}, the stopped local martingale part is a true martingale and we obtain that 
	\begin{align*}
	E^P \Big[ \phi (X_{(s + h) \wedge \rho_n}) &- \phi (X_{s \wedge \rho_n}) - \int_{s \wedge \rho_n}^{(s + h) \wedge \rho_n} G (r, X, \phi) dr \Big] 
	\\&= E^P \Big[  \int_{s \wedge \rho_n}^{(s + h) \wedge \rho_n} \Big(\Big( \phi' (X_r) b^P_r + \frac{1}{2} \phi'' (X_{r}) a^{P}_{r} \Big) - G (r, X, \phi) \Big) dr \Big].
	\end{align*}
	By definition of \(G\), we clearly have 
	\[ 
	E^P \Big[  \int_{s \wedge \rho_n}^{(s + h) \wedge \rho_n} \Big(\Big( \phi' (X_r) b^P_r + \frac{1}{2} \phi'' (X_{r}) a^{P}_{r} \Big) - G (r, X, \phi) \Big) dr \Big] \leq 0,
	\]
	which implies 
	\[
	\mathcal{E}_s \Big(\phi (X_{(s + h) \wedge \rho_n}) - \phi (X_{s \wedge \rho_n}) - \int_{s \wedge \rho_n}^{(s + h) \wedge \rho_n} G (r, X, \phi) dr\Big) (\omega) \leq 0.
	\]
	Let us now prove the converse inequality. 
	By Lemma \ref{lem: real worst case}, there exists a probability measure \(\oP \in \cA (s, \omega)\) such that \((\llambda \otimes \oP)\)-a.e.
	\begin{align*}
	b^{\oP}_{\cdot +s} = \sup \big\{ b (f_0, \cdot + s, \omega \otimes_s X) \colon f_0 \in F_0\big\}, \qquad 
	a^{\oP}_{\cdot + s} = \sup \big\{ a (f_1, \cdot + s, \omega \otimes_s X ) \colon f_1 \in F_1 \big\}.
	\end{align*}
	Now, as \(\phi', \phi'' \geq 0\), we get
	\begin{align*}
	E^{\oP} \Big[  \int_{s \wedge \rho_n}^{(s + h) \wedge \rho_n} &\Big(\Big( \phi' (X_r) b^{\oP}_r+ \frac{1}{2} \phi'' (X_{r}) a^{\oP}_{r} \Big) - G (r, X, \phi) \Big) dr \Big] 
	\\&= E^{\oP} \Big[  \int_{s \wedge \rho_n}^{(s + h) \wedge \rho_n} \Big(G(r, \omega \otimes_t X, \phi) - G (r, X, \phi) \Big) dr \Big]
	= 0,
	\end{align*}
	which implies 
	\[
	\mathcal{E}_s \Big(\phi (X_{(s + h) \wedge \rho_n}) - \phi (X_{s \wedge \rho_n}) - \int_{s \wedge \rho_n}^{(s + h) \wedge \rho_n} G (r, X, \phi) dr\Big) (\omega) \geq 0.
	\]
	The proof is complete.
\end{proof}

\appendix

\section{Continuity of \(\cR\) in the L\'evy case} \label{subsec: levy}
This appendix is dedicated to the fact that the correspondence \((t, \omega) \mapsto \cR (t, \omega)\) is continuous if the uncertainty set \(\Theta (t, \omega) \equiv \Theta\) is independent of \((t, \omega)\) and convex and compact. This result is covered by our more general Theorems \ref{theo: upper hemi} and \ref{theo: lower hemi}. The purpose of this section is to explain a substantially simpler proof for this L\'evy setting. 
\begin{proposition} 
	Let \( \Theta \subset \mathbb{R} \times \mathbb{R_+} \) be a convex and compact set. Then, the correspondence
	\[
	(t, \omega) \mapsto \cR(t, \omega) = \big\{ P \in \fPas \colon  P \circ X_0^{-1} = \delta_{\omega (t)}, \ (\llambda \otimes P)\text{-a.e. } (dB^{P} /d\llambda, dC^{P}/d\llambda) \in \Theta \big\}
	\]
	is continuous.
\end{proposition}
\begin{proof}
	Define 
	\[
	\mathfrak{P}(\Theta) :=  \big\{ P \in \fPas \colon P \circ X_0^{-1} \in \{\delta_x \colon x \in \mathbb{R}\}, \ (\llambda \otimes P)\text{-a.e. } (dB^{P} /d\llambda, dC^{P}/d\llambda) \in \Theta \big\}.
	\]
	By Theorem \ref{theo: upper hemi}, the set \(\mathfrak{P}(\Theta) \) is closed, cf. also \cite[Proposition 3.8]{CN22b} and \cite[Proposition~4.4]{neufeld}.
	Furthermore, the function
	\[
	\Omega \times \of 0, \infty \of \ \ni (\omega', t, \omega) \mapsto \omega(t) + 
	\omega' - \omega'(0) = \omega(t) \otimes_0 \omega'
	\]
	is continuous, and hence, by \cite[Theorem 8.10.61]{bogachev}, so is the map 
	\[
	\of 0, \infty \of \ \times \ \mathfrak{P}(\Omega) \ni (t, \omega, P) \mapsto P \circ (\omega(t) \otimes_0 X)^{-1}.
	\]
	We claim that 
	\[ \cR(t,\omega) = \big\{ P \circ (\omega(t) \otimes_0 X)^{-1} \colon P \in \mathfrak{P}(\Theta) \big\}.
	\]
	Indeed, Lemmata \ref{lem: filtraation shift} and \ref{lem: jacod restatements} show that the map
	\( P \mapsto P \circ (\omega(t) \otimes_0 X)^{-1} \) leaves \( \mathfrak{P}(\Theta) \) invariant.
	We stress that this part of the argument is special for the L\'evy case. Providing an intuition, it corresponds to the fact that L\'evy processes have an additive structure in their initial values, i.e., if \(L^x\) is a L\'evy process with starting value \(x\), then \(L^x - x\) is a L\'evy process (with the same L\'evy--Khinchine triplet) starting at zero.
	
	Finally, we are in the position to prove that \( (t,\omega) \mapsto \cR(t,\omega) \) is continuous. By virtue of \cite[Theorem~17.23]{charalambos2013infinite}, \( \cR \) is lower hemicontinuous, being the composition of the lower hemicontinuous correspondence \( (t, \omega) \mapsto \varphi (t, \omega) := \{(t,\omega)\} \times \mathfrak{P}(\Theta) \) 
	and the (single-valued) continuous correspondence \( (t, \omega, P ) \mapsto \{ P \circ (\omega(t) \otimes_0 X)^{-1} \} \). Here, lower hemicontinuity of \( \varphi \) follows from \cite[Theorem~17.28]{charalambos2013infinite}.
	We would like to point out that this part of the argument hinges on the fact that \(\Theta\) is constant.
	For the upper hemicontinuity of \( \cR \) we need to argue separately. 
	Let \( F \subset \mathfrak{P}(\Omega) \) be closed. We need to show that the lower inverse
	\( \cR^l(F) = \{ (t,\omega) \colon \cR((t, \omega)) \cap F \neq \emptyset \} \) is closed.
	Assume that \( (t^n,\omega^n)_{n \in \mathbb{N}} \subset \cR^l(F) \) converges to \( (t, \omega )  \in \of 0, \infty \of \).
	For each \(n \in \mathbb{N} \), there exists a probability measure \( P^n \in \cR (t^n, \omega^n) \cap F \).
	As \( \{ \omega^n(t^n) \colon n \in \mathbb{N} \} \subset \mathbb{R} \) is bounded, the set \( \{ P^n \colon n \in \mathbb{N} \} \) is relatively compact by Lemma \ref{lem: rel comp}. Hence, passing to a subsequence if necessary, we can assume that \( P^n \to P \) weakly for some 
	\( P \in \on{cl} (\mathfrak{P}(\Theta) \cap F ) = \mathfrak{P}(\Theta) \cap F \).
	As \( P \circ X_0^{-1} = \lim_{n \to \infty} P^n \circ X_0^{-1} = \lim_{n \to \infty} \delta_{\omega^n(t^n)} = \delta_{\omega(t)} \), we conclude that \(P \in \cR (t, \omega) \cap F\), which implies
	\( (t,\omega) \in \cR^l(F) \).
\end{proof}


\begin{thebibliography}{1}
	\bibitem{charalambos2013infinite}
	C.~D.~Aliprantis and K.~B.~Border. 
	\newblock {\em Infinite Dimensional Analysis: A Hitchhiker's Guide}.
	\newblock Springer Berlin Heidelberg, 3rd ed., 2006.
	
	
	\bibitem{bogachev}
	V.~I.~Bogachev.
	\newblock {\em Measure Theory}.
	\newblock Springer Berlin Heidelberg, 2007.
	
	
	\bibitem{cosso}
	A.~Cosso and F.~Russo.
	\newblock Crandall–Lions viscosity solutions for path-dependent PDEs: The case of heat equation.
	\newblock {\em Bernoulli}, 28(1):481--503, 2022.
	
	\bibitem{cosso2}
	A.~Cosso, F.~Gozzi, M.~Rosestolato and F.~Russo.
	\newblock Path-dependent Hamilton-Jacobi-Bellman equation:
	Uniqueness of Crandall-Lions viscosity solutions.
	\newblock arXiv:2107.05959v2, 2022.
	
	\bibitem{CN22b}
	D.~Criens and L.~Niemann.
	\newblock Markov selections and Feller properties of nonlinear diffusions.
	\newblock arXiv:2205.15200v5, 2023.
	
	
	\bibitem{dellacheriemeyer}
	C.~Dellacherie and P.~A.~Meyer.
	\newblock {\em Probability and Potential}.
	\newblock North-Holland Publishing Company - Amsterdam, New York, Oxford, 1978.
	
	\bibitem{denk2020semigroup}
	R.~Denk, M.~Kupper, and M.~Nendel. 
	\newblock A semigroup approach to nonlinear L{\'e}vy processes. 
	\newblock {\em Stochastic Processes and their Applications}, 130:1616--1642, 2020.
	
	\bibitem{DenMul02}
	M.~Denuit and A.~M\"uller.
	\newblock Smooth generators of integral stochastic orders.
	\newblock  \newblock {\em The Annals of Applied Probability}, 12(4):1174--1184, 2002.
	
	\bibitem{diestel}
	J.~Diestel and J.~J.~Uhl, Jr.
	\newblock {\em Vector Measures}.
	\newblock American Mathematical Society, 1977.
	
	\bibitem{nicole1987compactification}
	N.~El Karoui, D. Nguyen and M. Jeanblanc-Picqu{\'e}.
	\newblock Compactification methods in the control of degenerate diffusions: existence of an optimal control.
	\newblock {\em Stochastics}, 20(3):169--219, 1987.
	
	\bibitem{ElKa15}
	N.~El Karoui and X.~Tan.
	\newblock Capacities, measurable selection and dynamic programming part II: application in stochastic control problems.
	\newblock arXiv:1310.3364v2, 2015.
	
	
	\bibitem{fadina2019affine}
	T.~Fadina, A.~Neufeld, and T.~Schmidt. 
	\newblock Affine processes under parameter uncertainty. 
	\newblock {\em Probability, Uncertainty and Quantitative Risk}, 4(5), 2019.
	
	\bibitem{skorokhod}
	I.~I.~Gikhman and A.~V.~Skorokhod.
	\newblock {\em The Theory of Stochastic Processes III}.
	\newblock Springer Berlin Heidelberg, reprint of the 1974 ed., 2007.
	
	\bibitem{gozzi}
	F.~Gozzi, C.~Marinelli and S.~Savin.
	\newblock On controlled linear diffusions with delay in a model of optimal advertising under uncertainty with memory effects.
	\newblock {\em Journal of Optimization Theory and Applications}, 142:291--321, 2009.
	
	\bibitem{guo2018martingale}
	X.~Guo, C.~Pan and S.~Peng.
	\newblock  Martingale problem under nonlinear expectations.
	\newblock  {\em Mathematics and Financial Economics}, 12:135--164, 2018.
	
	
	\bibitem{HWY}
	S.-W.~He, J.-G.~Wang and J.-A-~Yan.
	\newblock {\em Semimartingale Theory and Stochastic Calculus}.
	\newblock Routledge, 1992.
	
	
	\bibitem{hol16}
	J.~Hollender.
	\newblock { \em L{\'e}vy-Type Processes under Uncertainty and Related Nonlocal Equations. }
	\newblock PhD thesis, TU Dresden, 2016. 
	
	\bibitem{hu2021g}
	M.~Hu and S.~Peng.
	\newblock  G-L{\'e}vy processes under sublinear expectations.
	\newblock  {\em Probability, Uncertainty and Quantitative Risk}, 6(1), 2021.
	
	\bibitem{IW}
	N.~Ikeda and S.~Watanabe.
	\newblock {\em Stochastic Differential Equatrions and Diffusion Processes}.
	\newblock North-Holland Publishing Company Amsterdam Oxford New York, 2nd edition, 1989.
	
	\bibitem{jacod79}
	J.~Jacod.
	\newblock {\em Calcul stochastique et probl\`emes de martingales}.
	\newblock Springer Berlin Heidelberg New York, 1979.
	
	\bibitem{jacod80}
	J.~Jacod. 
	\newblock Weak and strong solutions of stochastic differential equations.
	\newblock {\em Stochastics}, 3:171--191, 1980.
	
	\bibitem{jacod1981weak}
	J.~Jacod and J.~M\'emin.
	\newblock Weak and strong solutions of stochastic differential equations:
	Existence and stability.
	\newblock In D.~Williams, editor, {\em Stochastic Integrals}, volume 851 of
	{\em Lecture Notes in Mathematics}, pages 169--212. Springer Berlin
	Heidelberg, 1981.
	
	\bibitem{JS}
	J.~Jacod and A.~N.~Shiryaev.
	\newblock {\em Limit Theorems for Stochastic Processes}.
	\newblock Springer Berlin Heidelberg, 2nd ed., 2003.
	
	\bibitem{Kallenberg}
	O.~Kallenberg.
	\newblock {\em Foundations of Modern Probability}.
	\newblock Springer New York, 3rd ed., 2021.
	
	\bibitem{KaraShre}
	I.~Karatzas and S.~E.~Shreve.
	\newblock {\em Brownian Motion and Stochastic Calculus}.
	\newblock Springer New York, 2nd edition, 1991.
	
	\bibitem{klenke}
	A.~Klenke.
	\newblock {\em Probability Theory}.
	\newblock Springer London, 2nd ed., 2014.
	
	\bibitem{K19}
	F.~K\"uhn.
	\newblock Viscosity solutions to Hamilton–Jacobi–Bellman equations associated with sublinear
	L{\'e}vy(-type) processes.
	\newblock {\em ALEA}, 16:531--559, 2019.
	
	\bibitem{K21}
	F.~K\"uhn.
	\newblock On infinitesimal generators of sublinear Markov semigroups.
	\newblock {\em Osaka Journal of Mathematics}, 58(3):487--508, 2021.
	
	
	\bibitem{neufeld}
	C. ~Liu and A. ~Neufeld.
	\newblock Compactness criterion for semimartingale laws and semimartingale optimal transport.
	\newblock {\em Transactions of the American Mathematical Society}, 372(1):187--231, 2019.
	
	\bibitem{MP80}
	M.~Metivier and J.~Pellaumail.
	\newblock {\em Stochastic Integration}.
	\newblock Academic Press, 1980.
	
	\bibitem{muller97}
	A.~M\"uller.
	\newblock Stochastic orders generated by integrals: a unified study.
	\newblock  \newblock {\em Advances in Applied Probability}, 29(2):414--428, 1997.
	
	\bibitem{neufeld2014measurability}
	A.~Neufeld and M.~Nutz.
	\newblock Measurability of semimartingale characteristics with respect to the probability law. 
	\newblock {\em Stochastic Processes and their Applications}, 124:3819--3845, 2014.
	
	\bibitem{neufeld2017nonlinear}
	A.~Neufeld and M.~Nutz.
	\newblock Nonlinear L{\'e}vy processes and their characteristics. \newblock {\em Transactions of the American Mathematical Society}, 369:69--95, 2017.
	
	\bibitem{nutz}
	M.~Nutz.
	\newblock Random G-expectations.
	\newblock {\em Annals of Applied Probability}, 23(5):1755–1777, 2013.
	
	
	\bibitem{NVH}
	M.~Nutz and R.~van Handel.
	\newblock Constructing sublinear expectations on path space. 
	\newblock {\em Stochastic Processes and their Applications}, 123(8):3100--3121, 2013.
	
	\bibitem{peng2007g}
	S.~G.~Peng.
	\newblock G-expectation, G-Brownian motion and related stochastic calculus of It{\^o} type.
	\newblock In F.~E.~Benth et. al., editors, {\em Stochastic Analysis and Applications: The Abel Symposium 2005}, pages 541--567, Springer Berlin Heidelberg, 2007.
	
	\bibitem{peng2008multi}
	S.~G.~Peng.
	\newblock  Multi-dimensional G-Brownian motion and related stochastic calculus under G-expectation.
	\newblock {\em Stochastic Processes and their Applications}, 118:2223--2253, 2008.
	
	\bibitem{peng2010}
	S.~G.~Peng.
	\newblock Nonlinear expectations and stochastic calculus under uncertainty.
	\newblock {\em arXiv:1002.4546}, 2010.
	
	\bibitem{pollard}
	D.~Pollard.
	\newblock {\em Convergence of Stochastic Processes}.
	\newblock Springer New York, 1984.
	
	\bibitem{protter}
	P.~E.~Protter.
	\newblock {\em Stochastic Integration and Differential Equations}.
	\newblock Springer Berlin Heidelberg, 2nd ed., 2004.
	
	
	
	\bibitem{stroockanalytic}
	D.~W.~Stroock.
	\newblock {\em Probability Theory: An Analytic View}.
	\newblock Cambridge University Press, 2nd ed., 2011.
	
	\bibitem{SV}
	D.~W.~Stroock and S.R.S.~Varadhan.
	\newblock {\em Multidimensional Diffusion Processes}.
	\newblock Springer Berlin Heidelberg, reprint of 1997 ed., 2006.
	
	\bibitem{tang}
	S.~Tang and F.~Zhang.
	\newblock Path-Dependent Optimal Stochastic Control and
	Viscosity Solution of Associated Bellman Equations.
	\newblock {\em Discrete and Continuous Dynamical Systems}, 35(11):5521-5553, 2015.
	
		\bibitem{zhou}
	J.~Zhou.
	\newblock Viscosity Solutions to Second Order Path-Dependent Hamilton-Jacobi-Bellman Equations and Applications.
	\newblock arXiv:2005.05309v3, 2022.
	
\end{thebibliography}
\end{document}